\documentclass[oneside,11pt,reqno]{amsart}

\usepackage[a4paper,margin=30mm]{geometry}

\usepackage[pagebackref=false,colorlinks=true]{hyperref}

\usepackage{verbatim,upref,amsxtra,amssymb,amscd,graphicx}

\usepackage{mathtools,mathdots}

\usepackage{amsrefs}

\usepackage{epsfig,color}

\usepackage{times}

\usepackage{bm}

\usepackage{varioref}

\usepackage{lineno}

\DeclareMathAlphabet\mathscr{U}{eus}{m}{n}
\SetMathAlphabet\mathscr{bold}{U}{eus}{b}{n}
\DeclareMathAlphabet\matheur{U}{eur}{m}{n}
\SetMathAlphabet\matheur{bold}{U}{eur}{b}{n}

\numberwithin{equation}{section}

\newtheorem{theorem}{Theorem}[section]
\newtheorem{proposition}[theorem]{Proposition}
\newtheorem{lemma}[theorem]{Lemma}
\newtheorem{corollary}[theorem]{Corollary}

\theoremstyle{definition}

\newtheorem{definition}[theorem]{Definition}
\newtheorem{example}[theorem]{Example}

\theoremstyle{remark}

\newtheorem{remark}[theorem]{Remark}

\newtheorem{question}[theorem]{Question}

\newcommand{\htop}{{\textup{h}}_{\textup{top}}}

\newcommand{\supp}{\textup{supp}}

\newcommand{\ddet}{\textup{det}}
\newcommand{\Sym}{\mathscr{S}}

\newcommand{\per}{\textup{per}}
\newcommand{\iper}{\textup{per}^{\mspace{1.5mu}\iota}}

\newcommand{\sgn}{\textup{sgn}}

\begin{document}\allowdisplaybreaks\frenchspacing



\title{Restricted Permutations and Permanents of Infinite Amenable Groups}

\author{Hanfeng Li}
\address{Hanfeng Li:
Department of Mathematics, SUNY at Buffalo,
Buffalo, NY 14260-2900, U.S.A.}
\email{hfli@math.buffalo.edu}

\author{Klaus Schmidt}
\address{K. Schmidt:
Mathematics Institute, University of Vienna, Oskar-Morgenstern-Platz 1, 1090 Vienna, Austria}
\email{klaus.schmidt@univie.ac.at}


\dedicatory{Dedicated to the memory of Anatoly Moiseevich Vershik}

\date{24-12-31}

	\begin{abstract}
Let $\Gamma $ be an infinite discrete group and $\mathsf{A}\subset \Gamma $ a nonempty finite subset. The set of permutations $\sigma $ of $\Gamma $ such that $s^{-1}\sigma (s)\in \mathsf{A}$ for every $s\in \Gamma $ can be identified with a shift of finite type $X_\mathsf{A}\subset \mathsf{A}^{\Gamma}$ over $\Gamma $. In this paper we study dynamical properties of such shift spaces, like invariant probability measures, topological entropy, and topological pressure, under the hypothesis that $\Gamma $ is amenable.

In this case the topological entropy $\htop(X_\mathsf{A})$ can be expressed as logarithmic growth rate of permanents of certain finite (0,1)-matrices associated with right F{\o}lner sequences in $\Gamma $. Motivated by the difficulty of computing such permanents we introduce the notion of the permanent $\per(f)$ for nonnegative elements $f$ in the real group ring $\mathbb{R}\Gamma $ of $\Gamma $ whose support is the alphabet $\mathsf{A}$ of the \emph{SFT} $X_\mathsf{A}$, and compare, for arbitrary $f \in \mathbb{R}\Gamma $, the \emph{Fuglede-Kadison determinant} $\ddet _\textup{FK}(f)$ with the permanent $\per(|f|)$ of the absolute value $|f|$ of $f$. Although this approach is effective in only few examples, discussed below, it is interesting from a conceptual point of view that the permanent $\per(f)$ of a nonnegative element $f\in \mathbb{R}\Gamma $ can be viewed as \emph{topological pressure} of the restricted-permutation \emph{SFT} $X_\mathsf{A}$ associated with the function $\log f$ on the alphabet $\mathsf{A}=\textup{supp}(f)$ of $X_\mathsf{A}$.
	\end{abstract}

\maketitle

\section{Introduction}

Let $\mathsf{A}\subset \mathbb{Z}$ be a finite set, and let $X_\mathsf{A}^{(n)}$ be the collection of all permutations $\sigma $ of $\mathbb{N}_n=\{0,\dots ,n-1\}$ with the property that $\sigma (i)-i\in \mathsf{A}\pmod n$\, for every $i\in \mathbb{N}_n$. What can one say about the number $\bigl|X_\mathsf{A}^{(n)}\bigr|$ of all such permutations as $\mathsf{A}$ remains fixed, but $n\to\infty $? Both for the history and for variations of this problem we refer to \cites{Baltic10, Diaconis99, Edwards15, Kaplansky44, Kaplansky46, Klove09, Lehmer70}, to mention a few references.

As $n\to \infty $, one is led to consider permutations $\sigma $ of $\mathbb{Z}$ which move each $s\in \mathbb{Z}$ by some element of a given finite set $\mathsf{A}\subset \mathbb{Z}$. This embeds questions about the sizes of the sets $X_\mathsf{A}^{(n)},\,n\ge1$, in problems concerning the dynamics of the shift action of $\mathbb{Z}$ on the shift-invariant set of all permutations $\sigma $ of $\mathbb{Z}$ with $\sigma (i)-i\in \mathsf{A}$ for every $i\in \mathbb{Z}$. If one replaces $\mathbb{Z}$ by $\mathbb{Z}^d$, $d\ge1$, or more generally, by a countably infinite discrete amenable group $\Gamma $, the analogous problem becomes considerably more challenging: for every nonempty finite set $\mathsf{A}\subset \Gamma $ (in symbols: $\mathsf{A}\Subset \Gamma $), the set of permutations $\sigma $ of $\Gamma $ satisfying that $s^{-1}\sigma (s)\in \mathsf{A}$ for every $s\in \Gamma $, can be identified with a shift of finite type (\emph{SFT}) $X_\mathsf{A}\subset \mathsf{A}^{\Gamma}$ over $\Gamma$. For $\Gamma =\mathbb{Z}^d$, the study of such restricted-permutation \emph{SFT}s was initiated in the paper \cite{SS16} and subsequently developed further in \cite{Elimelech21}, where the author used a correspondence between certain $\mathbb{Z}^2$-\emph{SFT}s of this form and dimer coverings of specific locally finite planar graphs to compute the topological entropy of $X_\mathsf{A}$ in some examples.

In the general setting of countable discrete groups $\Gamma $, these spaces $X_\mathsf{A}, \,\mathsf{A}\Subset \Gamma $, form a distinguished class of \emph{SFT}s over $\Gamma $ with interesting and only partially understood dynamical properties. We should also mention connections between restricted permutations of general groups (and, more generally, of metric spaces), \emph{restricted orbit equivalence} of group actions (in the sense of \cites{BT98, KR} or \cite{R}), and \emph{wobbling groups} as discussed, for example, in \cites{Juschenko13, Juschenko15, Juschenko22}. Note, however, that the restricted permutations considered here do not form a group (unlike the families of restricted permutations arising from orbit equivalence or in geometric group theory).

\smallskip This paper is organized as follows: Section \ref{S-local entropy} is devoted to local entropy and local pressure for shift-actions of countable amenable groups $\Gamma $, including \emph{SFT}s with finite alphabets over $\Gamma $ (cf. e.g. \cites{HM10, CM21, DFR}) and algebraic $\Gamma $-actions as defined in \cites{DSAO, Deninger, LT}. In Section \ref{S-local permutation} we focus on local entropy and invariant probability measures for the specific $\Gamma $-\emph{SFT}s arising from restricted permutations of countable amenable groups $\Gamma $. In Section \ref{S-permanent} we look in more detail at the topological entropies $\htop(X_\mathsf{A})$ of such restricted-permutation \emph{SFT}s $X_\mathsf{A}$. These can be expressed as logarithmic growth rates of permanents of certain finite (0,1)-matrices associated with right F{\o}lner sequences in $\Gamma $ (Proposition \ref{P-entropy in local} and Remark \ref{R-permanent and entropy}).

Since the computation of permanents of $(n \negthinspace \times \negthinspace n)$-matrices is NP-hard and conjecturally not possible in polynomial time (cf. e.g., \cite{Valiant79}), there is little hope for general numerical results about $\htop(X_\mathsf{A}),\, \mathsf{A} \Subset \Gamma $. One attempt to overcome the difficulty of computing explicitly the permanent of a large scalar matrix $M$ is to search for linear operations (like sign changes) on the entries of $M$, resulting in a matrix $M'$ whose determinant is equal to the permanent of $M$. This leads to the notion of the \emph{permanent} $\per(f)$ for nonnegative elements $f$ in the real group ring $\mathbb{R}\Gamma $ whose support is the alphabet $\mathsf{A}$ of the \emph{SFT} $X_\mathsf{A}$, and to the problem of comparing, for an arbitrary element $f \in \mathbb{R}\Gamma $, the \emph{Fuglede-Kadison determinant} $\ddet _\textup{FK}(f)$ with the permanent $\per(|f|)$ of the absolute value $|f|$ of $f$. As is to be expected from negative results about this problem (\cites{Marcus-Minc61, Gibson71}), this approach can succeed in only few examples, investigated in Section \ref{S-per vs det}. However, from a conceptual point of view it is interesting that the permanent $\per(f)$ of a nonnegative element $f\in \mathbb{R}\Gamma $ can be viewed as \emph{topological pressure} of the restricted-permutation \emph{SFT} $X_\mathsf{A}$ associated with the function $\log f$ on the alphabet $\mathsf{A}=\textup{supp}(f)$ of $X_\mathsf{A}$.

In Section \ref{S-approximation} we discuss finite approximations of the permanents $\per(f)$, $f\in \mathbb{R}_{\ge0}\Gamma $, for residually finite amenable groups. For $\Gamma =\mathbb{Z}$, $\per(f)$ is indeed equal to the upper limit of permanents corresponding to finite quotients of $\mathbb{Z}$ (Corollary \ref{C-approximation lower bound}), but in general $\per(f)$ may be larger than the $\limsup $ of the permanents of $f$ induced on the finite quotients of $\Gamma $ (Proposition \ref{P-approximation lower bound}). When applying this to the function $f \equiv 1$ on $\mathsf{A}$ this translates into the familiar statement that the logarithmic growth rate of the number of periodic points of the \emph{SFT} $X_\mathsf{A}$ is less than or qual to $\htop(X_\mathsf{A})$, but even for $\Gamma = \mathbb{Z}^2$ equality of these two quantities is an open problem.

In the final Section \ref{S-bounds} we apply a few known results on permanents of matrices to obtain estimates like $\htop(X_\mathsf{A})\le \frac{1}{|\mathsf{A}|}\log(|\mathsf{A}|!)$ if $\Gamma $ is finitely generated with polynomial growth and $\mathsf{A}\Subset \Gamma $ (Theorem \ref{T-upper bound}), or $\htop(X_\mathsf{A})\ge \log \frac{|\mathsf{A}|}{e}$ if $\Gamma $ is infinite amenable and $\mathsf{A}\Subset \Gamma $ (Corollary \ref{C-lower bound}).

\medskip

\noindent{\it Acknowledgements.}
This work started while the first author was visiting the Erwin Schr\"{o}dinger International Institute for Mathematics and 
Physics in the spring of 2020. He thanks the institute for providing a nice environment. 

\section{The local pressure formula}
	\label{S-local entropy}

\subsection{The local pressure formula for shift actions}
	\label{SS-local shift}

Let $\Gamma $ be a countable discrete group with identity element $\matheur{e}_\Gamma $, and let $\mathscr{F}(\Gamma )$ be the collection of all nonempty finite subsets of $\Gamma $.

If $\mathsf{B}$ is a compact metrizable space we consider the \emph{full shift} $\mathsf{B}^\Gamma $, furnished with the product topology, and write $\lambda $ and $\rho $ for the left and right shift actions of $\Gamma $ on $\mathsf{B}^\Gamma $, defined by
	\begin{equation}
	\label{eq:shift}
(\lambda ^sx)_t=x_{s^{-1}t}, \qquad (\rho ^sx)_t=x_{ts}
	\end{equation}
for every $s\in \Gamma $ and $x=(x_t)_{t\in \Gamma }\in \mathsf{B}^\Gamma $. For $\varnothing \ne F\subseteq \Gamma $, $s\in \Gamma $ and $x\in \mathsf{B}^F$, we denote by $sx\colon sF\to \mathsf{B}$ the function sending $st$ to $x_t$ for every $t\in F$. If $F'$ is a subset of $F$ we denote by $x|_{F'}= \pi _{F'}(x)\in \mathsf{B}^{F'}$ the restriction of $x\in \mathsf{B}^F$ to $F'$.

\smallskip Let $\mathcal{P}$ be a set consisting of pairs $(K,\mathscr{O})$, where $K\in \mathscr{F}(\Gamma )$ and $\mathscr{O}$ is an open subset of $\mathsf{B}^K$. We treat $\mathcal{P}$ as a set of \emph{forbidden patterns} and write
	\begin{equation}
	\label{eq:admissible}
X_\mathcal{P} = \{x\in \mathsf{B}^\Gamma \colon sx|_K= \pi _K(\lambda ^s x) \notin \mathscr{O}\; \textup{for every}\; s\in \Gamma \enspace \textup{and all}\enspace (K,\mathscr{O})\in \mathcal{P}\}
	\end{equation}
for the corresponding closed, left shift-invariant set of \emph{admissible elements} in $\mathsf{B}^\Gamma $. For given $F\subseteq \Gamma $ we denote by
	\begin{equation}
	\label{eq:locally admissible}
X_{\mathcal{P},F} = \{x\in \mathsf{B}^F\colon s^{-1}x|_K\notin \mathscr{O}\;\textup{for every}\; (K,\mathscr{O})\in \mathcal{P}\;\textup{and}\;s\in \Gamma \;\textup{with}\;sK\subseteq F\}
	\end{equation}
the set of \emph{locally admissible} patterns on $F$.

\smallskip Note that every closed, left shift-invariant subset $X\subseteq \mathsf{B}^\Gamma $ can be written as $X_\mathcal{P}$ for suitable $\mathcal{P}$: indeed, one may take $\mathcal{P} = \{(K,\mathsf{B}^K\smallsetminus \pi _K(X)):K\in \mathscr{F}(\Gamma )\}$.

\smallskip Now assume that $\Gamma $ is amenable. We use right F{\o}lner sets.

\smallskip Let $\mathscr{U}$ be a finite open cover of $\mathsf{B}$. For each $F\in \mathscr{F}(\Gamma )$ we denote by $\mathscr{U}^F$ the finite open cover of $\mathsf{B}^F$ consisting of $\prod _{t\in F}U_t$ for all maps $t\mapsto U_t$ from $F$ to $\mathscr{U}$. For each $Y\subseteq \mathsf{B}^F$, we denote by $N(\mathscr{U}^F, Y)$ the minimal number of elements of $\mathscr{U}^F$ needed to cover $Y$.

\smallskip Denote by $\pi = \pi _{\{\matheur{e}_\Gamma \}}$ the coordinate map $\mathsf{B}^\Gamma \to \mathsf{B}$ sending each $x = (x_s)\in \mathsf{B}^\Gamma $ to $x_{\matheur{e}_\Gamma}$. Then $\pi^{-1}(\mathscr{U})$ is a finite open cover of $\mathsf{B}^\Gamma $. For each $F\in \mathscr{F}(\Gamma)$, denote by $\pi^{-1}(\mathscr{U})_F$ the open cover $\bigvee_{s\in F} s\pi^{-1}(\mathscr{U})$ of $\mathsf{B}^\Gamma $. It is well known that the limit
	\begin{equation}
	\label{eq:limit}
\htop(X_\mathcal{P}, \pi^{-1}(\mathscr{U})) \coloneqq \lim_F\frac{1}{|F|}\log N(\pi^{-1}(\mathscr{U})_F, X_\mathcal{P})
	\end{equation}
exists in the sense that there is some $C\in [-\infty ,+\infty )$ such that, for any neighbourhood $U$ of $C$ in $[-\infty ,+\infty )$, there are $K\in \mathscr{F}(\Gamma)$ and $\delta>0$ such that $\frac{1}{|F|}\log N(\pi^{-1}(\mathscr{U})_F, X_\mathcal{P}) \in U$ for all $F\in \mathscr{F}(\Gamma)$ with $|FK\vartriangle F| < \delta |F|$ (cf. \cite{KL16}*{page 220}). For each $F\in \mathscr{F}(\Gamma)$, it is easy to check that $\pi^{-1}(\mathscr{U})_F=\pi _F^{-1}(\mathscr{U}^F)$. Hence
	\begin{displaymath}
N(\pi^{-1}(\mathscr{U})_F, X_\mathcal{P})=N(\pi _F^{-1}(\mathscr{U}^F), X_\mathcal{P})=N(\mathscr{U}^F, \pi _F(X_\mathcal{P})).
	\end{displaymath}
Thus
	\begin{displaymath}
\htop(X_\mathcal{P}, \pi^{-1}(\mathscr{U}))=\lim_F \frac{1}{|F|}\log N(\mathscr{U}^F, \pi _F(X_\mathcal{P})).
	\end{displaymath}
It is also easy to check that for every finite open cover $\mathscr{V}$ of $\mathsf{B}^\Gamma $, one has $\pi^{-1}(\mathscr{U})_K\succeq \mathscr{V}$ for some finite open cover $\mathscr{U}$ of $\mathsf{B}$ and some $K\in \mathscr{F}(\Gamma)$.
It follows that
	\begin{equation}
	\label{eq:htop}
\htop(X_\mathcal{P})= \sup_{\mathscr{U}}\htop(X_\mathcal{P}, \pi^{-1}(\mathscr{U}))
	\end{equation}
for $\mathscr{U}$ ranging over finite open covers of $\mathsf{B}$.

We denote by $C(\mathsf{B})$ the set of real-valued continuous functions on $\mathsf{B}$. Let $f\in C(\mathsf{B})$, and let $\mathscr{U}$ be a finite open cover of $\mathsf{B}$. For each $F\in \mathscr{F}(\Gamma)$, $Y\subseteq \mathsf{B}^\Gamma$ and $Z\subseteq \mathsf{B}^F$,
we set
	\begin{displaymath}
p_F(f, \pi^{-1}(\mathscr{U}), Y)=\inf\Biggl\{\sum_{V\in \mathscr{V}}\,\sup_{x\in V\cap Y} e^{\sum_{s\in F}f(x_s)}\colon \mathscr{V} \text{ is a finite cover of } \mathsf{B}^\Gamma \text{ refining } \pi^{-1}(\mathscr{U})_F\Biggr\}
	\end{displaymath}
and
	\begin{displaymath}
p_F(f, \mathscr{U}, Z)=\inf\Biggl\{\sum_{V\in \mathscr{V}}\,\sup_{x\in V\cap Z} e^{\sum_{s\in F}f(x_s)}\colon \mathscr{V} \text{ is a finite cover of } \mathsf{B}^F \text{ refining } \mathscr{U}^F\Biggr\}.
	\end{displaymath}
Clearly $p_F(0, \pi^{-1}(\mathscr{U}), Y)=N(\pi^{-1}(\mathscr{U}), Y)$ and $p_F(0, \mathscr{U}, Z)=N(\mathscr{U}^F, Z)$. For any finite cover $\mathscr{V}$ of $\mathsf{B}^\Gamma$ refining $\pi^{-1}(\mathscr{U})_F$, if we set $\mathscr{V}'=\{\pi_F(V\cap Y)\colon V\in \mathscr{V}\}\cup \{U\smallsetminus \pi_F(Y)\colon U\in \mathscr{U}^F\}$, then $\mathscr{V}'$ is a finite cover of $\mathsf{B}^F$ refining $\mathscr{U}^F$, and
	\begin{displaymath}
\sum_{V\in \mathscr{V}}\,\sup_{x\in V\cap Y} e^{\sum_{s\in F}f(x_s)}=\sum_{V'\in \mathscr{V}'}\,\sup_{x\in (V'\times \mathsf{B}^{\Gamma\smallsetminus F})\cap Y} e^{\sum_{s\in F}f(x_s)}.
	\end{displaymath}
In the definition of $p_F(f, \pi^{-1}(\mathscr{U}), Y)$ we may thus require that $\mathscr{V}=\pi_F^{-1}(\mathscr{V}')$ for some finite cover $\mathscr{V}'$ of $\mathsf{B}^F$ refining $\mathscr{U}^F$. Therefore
	\begin{displaymath}
p_F(f, \pi^{-1}(\mathscr{U}), Y)=p_F(f, \mathscr{U}, \pi_F(Y)).
	\end{displaymath}

	\begin{lemma}
	\label{L-local}
Let $f\in C(\mathsf{B})$ and let $\mathscr{U}$ be a finite open cover of $\mathsf{B}$. As $F\in \mathscr{F}(\Gamma )$ becomes more and more right invariant, the limits $\lim_F \frac{1}{|F|}\log p_F(f, \pi^{-1}(\mathscr{U}), X_\mathcal{P})$ and $\lim_F \frac{1}{|F|}\log p_F(f, \mathscr{U}, X_{\mathcal{P}, F})$ exist in the same sense as in \eqref{eq:limit}.
	\end{lemma}

	\begin{proof}
We prove the assertion about $p_F(f, \mathscr{U}, X_{\mathcal{P}, F})$. The proof for $p_F(f, \pi^{-1}(\mathscr{U}), X_\mathcal{P})$ is similar.

If $X_{\mathcal{P}, F_1}=\varnothing $ for some $F_1\in \mathscr{F}(\Gamma)$, then $X_{\mathcal{P}, F}=\varnothing $ for all sufficiently right invariant $F\in \mathscr{F}(\Gamma)$, so that $\lim_F \frac{1}{|F|}\log p_F(f, \mathscr{U}, X_{\mathcal{P}, F})=-\infty$. Thus we may assume that $X_{\mathcal{P}, F}\neq \varnothing $ for all $F\in \mathscr{F}(\Gamma)$.

We consider first the case $f\ge 0$.

For any $F\in \mathscr{F}(\Gamma)$ and $s\in \Gamma $, clearly $sX_{\mathcal{P}, F}=X_{\mathcal{P}, sF}$ and
	\begin{displaymath}
p_F(f, \mathscr{U}, X_{\mathcal{P}, F})=p_{sF}(f, \mathscr{U}, X_{\mathcal{P}, sF}).
	\end{displaymath}

For any $F_1, F_2\in \mathscr{F}(\Gamma)$, note that if $\mathscr{V}_j$ is a finite cover of $\mathsf{B}^{F_j}$ refining $\mathscr{U}^{F_j}$ for $j=1, 2$, then the sets $(V_1\times \mathsf{B}^{F_2\smallsetminus F_1})\cap (\mathsf{B}^{F_1\smallsetminus F_2}\times V_2)$ for $V_1\in \mathscr{V}_1$ and $V_2\in \mathscr{V}_2$ form a finite cover of $\mathsf{B}^{F_1\cup F_2}$ refining $\mathscr{U}^{F_1\cup F_2}$. Hence
	\begin{align*}
p_{F_1\cup F_2}&(f, \mathscr{U}, X_{\mathcal{P}, F_1\cup F_2})
	\\
&\le \sum_{V_1\in \mathscr{V}_1, V_2\in \mathscr{V}_2}\,\sup_{x\in (V_1\times \mathsf{B}^{F_2\smallsetminus F_1})\cap (\mathsf{B}^{F_1\smallsetminus F_2}\times V_2)\cap X_{\mathcal{P}, F_1\cup F_2}}e^{\sum_{s\in F_1\cup F_2}f(x_s)}
	\\
&\le \sum_{V_1\in \mathscr{V}_1, V_2\in \mathscr{V}_2}\,\sup_{x\in (V_1\times \mathsf{B}^{F_2\smallsetminus F_1})\cap (\mathsf{B}^{F_1\smallsetminus F_2}\times V_2)\cap X_{\mathcal{P}, F_1\cup F_2}}e^{\sum_{s\in F_1}f(x_s)}e^{\sum_{s\in F_2}f(x_s)}
	\\
&\le \sum_{V_1\in \mathscr{V}_1, V_2\in \mathscr{V}_2}\Biggl(\sup_{x\in V_1\cap \pi_{F_1}(X_{\mathcal{P}, F_1\cup F_2})}e^{\sum_{s\in F_1}f(x_s)}\Biggr)\Biggl(\sup_{x\in V_2\cap \pi_{F_2}(X_{\mathcal{P}, F_1\cup F_2})}e^{\sum_{s\in F_2}f(x_s)}\Biggr)
	\\
&=\Biggl(\sum_{V_1\in \mathscr{V}_1}\sup_{x\in V_1\cap \pi_{F_1}(X_{\mathcal{P}, F_1\cup F_2})}e^{\sum_{s\in F_1}f(x_s)}\Biggr)\Biggl(\sum_{V_2\in \mathscr{V}_2}\sup_{x\in V_2\cap \pi_{F_2}(X_{\mathcal{P}, F_1\cup F_2})}e^{\sum_{s\in F_2}f(x_s)}\Biggr),
	\end{align*}
where in the second inequality we use that $f\ge 0$. It follows that
	\begin{equation}
	\label{E-local}
	\begin{aligned}
p_{F_1\cup F_2}(f, \mathscr{U}, X_{\mathcal{P}, F_1\cup F_2})\le p_{F_1}(f, \mathscr{U}, \pi_{F_1}(X_{\mathcal{P}, F_1\cup F_2})) \cdot p_{F_2}(f, \mathscr{U}, \pi_{F_2}(X_{\mathcal{P}, F_1\cup F_2})).
	\end{aligned}
	\end{equation}

Now by the Ornstein-Weiss lemma \cite{KL16}*{Theorem 4.38} the limit $\lim_F \frac{1}{|F|}\log p_F(f, \mathscr{U}, X_{\mathcal{P}, F})$ exists.

For general $f\in C(\mathsf{B})$, we can find a constant $C$ such that $f\ge C$. Then
	\begin{displaymath}
\lim_F \frac{1}{|F|}\log p_F(f, \mathscr{U}, X_{\mathcal{P}, F})=C+\lim_F \frac{1}{|F|}\log p_F(f-C, \mathscr{U}, X_{\mathcal{P}, F}).\tag*{\qedsymbol}
	\end{displaymath}
	\renewcommand{\qedsymbol}{}
	\vspace{-\baselineskip}
	\end{proof}

For any $f\in C(\mathsf{B})$ and any finite open cover $\mathscr{U}$ of $\mathsf{B}$, we set
	\begin{displaymath}
P(X_\mathcal{P}, f, \pi^{-1}(\mathscr{U}))=\lim_F \frac{1}{|F|}\log p_F(f, \pi^{-1}(\mathscr{U}), X_\mathcal{P}).
	\end{displaymath}
For every finite open cover $\mathscr{V}$ of $\mathsf{B}^\Gamma $, one has $\pi^{-1}(\mathscr{U})_K\succeq \mathscr{V}$ for some finite open cover $\mathscr{U}$ of $\mathsf{B}$ and some $K\in \mathscr{F}(\Gamma)$. Thus
the {\it topological pressure of the $\Gamma $-action $\lambda $ on $X_\mathcal{P}$ associated to $f$} is
	\begin{displaymath}
P(X_\mathcal{P}, f)=\sup_{\mathscr{U}}P(X_\mathcal{P}, f, \pi^{-1}(\mathscr{U}))
	\end{displaymath}
for $\mathscr{U}$ ranging over finite open covers of $\mathsf{B}$ (cf. \cite{MO}*{Section 5.2}).

	\begin{lemma}
	\label{L-neighborhood}
Let $f\in C(\mathsf{B})$ and let $\mathscr{U}$ be a finite open cover of $\mathsf{B}$.
Let $F\in \mathscr{F}(\Gamma)$, $X\subseteq \mathsf{B}^F$, and $\eta>0$. Then there
is an open subset $W$ of $\mathsf{B}^F$ such that $X\subseteq W$ and, for any $Z\subseteq W$, one has
	\begin{displaymath}
p_F(f, \mathscr{U}, Z)\le e^{\eta} p_F(f, \mathscr{U}, X).
	\end{displaymath}
	\end{lemma}
	\begin{proof}
If $X$ is empty we may take $W=\varnothing $. Thus we may assume that $X$ is nonempty. Then $p_F(f, \mathscr{U}, X)>0$.
Let $\rho$ be a compatible metric on $\mathsf{B}^F$.

Take a finite cover $\mathscr{V}$ of $\mathsf{B}^F$ refining $\mathscr{U}^F$ such that
	\begin{displaymath}
\sum_{V\in \mathscr{V}}\sup_{x\in V\cap X}e^{\sum_{s\in F}f(x_s)}< e^{\eta/2} p_F(f, \mathscr{U}, X).
	\end{displaymath}
Enlarging each member of $\mathscr{V}$ slightly, we may assume that $\mathscr{V}$ is an open cover of $\mathsf{B}^F$. For each $V\in \mathscr{V}$, take an open subset $V'$ of $\mathsf{B}^F$ such that the closure of $V'$ is contained in $V$ and that $\mathscr{V}'=\{V'\colon V\in \mathscr{V}\}$ covers $\mathsf{B}^F$.
Take a $\delta>0$ such that for any $x, z\in \mathsf{B}^F$ with $\rho(x, z)<\delta$ one has
	\begin{displaymath}
\Biggl|\sum_{s\in F}f(x_s)-\sum_{s\in F}f(z_s)\Biggr|< \eta/2.
	\end{displaymath}
Shrinking $\delta$ if necessary, we may assume that for any $V\in \mathscr{V}$, any $z\in V'$ and any $x\in \mathsf{B}^F$ with $\rho(x, z)<\delta $, one has $x\in V$. Denote by $W$ the set of $z\in \mathsf{B}^F$ satisfying that $\sup_{x\in X}\rho(x, z)<\delta $. Then $W$ is an open subset of $\mathsf{B}^F$ and $X\subseteq W$.

Let $Z\subseteq W$ and $V\in \mathscr{V}$. For each $z\in V'\cap Z$, since $z\in W$, we can find some $x\in X$ with $\rho(x, z)<\delta$. By our choice of $\delta$ we have $x\in V$ and
	\begin{displaymath}
\sum_{s\in F}f(z_s)<\eta/2+\sum_{s\in F}f(x_s).
	\end{displaymath}
Thus
	\begin{displaymath}
\sup_{z\in V'\cap Z}e^{\sum_{s\in F}f(z_s)}\le e^{\eta/2}\sup_{x\in V\cap X}e^{\sum_{s\in F}f(x_s)}.
	\end{displaymath}
Therefore
	\begin{displaymath}
p_F(f, \mathscr{U}, Z)\le \sum_{V'\in \mathscr{V}'}\,\sup_{z\in V'\cap Z}e^{\sum_{s\in F}f(z_s)}\le e^{\eta/2}\sum_{V\in \mathscr{V}}\,\sup_{x\in V\cap X}e^{\sum_{s\in F}f(x_s)}\le e^{\eta} p_F(f, \mathscr{U}, X). \tag*{\qedsymbol}
	 \end{displaymath}
	 \renewcommand{\qedsymbol}{}
	 \vspace{-\baselineskip}
	\end{proof}

The following is the Ornstein-Weiss quasitiling theorem \cite{KL16}*{Theorem 4.36}.

	\begin{lemma}
	\label{L-OW}
Let $K\in \mathscr{F}(\Gamma)$ and $0<\delta, \delta_1, \eta<1$.
There are $K', F_1, F_2, \dots, F_N\in \mathscr{F}(\Gamma)$ and $\delta'>0$ such that $|F_j K|<(1+\delta)|F_j|$ for each $j=1, \dots, N$, and for any $F\in \mathscr{F}(\Gamma)$ satisfying $|FK'|<(1+\delta')|F|$ there are finite sets $C_1, \dots, C_N\subseteq \Gamma $ satisfying the following conditions:
	\begin{enumerate}
	\item
There is a set $\widehat{cF_j}\subseteq cF_j$ satisfying $|\widehat{cF_j}|\ge (1-\eta)|cF_j|$ for every $1\le j\le N$ and $c\in C_j$ such that the sets $\widehat{cF_j}$ for $1\le j\le N$ and $c\in C_j$ are pairwise disjoint;
	\item
$\bigcup_{1\le j\le N}\bigcup_{c\in C_j}cF_j\subseteq F$ and $\bigl|\bigcup_{1\le j\le N}\bigcup_{c\in C_j}cF_j\bigr|\ge (1-\delta_1)|F|$.
	\end{enumerate}
	\end{lemma}

The next lemma establishes the upper semicontinuity for topological pressure with respect to a fixed finite open cover.

	\begin{lemma}
	\label{L-upper continuity}
Let $f\in C(\mathsf{B})$ and let $\mathscr{U}$ be a finite open cover of $\mathsf{B}$.
Let $X$ be a closed $\Gamma $-invariant subset of $\mathsf{B}^\Gamma$, and let $\varepsilon>0$. Then there is an open subset $W$ of $\mathsf{B}^\Gamma$ such that $X\subseteq W$ and for any closed $\Gamma $-invariant subset $Y$ of $\mathscr{B}^\Gamma$ with $Y\subseteq W$, one has
	\begin{displaymath}
P(Y, f, \pi^{-1}(\mathscr{U}))\le P(X, f, \pi^{-1}(\mathscr{U}))+\varepsilon.
	\end{displaymath}
	\end{lemma}
	\begin{proof}
Since $ P(X, f+C, \pi^{-1}(\mathscr{U}))= C+P(X, f, \pi^{-1}(\mathscr{U}))$ for every constant $C$, replacing $f$ by $f+C$ for a suitable $C$, we may assume that $f\ge 0$. Arguing as in \eqref{E-local}, we have
	\begin{displaymath}
p_{F_1\cup F_2}(f, \mathscr{U}, \pi_{F_1\cup F_2}(X))\le p_{F_1}(f, \mathscr{U}, \pi_{F_1}(X))\cdot p_{F_1}(f, \mathscr{U}, \pi_{F_2}(X))
	\end{displaymath}
for all $F_1, F_2\in \mathscr{F}(\Gamma)$. Set $\|f\|=\max_{y\in \mathsf{B}}|f(y)|$.

Take an $\eta>0$ such that
	\begin{displaymath}
\eta+\frac{P(X, f, \pi^{-1}(\mathscr{U}))+2\eta}{1-\eta}\le P(X, f, \pi^{-1}(\mathscr{U}))+\varepsilon.
	\end{displaymath}
Take $0<\delta_1<1$ such that $(e^{\|f\|}|\mathscr{U}|)^{\delta_1}\le e^\eta$. Take $0<\delta<1$ and $K\in \mathscr{F}(\Gamma)$ containing $\matheur{e}_\Gamma $ such that
	\begin{displaymath}
\frac{1}{|F|}\log p_F(f, \mathscr{U}, \pi_F(X))=\frac{1}{|F|}\log p_F(f, \pi^{-1}(\mathscr{U}), X)\le P(X, f, \pi^{-1}(\mathscr{U}))+\eta
	\end{displaymath}
for every $F\in \mathscr{F}(\Gamma)$ satisfying $|FK|\le (1+\delta)|F|$. Then we have $K', F_1, F_2, \dots, F_N\in \mathscr{F}(\Gamma)$ and $\delta'>0$ given by Lemma~\ref{L-OW}.

For each $1\le j\le N$, by Lemma~\ref{L-neighborhood} we can find an open subset $W_j$ of $\mathsf{B}^{F_j}$ such that $\pi_{F_j}(X)\subseteq W_j$ and for any $Z\subseteq W_j$ one has
	\begin{displaymath}
p_{F_j}(f, \mathscr{U}, Z)\le e^{|F_j|\eta} p_{F_j}(f, \mathscr{U}, \pi_{F_j}(X)).
	\end{displaymath}
Set $W=\bigcap_{1\le j\le N} (W_j\times \mathsf{B}^{\Gamma\smallsetminus F_j})$. Then $W$ is an open subset of $\mathsf{B}^\Gamma$ and $X\subseteq W$.

Let $Y$ be a closed $\Gamma $-invariant subset of $\mathsf{B}^\Gamma$ such that $Y\subseteq W$.

For each $1\le j\le N$, we have $\pi_{F_j}(Y)\subseteq W_j$, whence
	\begin{displaymath}
p_{F_j}(f, \mathscr{U}, \pi_{F_j}(Y))\le e^{|F_j|\eta} p_{F_j}(f, \mathscr{U}, \pi_{F_j}(X)).
	\end{displaymath}

Let $F\in \mathscr{F}(\Gamma)$ such that $|FK'|<(1+\delta')|F|$. Then we have $C_1, \dots, C_N$ given by Lemma~\ref{L-OW}. Set $L=F\smallsetminus \bigcup_{1\le i\le N}\bigcup_{c\in C_i}cF_i$.
Note that
	\begin{displaymath}
|L|\le \delta_1|F|
	\end{displaymath}
and
	\begin{displaymath}
\sum_{1\le j\le N}\sum_{c\in C_j}|F_j|=\sum_{1\le j\le N}\sum_{c\in C_j}|cF_j|\le \sum_{1\le j\le N}\sum_{c\in C_j}\frac{1}{1-\eta}|\widehat{c F_j}|\le \frac{1}{1-\eta}|F|.
	\end{displaymath}
For each $1\le j\le N$ and $c\in C_j$, we have
	\begin{displaymath}
p_{cF_j}(f, \mathscr{U}, \pi_{cF_j}(Y))=p_{F_j}(f, \mathscr{U}, \pi_{F_j}(Y))\le e^{|F_j|\eta} p_{F_j}(f, \mathscr{U}, \pi_{F_j}(X)).
	\end{displaymath}
Now we have
	\begin{align*}
p_F(f, \mathscr{U}, \pi_F(Y))&\le p_L(f, \mathscr{U}, \pi_L(Y))\cdot \prod _{1\le j\le N}\prod _{c\in C_j}p_{cF_j}(f, \mathscr{U}, \pi_{cF_j}(Y))
	\\
&\le (e^{\|f\|}|\mathscr{U}|)^{|L|}\cdot \prod _{1\le j\le N}\prod _{c\in C_j}\bigl(e^{|F_j|\eta} p_{F_j}(f, \mathscr{U}, \pi_{F_j}(X))\bigr)
	\\
&\le (e^{\|f\|}|\mathscr{U}|)^{\delta_1|F|}\cdot e^{\sum_{1\le j\le N}\sum_{c\in C_j}(P(X, f, \pi^{-1}(\mathscr{U}))+2\eta)|F_j|}
	\\
&\le e^{\eta|F|}\cdot e^{\frac{P(X, f, \pi^{-1}(\mathscr{U}))+2\eta}{1-\eta}|F|}.
	\end{align*}
Thus
	\begin{align*}
P(Y, f, \pi^{-1}(\mathscr{U}))&=\lim_F \frac{1}{|F|}\log p_F(f, \mathscr{U}, \pi_F(Y))
	\\
&\le \eta+\frac{P(X, f, \pi^{-1}(\mathscr{U}))+2\eta}{1-\eta}\le P(X, f, \pi^{-1}(\mathscr{U}))+\varepsilon.\tag*{\qedsymbol}
	\end{align*}
\renewcommand{\qedsymbol}{}
\vspace{-\baselineskip}
	\end{proof}

	\begin{lemma}
	\label{L-local extension}
Let $f\in C(\mathsf{B})$ and let $\mathscr{U}$ be a finite open cover of $\mathsf{B}$. Let $F\in \mathscr{F}(\Gamma)$ and $\eta>0$. Then there is an $F'\in \mathscr{F}(\Gamma)$ such that $F\subseteq F'$ and
	\begin{displaymath}
p_F(f, \mathscr{U}, \pi _F(X_{\mathcal{P}, F'}))\le e^{\eta} p_F(f, \mathscr{U}, \pi _F(X_\mathcal{P})).
	\end{displaymath}
	\end{lemma}

	\begin{proof}
By Lemma~\ref{L-neighborhood} we can find an open subset $W$ of $\mathsf{B}^F$ such that $\pi_F(X_\mathcal{P})\subseteq W$ and for any $Z\subseteq W$ one has
	\begin{displaymath}
p_F(f, \mathscr{U}, Z)\le e^{\eta} p_F(f, \mathscr{U}, \pi _F(X_\mathcal{P})).
	\end{displaymath}

Take a sequence $(F_n)_{n\in \mathbb{N}}$ in $\mathscr{F}(\Gamma)$ such that $F\subseteq F_n\subseteq F_{n+1}$ for every $n\in \mathbb{N}$ and $\Gamma=\bigcup_{n\in \mathbb{N}}F_n$. Then
	\begin{displaymath}
\pi _F(X_\mathcal{P})\subseteq \pi _F(X_{\mathcal{P}, F_{n+1}})\subseteq \pi _F(X_{\mathcal{P}, F_n})
	\end{displaymath}
for every $n\in \mathbb{N}$.

We claim that there is some $n\in \mathbb{N}$ such that $\pi _F(X_{\mathcal{P}, F_n})\subseteq W$. We argue by contradiction. So we assume that
$\pi _F(X_{\mathcal{P}, F_n})\nsubseteq W$ for every $n\in \mathbb{N}$. Then for each $n$ we can find a $y_n\in X_{\mathcal{P}, F_n}$ such that $y_n|_F\not\in W$. We take a $y_n'\in \mathsf{B}^\Gamma $ extending $y_n$. Since $\mathsf{B}^\Gamma $ is compact, we can find a subsequence $(y_{n_k}')_{k\in \mathbb{N}}$ of $(y_n')_{n\in \mathbb{N}}$ converging to some $y\in \mathsf{B}^\Gamma $.

We claim that $y\in X_\mathcal{P}$. Let $(K, \mathscr{O})\in \mathcal{P}$ and $s\in \Gamma $. For all large enough $k$, we have $sK\subseteq F_{n_k}$, whence $s^{-1}y_{n_k}'|_K=s^{-1}y_{n_k}|_K\not\in \mathscr{O}$. Since $\mathscr{O}$ is open and $s^{-1}y_{n_k}'|_K$ converges to $s^{-1}y|_K$ as $k\to \infty$, we get that $s^{-1}y|_K\not\in \mathscr{O}$. Thus $y\in X_\mathcal{P}$, as desired.

Note that $y_{n_k}'|_F=y_{n_k}|_F\not\in W$ for every $k$. Since $W$ is open and $y_{n_k}'|_F$ converges to $y|_F$ as $k\to \infty$, we get $y|_F\not\in W$. This contradicts the fact that $\pi _F(X_\mathcal{P})\subseteq W$. Thus we do find some $n\in \mathbb{N}$ such that $\pi _F(X_{\mathcal{P}, F_n})\subseteq W$. Then
	\begin{displaymath}
p_{F_n}(f, \mathscr{U}, \pi _F(X_{\mathcal{P}, F_n}))\le e^{\eta} p_F(f, \mathscr{U}, \pi _F(X_\mathcal{P})). \tag*{\qedsymbol}
	\end{displaymath}
	\renewcommand{\qedsymbol}{}
	\vspace{-\baselineskip}
	\end{proof}

	\begin{theorem}
	\label{T-local}
Let $f\in C(\mathsf{B})$. We have
	\begin{displaymath}
P(X_\mathcal{P}, f, \pi ^{-1}(\mathscr{U}))=\lim_F \frac{1}{|F|}\log p_F(f, \mathscr{U}, X_{\mathcal{P}, F})=\inf_{\mathcal{P}'}\lim_F \frac{1}{|F|}\log p_F(f, \mathscr{U}, X_{\mathcal{P}', F})
	\end{displaymath}
for every finite open cover $\mathscr{U}$ of $\mathsf{B}$, where $\mathcal{P}'$ ranges over finite subsets of $\mathcal{P}$. In particular,
	\begin{displaymath}
P(X_\mathcal{P}, f)=\sup_{\mathscr{U}}\lim_F \frac{1}{|F|}\log p_F(f, \mathscr{U}, X_{\mathcal{P}, F})=\sup_{\mathscr{U}}\inf_{\mathcal{P}'}\lim_F \frac{1}{|F|}\log p_F(f, \mathscr{U}, X_{\mathcal{P}', F})
	\end{displaymath}
for $\mathscr{U}$ ranging over finite open covers of $\mathsf{B}$ and $\mathcal{P}'$ ranging over finite subsets of $\mathcal{P}$.
	\end{theorem}

	\begin{proof}
Replacing $f$ by $f+C$ for a suitable constant $C$, we may assume that $f\ge 0$. Set $\|f\|=\max_{y\in \mathsf{B}}|f(y)|$.

Let $\mathscr{U}$ be a finite open cover of $\mathsf{B}$. We first show that
	\begin{equation}
	\label{E-local10}
P(X_\mathcal{P}, f, \pi^{-1}(\mathscr{U}))=\lim_F \frac{1}{|F|}\log p_F(f, \mathscr{U}, X_{\mathcal{P}, F}).
	\end{equation}

For each $F\in \mathscr{F}(\Gamma)$, we have $X_{\mathcal{P}, F}\supseteq \pi_F(X_\mathcal{P})$, so that
	\begin{displaymath}
p_F(f, \mathscr{U}, X_{\mathcal{P}, F})\ge p_F(f, \mathscr{U}, \pi_F(X_\mathcal{P})).
	\end{displaymath}
Thus
	\begin{displaymath}
P(X_\mathcal{P}, f, \pi^{-1}(\mathscr{U}))=\lim_F \frac{1}{|F|}\log p_F(f, \mathscr{U}, \pi_F(X_\mathcal{P}))\le \lim_F \frac{1}{|F|}\log p_F(f, \mathscr{U}, X_{\mathcal{P}, F}).
	\end{displaymath}

Set $D=\lim_F \frac{1}{|F|}\log p_F(f, \mathscr{U}, X_{\mathcal{P}, F})$. We are left to show that
	\begin{displaymath}
D\le \lim_F \frac{1}{|F|}\log p_F(f, \mathscr{U}, \pi_F(X_\mathcal{P})).
	\end{displaymath}
We may assume that $D\neq -\infty$. Then from Lemma~\ref{L-local extension} we know that $X_\mathcal{P}\neq \varnothing $.

It suffices to show that for any $0<\eta<1$, $K\in \mathscr{F}(\Gamma)$ and $\delta>0$ there is some $F\in \mathscr{F}(\Gamma)$ with $|FK|<(1+\delta)|F|$ such that
	\begin{displaymath}
(1-\eta)(D-2\eta)-\eta \le \frac{1}{|F|}\log p_F(f, \mathscr{U}, \pi_F(X_\mathcal{P})).
	\end{displaymath}

Let $0<\eta<1$, $K\in \mathscr{F}(\Gamma)$ and $\delta>0$. We may assume that $D-2\eta>0$. Take $0<\delta_1<1$ such that $(e^{\|f\|}|\mathscr{U}|)^{2\delta_1}\le e^\eta$. Take $0<\delta_2<1$ and $K_2\in \mathscr{F}(\Gamma)$ containing $\matheur{e}_\Gamma $ such that
	\begin{displaymath}
\frac{1}{|F|}\log p_F(f, \mathscr{U}, X_{\mathcal{P}, F})\ge D-\eta
	\end{displaymath}
for every $F\in \mathscr{F}(\Gamma)$ satisfying $|FK_2|\le (1+\delta_2)|F|$. Then we have $K', F_1, F_2, \dots, F_N\in \mathscr{F}(\Gamma)$ and $\delta'>0$ given by Lemma~\ref{L-OW}. Set
	\begin{displaymath}
D'=\max_{1\le j\le N}\frac{1}{|F_j|}\log p_{F_j}(f, \mathscr{U}, \pi_{F_j}(X_\mathcal{P})).
	\end{displaymath}
Then it suffices to show that
	\begin{displaymath}
(1-\eta)(D-2\eta)-\eta\le D'.
	\end{displaymath}

For each $1\le j\le N$, by Lemma~\ref{L-local extension} we can find an $E_j\in \mathscr{F}(\Gamma)$ containing $\matheur{e}_\Gamma $ such that
	\begin{displaymath}
p_{F_j}(f, \mathscr{U}, \pi _{F_j}(X_{\mathcal{P}, F_jE_j}))\le e^{|F_j|\eta}p_{F_j}(f, \mathscr{U}, \pi _{F_j}(X_\mathcal{P})).
	\end{displaymath}
Set $E=\bigcup_{1\le j\le N}E_j$.

Take an $F\in \mathscr{F}(\Gamma)$ such that $|FK'|<(1+\delta')|F|$ and $|FEK_2|\le (1+\min(\delta_1, \delta_2))|F|$. Then we have $C_1, \dots, C_N$ given by Lemma~\ref{L-OW}. Set $L=FE\smallsetminus \bigcup_{1\le j\le N}\bigcup_{c\in C_j}cF_j$.
Note that
	\begin{displaymath}
|L|=|FE\smallsetminus F|+\Biggl|F\smallsetminus \bigcup_{1\le j\le N}\bigcup_{c\in C_j}cF_j\Biggr|\le |FEK_2\smallsetminus F|+\delta_1|F|\le 2\delta_1|F|,
	\end{displaymath}
and
	\begin{align*}
\sum_{1\le j\le N}\sum_{c\in C_j}|F_j|=\sum_{1\le j\le N}\sum_{c\in C_j}|cF_j|\le \sum_{1\le j\le N}\sum_{c\in C_j}\frac{1}{1-\eta}|\widehat{c F_j}|\le \frac{1}{1-\eta}|F|.
	\end{align*}
For each $1\le j\le N$ and $c\in C_j$, we have
	\begin{align*}
p_{cF_j}(f, \mathscr{U}, \pi _{c F_j}(X_{\mathcal{P}, FE}))&\le p_{cF_j}(f, \mathscr{U}, \pi _{c F_j}(X_{\mathcal{P}, cF_jE_j}))
	\\
&=p_{F_j}(f, \mathscr{U}, \pi _{F_j}(X_{\mathcal{P}, F_jE_j}))=e^{|F_j|\eta}p_{F_j}(f, \mathscr{U}, \pi _{F_j}(X_\mathcal{P}))\le e^{(D'+\eta)|F_j|}.
	\end{align*}
Now we have
	\begin{align*}
p_{FE}(f, \mathscr{U}, X_{\mathcal{P}, FE})&\overset{\eqref{E-local}}\le p_L(f, \mathscr{U}, \pi _L(X_{\mathcal{P}, FE}))\cdot \prod _{1\le j\le N}\prod _{c\in C_j}p_{cF_j}(f, \mathscr{U}, \pi _{c F_j}(X_{\mathcal{P}, FE}))
	\\
&\le (e^{\|f\|}|\mathscr{U}|)^{|L|}\cdot \prod _{1\le j\le N}\prod _{c\in C_j}e^{(D'+\eta)|F_j|}
	\\
&\le (e^{\|f\|}|\mathscr{U}|)^{2\delta_1|F|}\cdot e^{\sum_{1\le j\le N}\sum_{c\in C_j}(D'+\eta)|F_j|}\le e^{\eta|F|}\cdot e^{\frac{D'+\eta}{1-\eta}|F|}.
	\end{align*}
Since $|FEK_2|\le (1+\delta_2)|F|\le (1+\delta_2)|FE|$, we have
	\begin{displaymath}
D-\eta\le \frac{1}{|FE|}\log p_{FE}(f, \mathscr{U}, X_{\mathcal{P}, FE}).
	\end{displaymath}
Therefore
	\begin{align*}
e^{(D-\eta)|F|}\le e^{(D-\eta)|FE|}\le p_{FE}(f, \mathscr{U}, X_{\mathcal{P}, FE})\le e^{\eta|F|}\cdot e^{\frac{D'+\eta}{1-\eta}|F|}.
	\end{align*}
It follows that
	\begin{displaymath}
D-\eta\le \eta+\frac{D'+\eta}{1-\eta},
	\end{displaymath}
whence
	\begin{displaymath}
(1-\eta)(D-2\eta)-\eta\le D'
	\end{displaymath}
as desired. This proves \eqref{E-local10}.

Note that
	\begin{displaymath}
X_\mathcal{P}=\bigcap_{\mathcal{P}'}X_{\mathcal{P}'}
	\end{displaymath}
for $\mathcal{P}'$ ranging over finite subsets of $\mathcal{P}$.
Thus by Lemma~\ref{L-upper continuity} we have
	\begin{displaymath}
P(X_\mathcal{P}, f, \pi ^{-1}(\mathscr{U}))= \inf_{\mathcal{P}'}P(X_{\mathcal{P}'}, f, \pi^{-1}(\mathscr{U}))=\inf_{\mathcal{P}'}\lim_F \frac{1}{|F|}\log p_F(f, \mathscr{U}, X_{\mathcal{P}', F})
	\end{displaymath}
for $\mathcal{P}'$ ranging over finite subsets of $\mathcal{P}$.
	\end{proof}

	\begin{remark}
	\label{R-exchange order}
From Theorem~\ref{T-local} we have
	\begin{displaymath}
\htop(X_\mathcal{P})\le \inf_{\mathcal{P}'}\sup_{\mathscr{U}}\lim_F \frac{1}{|F|}\log N(\mathscr{U}^F, X_{\mathcal{P}', F})=\inf_{\mathcal{P}'}\htop(X_{\mathcal{P}'})
	\end{displaymath}
for $\mathscr{U}$ ranging over finite open covers of $\mathsf{B}$ and $\mathcal{P}'$ ranging over finite subsets of $\mathcal{P}$. In general one might have
	\begin{displaymath}
\htop(X_\mathcal{P})<\inf_{\mathcal{P}'}\htop(X_{\mathcal{P}'})
	\end{displaymath}
for $\mathcal{P}'$ ranging over finite subsets of $\mathcal{P}$. For example, let $\mathsf{B}=[0, 1]$ and let $\mathcal{P}$ be the set of $(\{\matheur{e}_\Gamma\}, (r, 1])$ for rational $0<r<1$. Then $X_\mathcal{P}$ is the singleton of the function $\Gamma\rightarrow \{0\}$ and $\htop(X_\mathcal{P})=0$. For each finite subset $\mathcal{P}'$ of $\mathcal{P}$, $X_{\mathcal{P}'}=[0, t]^\Gamma $ for some $t\in (0, 1]$, whence $\htop(X_{\mathcal{P}'})=\infty$.
	\end{remark}

Now let $\vartheta $ be a compatible metric on $\mathsf{B}$. For each $F\in \mathscr{F}(\Gamma)$, we define a compatible metric $\vartheta_F$ on
$\mathsf{B}^F$ by
	\begin{displaymath}
\vartheta _F(x, y)=\max_{s\in F}\vartheta (x_s, y_s).
	\end{displaymath}
For $Y\subseteq \mathsf{B}^F$ and $\varepsilon>0$, denote by $N_\varepsilon(Y, \vartheta_F)$ the maximal cardinality of $(\vartheta_F, \varepsilon)$-separated subsets of $Y$, where $Z\subseteq Y$ is $(\vartheta_F, \varepsilon)$-separated if $\vartheta_F(z, z')>\varepsilon$ for all distinct $z, z'\in Z$. Similar to the proof that topological entropy defined using open covers is the same as using separated sets, one can check easily that
	\begin{displaymath}
\sup_\mathscr{U} \lim_F \frac{1}{|F|}\log N(\mathscr{U}^F, X_{\mathcal{P}, F})=\sup_{\varepsilon>0}\limsup_F \frac{1}{|F|}\log N_\varepsilon(X_{\mathcal{P}, F}, \vartheta_F)
	\end{displaymath}
for $\mathscr{U}$ ranging over finite open covers of $\mathsf{B}$, and
	\begin{displaymath}
\sup_\mathscr{U} \inf_{\mathcal{P}'}\lim_F \frac{1}{|F|}\log N(\mathscr{U}^F, X_{\mathcal{P}', F})=\sup_{\varepsilon>0}\inf_{\mathcal{P}'}\limsup_F \frac{1}{|F|}\log N_\varepsilon(X_{\mathcal{P}', F}, \vartheta_F)
	\end{displaymath}
for $\mathscr{U}$ ranging over finite open covers of $\mathsf{B}$ and $\mathcal{P}'$ ranging over finite subsets of $\mathcal{P}$. Thus for $f=0$ the second assertion of Theorem~\ref{T-local} can also be stated as the following

	\begin{corollary}
	\label{C-local}
We have
	\begin{displaymath}
\htop(X_\mathcal{P})=\sup_{\varepsilon>0}\limsup_F \frac{1}{|F|}\log N_\varepsilon(X_{\mathcal{P}, F}, \vartheta _F)= \sup_{\varepsilon>0}\inf_{\mathcal{P}'}\limsup_F \frac{1}{|F|}\log N_\varepsilon (X_{\mathcal{P}', F}, \vartheta_F)
	\end{displaymath}
for $\mathcal{P}'$ ranging over finite subsets of $\mathcal{P}$.
	\end{corollary}

\subsection{The local entropy formula for algebraic actions}
	\label{SS-local algebraic}

We denote by $\mathbb{Z}\Gamma $ the \emph{integer group ring} of $\Gamma $, consisting of all finitely supported functions $f\colon \Gamma \to \mathbb{Z}$, written as $f=\sum_{s\in \Gamma }f_ss\in \mathbb{Z}\Gamma $ with $f_s\in \mathbb{Z}$ for all $s$, and with addition, multiplication and involution defined by
	\begin{displaymath}
f+g=\sum_{s\in \Gamma }(f_s+g_s)s,\quad fg = \sum_{s,t\in \Gamma } f_sg_tst, \quad \textup{and}\quad f^* = \sum_{s\in \Gamma }f_ss^{-1},
	\end{displaymath}
for all $f,g$ in $\mathbb{Z}\Gamma $. The set $\supp(f)=\{s\in \Gamma \colon f_s\ne0\}$ is called the \emph{support} of $f \in \mathbb{Z}\Gamma $.

We set $X=\mathbb{T}^\Gamma $ and extend the left and right shift actions $\lambda $ and $\rho $ of $\Gamma $ on $\mathbb{T}^\Gamma $ in \eqref{eq:shift} to $\mathbb{Z}\Gamma $-actions by setting
	\begin{equation}
	\label{eq:rho}
\lambda ^fx= \sum_{t\in \Gamma } f_t \lambda ^tx ,\quad
\rho ^fx= \sum_{t\in \Gamma } f_t \rho ^tx,
	\end{equation}
for every $f=\sum_{s\in \Gamma }f_ss\in \mathbb{Z}\Gamma $ and $x\in \mathbb{T}^\Gamma $.

\smallskip The additive group $\mathbb{Z}\Gamma $ can be identified with the Pontryagin dual of $\mathbb{T}^\Gamma $ through the pairing
	\begin{equation}
	\label{eq:pairing}
\langle f,x \rangle = \sum_{s\in \Gamma }f_s x_s = (\rho ^fx)_{\matheur{e}_\Gamma }
	\end{equation}
for every $f\in \mathbb{Z}\Gamma $ and $x=(x_s)\in\mathbb{T}^\Gamma $. In this pairing, the shifts $\lambda ^t$ (resp. $\rho ^t$) on $\mathbb{T}^\Gamma $ become dual to left (resp. right) multiplication by $t^{-1}$
on $\mathbb{Z}\Gamma $, and $\lambda ^f$ (resp. $\rho ^f$) is dual to left (resp. right) multiplication by $f^*$ on $\mathbb{Z}\Gamma $.

\smallskip For every subset $J\subseteq \mathbb{Z}\Gamma $ we set
	\begin{equation}
	\label{eq:XJ}
X_J = \{x=(x_s)\in \mathbb{T}^\Gamma \colon \rho ^hx=0\enspace \textup{for every}\enspace h\in J\}.
	\end{equation}
Since the shift actions $\lambda $ and $\rho $ commute, the subgroup $X_J\subseteq \mathbb{T}^\Gamma $ is $\lambda $-invariant. We denote by
	\begin{displaymath}
\lambda _J=\lambda |_{X_J}
	\end{displaymath}
the restriction of $\lambda $ to $X_J$ and call $(X_J,\lambda _J)$ the \emph{algebraic} $\Gamma $-action defined by $J\subseteq \mathbb{Z}\Gamma $. Note that $(X_J,\lambda _J) = (X_{(J)},\lambda _{(J)})$, where $(J)=\mathbb{Z}\Gamma J$ is the left ideal in $\mathbb{Z}\Gamma $ generated by $J$. If $J=(f)$ is the principal ideal generated by an element $f\in \mathbb{Z}\Gamma $, $(X_f,\lambda _f)\coloneqq (X_{(f)},\lambda _{(f)})$ is called the \emph{principal} algebraic $\Gamma $-action defined by $f$.

We fix $J\subseteq \mathbb{Z}\Gamma $ for the moment and set, for every $F\in \mathscr{F}(\Gamma )$,
	\begin{equation}
	\label{eq:alg local}
X_{J,F}=\{x\in \mathbb{T}^\Gamma \colon \rho ^f x=0\enspace \textup{on}\enspace F\enspace \textup{for every}\enspace f\in J\}.
	\end{equation}

The following local entropy formula for the algebraic $\Gamma $-action $(X_J,\lambda _J)$ is a straightforward consequence of Corollary~\ref{C-local}.

	\begin{corollary}
	\label{C-fg}
Let $J\subseteq \mathbb{Z}\Gamma $, and let $\vartheta $ be the usual metric $\vartheta (z+\mathbb{Z},z'+\mathbb{Z})=\min_{n\in \mathbb{Z}}|z-z'-n|$ on $\mathbb{T}=\mathbb{R}/\mathbb{Z}$. Then
	\begin{displaymath}
\htop(X_J)=\sup_{\varepsilon>0}\limsup_F \frac{1}{|F|}\log N_\varepsilon(X_{J, F}, \bar\vartheta _F).
	\end{displaymath}
where $\bar\vartheta _F$ is the pseudometric
	\begin{equation}
	\label{eq:pseudometric}
\bar\vartheta _F(x,y)=\max_{s\in F}\vartheta (x_s,y_s)
	\end{equation}
on $\mathsf{B}^\Gamma $.
	\end{corollary}

	\begin{proof}
For each $f\in J$, take a $K_f\in \mathscr{F}(\Gamma)$ such that $\supp(f)\subseteq K_f$, and
denote by $\mathscr{O}_f$ the open subset of $\mathbb{T}^{K_f}$ consisting of $x\in \mathbb{T}^{K_f}$ such that $(\rho ^fx)_{\matheur{e}_\Gamma }\neq 0_{\mathbb{T}}$. Set
	\begin{displaymath}
\mathcal{P}_J=\{(K_f, \mathscr{O}_f)\colon f\in J\}.
	\end{displaymath}
Then $X_{\mathcal{P}_J}=X_J$.

Note that the continuous pseudometric $\bar{\vartheta}_{\{\matheur{e}_\Gamma\}}$ on $\mathbb{T}^\Gamma $ is dynamically generating in the sense that $\sup_{s\in \Gamma}\bar{\vartheta}_{\{\matheur{e}_\Gamma\}}(\lambda^s x, \lambda^s y)>0$ for all distinct $x, y\in \mathbb{T}^\Gamma$. Thus by \cite{KL16}*{Theorem 9.38} one has
	\begin{displaymath}
\htop(X_J)=\sup_{\varepsilon>0}\limsup_F \frac{1}{|F|}\log N_\varepsilon(X_J, \bar{\vartheta}_F)\le \sup_{\varepsilon>0}\limsup_F \frac{1}{|F|}\log N_\varepsilon(X_{J, F}, \bar{\vartheta}_F).
	\end{displaymath}
In view of Corollary~\ref{C-local} it suffices to show that
	\begin{align*}
\sup_{\varepsilon>0}\limsup_F \frac{1}{|F|}\log N_\varepsilon(X_{J, F}, \bar{\vartheta}_F)\le \sup_{\varepsilon>0}\inf_{J'}\limsup_F \frac{1}{|F|}\log N_\varepsilon(X_{\mathcal{P}_{J'}, F}, \vartheta_F)
	\end{align*}
for $J'$ ranging over finite subsets of $J$.
In turn it suffices to show that for every $\varepsilon>0$ we have
	\begin{equation}
	\label{E-fg}
\limsup_F \frac{1}{|F|}\log N_\varepsilon(X_{J, F}, \bar{\vartheta}_F)\le \inf_{J'}\limsup_F \frac{1}{|F|}\log N_{\varepsilon/2}(X_{\mathcal{P}_{J'}, F}, \vartheta_F)
	\end{equation}
for $J'$ ranging over finite subsets of $J$.

\smallskip Let $\varepsilon>0$, and let $J'$ be a finite subset of $J$.
Take a $K\in \mathscr{F}(\Gamma)$ such that $K\supseteq K_f^{-1}$ for all $f\in J'$. For each $F\in \mathscr{F}(\Gamma)$, denote by $F'_K$ the set of $s\in F$ satisfying $sK\subseteq F$.

We claim that $\pi _{F'_K}(X_{J', F})\subseteq X_{\mathcal{P}_{J'}, F'_K}$ for every $F\in \mathscr{F}(\Gamma)$. Let $x\in X_{J', F}$. Let $s\in \Gamma $ and $f\in J'$ such that $sK_f\subseteq F'_K$. Then $s\in F'_KK_f^{-1}\subseteq F'_KK\subseteq F$. Since $x\in X_{J', F}$, we have $\rho ^fx=0_{\mathbb{T}}$ at $s$. It follows that $(s^{-1}(x|_{F'_K}))|_{K_f}=s^{-1}x|_{K_f}\not\in \mathscr{O}_f$. Therefore $x|_{F'_K}\in X_{\mathcal{P}_{J'}, F'_K}$. This proves our claim.

Take a finite $(\vartheta, \varepsilon/2)$-spanning subset $Z$ of $\mathbb{T}$, i.e., for every $y\in \mathbb{T}$ there is some $z\in Z$ with $\vartheta(y, z)\le \varepsilon/2$. Let $F\in \mathscr{F}(\Gamma)$ be sufficiently right Følner. Then $F'_K$ is nonempty. Take a maximal $(\vartheta_{F'_K}, \varepsilon/2)$-separated subset $Y$ of $X_{\mathcal{P}_{J'}, F'_K}$. Then $Y$ is also $(\vartheta_{F'_K}, \varepsilon/2)$-spanning in $X_{\mathcal{P}_{J'}, F'_K}$. Let $W\subseteq X_{J', F}$ be $(\bar{\vartheta}_F, \varepsilon)$-separated. Then for each $w\in W$ we can find some $u(w)\in Z^{F\smallsetminus F'_K}\times Y\subseteq \mathbb{T}^F$ such that $\vartheta_F(w|_F, u(w))\le \varepsilon/2$. Since $W$ is $(\bar{\vartheta}_F, \varepsilon)$-separated, the map
$W\rightarrow Z^{F\smallsetminus F'_K}\times Y$ sending $w$ to $u(w)$ must be injective. Thus
	\begin{displaymath}
|W|\le |Z^{F\smallsetminus F'_K}\times Y|\le |Z|^{|F\smallsetminus F'_K|}N_{\varepsilon/2}(X_{\mathcal{P}_{J'}, F'_K}, \vartheta_{F'_K}).
	\end{displaymath}
Then
	\begin{displaymath}
N_\varepsilon(X_{J', F}, \bar{\vartheta}_F)\le |Z|^{|F\smallsetminus F'_K|}N_{\varepsilon/2}(X_{\mathcal{P}_{J'}, F'_K}, \vartheta_{F'_K}).
	\end{displaymath}
As $F\in \mathscr{F}(\Gamma)$ becomes more and more right invariant in the same sense as in \eqref{eq:limit}, one has $\frac{|F\smallsetminus F'_K|}{|F|}\to 0$ and $F'_K$ also becomes more and more right invariant. Therefore
	\begin{align*}
\limsup_F\frac{1}{|F|}\log N_\varepsilon(X_{J', F}, \bar{\vartheta}_F)&\le \limsup_F \frac{|F\smallsetminus F'_K|}{|F|}\log |Z|
	\\
&\qquad \qquad +\limsup_F \frac{1}{|F|}\log N_{\varepsilon/2}(X_{\mathcal{P}_{J'}, F'_K}, \vartheta_{F'_K})
	\\
&\le \limsup_F \frac{1}{|F|}\log N_{\varepsilon/2}(X_{\mathcal{P}_{J'}, F}, \vartheta_F).
	\end{align*}
Then
	\begin{align*}
\limsup_F\frac{1}{|F|}\log N_\varepsilon(X_{J, F}, \bar{\vartheta}_F)&\le \inf_{J'}\limsup_F\frac{1}{|F|}\log N_\varepsilon(X_{J', F}, \bar{\vartheta}_F)
	\\
&\le \inf_{J'}\limsup_F \frac{1}{|F|}\log N_{\varepsilon/2}(X_{\mathcal{P}_{J'}, F}, \vartheta_F)
	\end{align*}
for $J'$ ranging over finite subsets of $J$, establishing \eqref{E-fg}.
	\end{proof}

As stated, the local entropy formula in Corollary \ref{C-fg} applies to all closed, left shift-invariant subgroups of $\mathbb{T}^\Gamma $. The local entropy formula for finitely generated algebraic actions in \cite{LT}*{Theorem 4.2} corresponds to the extension of Corollary \ref{C-fg} to all left shift-invariant subgroups of $(\mathbb{T}^d)^\Gamma $, $d\ge 1$. The proof of this latter extension requires only small notational changes to the argument in Corollary \ref{C-fg}, as described in the following brief discussion.

\smallskip Fix $d\ge 1$, set $\mathsf{B}=\mathbb{T}^d$, and write the elements of $\mathsf{B}$ as $b=(b^{(1)},\dots ,b^{(d)})$ with $b^{(i)}\in \mathbb{T}$. By identifying each $x =({b}_s)_{s\in \Gamma }\in \mathsf{B}^\Gamma $ with the $d$-tuple $(x^{(i)}=(b^{(i)}_s)_{s\in \Gamma }, i=1,\dots ,d)$ of elements in $\mathbb{T}^\Gamma $ we obtain a shift-equivariant isomorphism of the compact abelian groups $\mathsf{B}^\Gamma $ and $(\mathbb{T}^\Gamma )^d$.

For every $h=(h^{(1)},\dots ,h^{(d)})\in (\mathbb{Z}\Gamma )^d$ and every $x=(x^{(1)},\dots ,x^{(d)}) \in (\mathbb{T}^\Gamma )^d = \mathsf{B}^\Gamma $ we set
	\begin{equation}
	\label{eq:pairing2}
\langle \negthinspace \langle h,x \rangle \negthinspace \rangle = \rho ^{h^{(1)}}\hspace{-.4em} x^{(1)} + \dots + \rho ^{h^{(d)}}\hspace{-.4em} x^{(d)} \in \mathbb{T}^\Gamma .
	\end{equation}
If we evaluate the pairing $(\mathbb{Z}\Gamma )^d\times \mathsf{B}^\Gamma \to \mathbb{T}^\Gamma $ in \eqref{eq:pairing2} at the coordinate $\matheur{e}_\Gamma $, the map $(h,x)\to \langle h,x\rangle \coloneqq \langle \negthinspace \langle h,x \rangle \negthinspace \rangle _{\matheur{e}_\Gamma } = (\rho ^{h^{(1)}}\hspace{-.4em} x^{(1)})_{\matheur{e}_\Gamma } + \dots + (\rho ^{h^{(d)}}\hspace{-.4em} x^{(d)})_{\matheur{e}_\Gamma }$ is a map from $(\mathbb{Z}\Gamma )^d\times \mathsf{B}^\Gamma $ to $\mathbb{T}$ which allows us to identify $(\mathbb{Z}\Gamma )^d$ with the Pontryagin dual $\widehat{\mathsf{B}^\Gamma }$ of $\mathsf{B}^\Gamma $ (as in \eqref{eq:pairing}).

\smallskip In analogy to \eqref{eq:XJ} -- \eqref{eq:alg local} we define, for every subset $J\subseteq (\mathbb{Z}\Gamma )^d$ and $F\in \mathscr{F}(\Gamma )$,
	\begin{equation*}
	\begin{gathered}
X_J = \{x\in \mathsf{B}^\Gamma \colon \langle h,x \rangle = 0_{\mathbb{T}^\Gamma }\enspace \textup{for every}\enspace h\in J\},
	\\
X_{J,F}=\{x\in \mathsf{B}^\Gamma \colon \langle h,x \rangle =0\enspace \textup{on}\enspace F\enspace \textup{for every}\enspace h\in J\}.
	\end{gathered}
	\end{equation*}
Then we have the following restatement of \cite{LT}*{Theorem 4.2}.\footnote{\,Note that \cite{LT}*{Theorem 4.2} uses the \emph{right} shift action of $\Gamma $ on $(\mathbb{T}^d)^\Gamma $ and \emph{left} Følner sets.}

	\begin{corollary}
	\label{C-fg-d}
Let $d\ge 1$, $J\subseteq (\mathbb{Z}\Gamma )^d$, and let $\vartheta $ be a compatible metric on $\mathbb{T}^d$. Then
	\begin{displaymath}
\htop(X_J)=\sup_{\varepsilon>0}\limsup_F \frac{1}{|F|}\log N_\varepsilon(X_{J, F}, \bar{\vartheta}_F),
	\end{displaymath}
where $\bar{\vartheta }_F$ is defined as in \eqref{eq:pseudometric}.
	\end{corollary}

\subsection{The local pressure formula for subshifts}
	\label{SS-local subshift}

Let $\Gamma $ be a countable discrete group and let $\mathsf{B}$ a nonempty \emph{finite} set. As in Subsection \ref{SS-local shift} we assume that $\mathcal{P}$ is a set of pairs $(K,\mathscr{O})$, where $K\in \mathscr{F}(\Gamma )$ and $\mathscr{O}\subseteq \mathsf{B}^K$, and define $X_\mathcal{P}\subseteq \mathsf{B}^\Gamma $ by \eqref{eq:admissible}. If the set $\mathcal{P}$ can be chosen to be finite, the space $X_\mathcal{P}\subseteq \mathsf{B}^\Gamma $ is called a \emph{shift of finite type} (or \emph{SFT}).

For $F\in \mathscr{F}(\Gamma)$ and $k\in \mathbb{N}$, we say that a finite collection $\mathcal{E}$ of finite subsets of $\Gamma$, with possible repetitions, is a $k$-cover of $F$, if $\sum_{E\in \mathcal{E}}1_E\ge k 1_F$. The case $f=0$ of the following lemma was proved by Downarowicz, Frej and Romagnoli \cite{DFR}*{Lemma 6.1}, who in turned followed the proof of \cite{BT}*{Theorem 2}. We follow the proof of \cite{DFR}*{Lemma 6.1}.

	\begin{lemma}
	\label{L-infimum rule}
Let $f\colon \mathsf{B}\rightarrow \mathbb{R}_{\ge 0}$. Let $F\in \mathscr{F}(\Gamma)$, $k\in \mathbb{N}$, and $\mathcal{E}$ a $k$-cover of $F$. Then for each $X\subseteq \mathsf{B}^\Gamma$ one has
	\begin{displaymath}
\sum_{x\in \pi_F(X)}e^{\sum_{s\in F}f(x_s)}\le \prod_{E\in \mathcal{E}}\Biggl(\sum_{z\in \pi_E(X)}e^{\sum_{s\in E}f(z_s)}\Biggr)^{1/k},
	\end{displaymath}
where as convention for $E=\varnothing$ we set $\sum_{z\in \pi_E(X)}e^{\sum_{s\in E}f(z_s)}=1$.
	\end{lemma}

	\begin{proof}
For each finite set $E\subseteq \Gamma$ we set
	\begin{displaymath}
\varphi(X, E)=\sum_{x\in \pi_E(X)}e^{\sum_{s\in F}f(x_s)}.
	\end{displaymath}
Our convention is that $\varphi(X, \varnothing)=1$.

We argue via induction on $|F|$. Consider first the case $|F|=1$. Say, $F=\{t\}$. Since $\mathcal{E}$ is a $k$-cover of $F$, we can find an $\mathcal{E}'\subseteq \mathcal{E}$ such that $|\mathcal{E}'|=k$ and $t\in E$ for every $E\in \mathcal{E}'$. Since $f\ge 0$, we have $\varphi(X, E)\ge 1$ for every $E\in \mathcal{E}$ and $\varphi(X, E)\ge \varphi(X, F)$ for every $E\in \mathcal{E}'$. Then
	\begin{displaymath}
\varphi(X, F)\le \prod_{E\in \mathcal{E}'}\varphi(X, E)^{1/k}\le \prod_{E\in \mathcal{E}}\varphi(X, E)^{1/k}.
	\end{displaymath}

Assume that $|F|\ge 2$ and the statement holds for each $F'\subseteq F$ with $|F'|=|F|-1$. Fix one $t\in F$ and set $F'=F\smallsetminus \{t\}$. Then $F'\subseteq F$ and $|F'|=|F|-1$. We set $E'=E\smallsetminus \{t\}$ for each $E\in \mathcal{E}$. Then $\mathcal{E}'=\{E'\colon E\in \mathcal{E}\}$ is a $k$-cover of $F'$.

For each $b\in \mathsf{B}$, we set $X_b=\{x\in X\colon x_t=b\}$. Applying the inductive hypothesis to $F', \mathcal{E}'$ and $X_b$, we have
	\begin{equation}
	\label{E-infimum rule}
\varphi(X_b, F')\le \prod_{E\in \mathcal{E}}\varphi(X_b, E')^{1/k}.
	\end{equation}
We denote by $\mathcal{E}_1$ the set of $E\in \mathcal{E}$ satisfying $t\in E$, and by $\mathcal{E}_2$ the set of $E\in \mathcal{E}$ satisfying $t\not\in E$.
For each $b\in \mathsf{B}$ and $E\in \mathcal{E}_2$, we have
	\begin{displaymath}
\varphi(X_b, E')=\varphi(X_b, E)\le \varphi(X, E).
	\end{displaymath}
Then for each $b\in \mathsf{B}$ from \eqref{E-infimum rule} we have
	\begin{align*}
\varphi(X_b, F')&\le \Biggl(\prod_{E\in \mathcal{E}_1}\varphi(X_b, E')^{1/k}\Biggr)\Biggl(\prod_{E\in \mathcal{E}_2}\varphi(X_b, E')^{1/k}\Biggr)
	\\
& \le \Biggl(\prod_{E\in \mathcal{E}_1}\varphi(X_b, E')^{1/k}\Biggr)\Biggl(\prod_{E\in \mathcal{E}_2}\varphi(X, E)^{1/k}\Biggr).
	\end{align*}
Thus
	\begin{align*}
\varphi(X, F)&=\sum_{b\in \mathsf{B}}e^{f(b)}\varphi(X_b, F')\le \Biggl(\prod_{E\in \mathcal{E}_2}\varphi(X, E)^{1/k}\Biggr)\sum_{b\in \mathsf{B}}e^{f(b)}\prod_{E\in \mathcal{E}_1}\varphi(X_b, E')^{1/k}
	\\
&=\Biggl(\prod_{E\in \mathcal{E}_2}\varphi(X, E)^{1/k}\Biggr)\sum_{b\in \mathsf{B}}\prod_{E\in \mathcal{E}_1}(e^{f(b)k/|\mathcal{E}_1|}\varphi(X_b, E'))^{1/k}.
	\end{align*}
Since $\mathcal{E}$ is a $k$-cover of $F$, we have $|\mathcal{E}_1|\ge k$. As $f\ge 0$, we have $f(b)k/|\mathcal{E}_1|\le f(b)$ for every $b\in \mathsf{B}$. Then from the above inequality we have
	\begin{align*}
\varphi(X, F)&\le \Biggl(\prod_{E\in \mathcal{E}_2}\varphi(X, E)^{1/k}\Biggr)\sum_{b\in \mathsf{B}}\prod_{E\in \mathcal{E}_1}(e^{f(b)}\varphi(X_b, E'))^{1/k}
	\\
&=\Biggl(\prod_{E\in \mathcal{E}_2}\varphi(X, E)^{1/k}\Biggr)\sum_{b\in \mathsf{B}}\prod_{E\in \mathcal{E}_1}\varphi(X_b, E)^{1/k}.
	\end{align*}
For each $E\in \mathcal{E}_1$, we define a function $g_E\colon \mathsf{B}\rightarrow \mathbb{R}$ by $g_E(b)=\varphi(X_b, E)^{1/k}$ for $b\in \mathsf{B}$. For each $q\ge 1$, we denote by $\|\cdot\|_q$ the $\ell_q$-norm on $\ell_q(\mathsf{B})$ given by $\|g\|_q=\left(\sum_{b\in \mathsf{B}}|g(b)|^q\right)^{1/q}$. By the Hölder inequality we have
	\begin{displaymath}
\sum_{b\in \mathsf{B}}\prod_{E\in \mathcal{E}_1}\varphi(X_b, E)^{1/k}=\sum_{b\in \mathsf{B}}\prod_{E\in \mathcal{E}_1}g_E(b)\le \prod_{E\in \mathcal{E}_1}\|g_E\|_{|\mathcal{E}_1|}\le \prod_{E\in \mathcal{E}_1}\|g_E\|_k=\prod_{E\in \mathcal{E}_1}\varphi(X, E)^{1/k}.
	\end{displaymath}
Therefore
	\begin{displaymath}
\varphi(X, F)\le \Biggl(\prod_{E\in \mathcal{E}_2}\varphi(X, E)^{1/k}\Biggr)\Biggl(\prod_{E\in \mathcal{E}_1}\varphi(X, E)^{1/k}\Biggr)=\prod_{E\in \mathcal{E}}\varphi(X, E)^{1/k}.
	\end{displaymath}
This finishes the induction step and proves the lemma.
	\end{proof}

Now assume that $\Gamma $ is amenable. If $\mathscr{U}$ is the finite open cover of $\mathsf{B}$ consisting of singletons we have
	\begin{displaymath}
N(\mathscr{U}^F, X_{\mathcal{P}, F})=|X_{\mathcal{P}, F}| \text{ and } p_F(f, \mathscr{U}, X_{\mathcal{P}, F})=\sum_{x\in X_{\mathcal{P}, F}}e^{\sum_{s\in F}f(x_s)}
	\end{displaymath}
for every $F\in \mathscr{F}(\Gamma)$ and $f\colon \mathsf{B}\rightarrow \mathbb{R}$. From Lemma~\ref{L-local} we get

	\begin{lemma}
	\label{L-local entropy for subshift}
The limits $\lim_F \frac{1}{|F|}\log |X_{\mathcal{P}, F}|$ and $\lim_F \frac{1}{|F|}\log \left(\sum_{x\in X_{\mathcal{P}, F}}e^{\sum_{s\in F}f(x_s)}\right)$ exist in the same sense as in \eqref{eq:limit} for every $f\colon \mathsf{B}\rightarrow \mathbb{R}$.
	\end{lemma}

In \cite{Friedland97}, when $X_\mathcal{P}$ is of finite type, the limit $\lim_F \frac{1}{|F|}\log |X_{\mathcal{P}, F}|$ is called the \textup{combinatorial entropy} of $X_\mathcal{P}$.

It is well known \cite{DFR}*{Corollary 6.3} that
	\begin{displaymath}
\htop(X_\mathcal{P})=\inf_{F\in \mathscr{F}(\Gamma)}\frac{1}{|F|}\log |\pi _F(X_\mathcal{P})|=\lim_F\frac{1}{|F|}\log|\pi _F(X_\mathcal{P})|.
	\end{displaymath}

	\begin{theorem}
	\label{T-local subshift}
For every $f\colon \mathsf{B}\rightarrow \mathbb{R}$, we have
	\begin{align*}
P(X_\mathcal{P}, f)&=\lim_F \frac{1}{|F|}\log \Biggl(\sum_{x\in X_{\mathcal{P}, F}}e^{\sum_{s\in F}f(x_s)}\Biggr)
	\\
&=\inf_{F\in \mathscr{F}(\Gamma)} \frac{1}{|F|}\log \Biggl(\sum_{x\in X_{\mathcal{P}, F}}e^{\sum_{s\in F}f(x_s)}\Biggr)=\inf_{F\in \mathscr{F}(\Gamma)} \frac{1}{|F|}\log \Biggl(\sum_{x\in \pi_F(X_\mathcal{P})}e^{\sum_{s\in F}f(x_s)}\Biggr).
	\end{align*}
In particular,
	\begin{displaymath}
\htop(X_\mathcal{P})=\inf_{F\in \mathscr{F}(\Gamma)} \frac{1}{|F|}\log |X_{\mathcal{P}, F}|=\lim_F \frac{1}{|F|}\log |X_{\mathcal{P}, F}|.
	\end{displaymath}
	\end{theorem}

	\begin{proof}
Let $\mathscr{U}$ be the finite open cover of $\mathsf{B}$ consisting of singletons. From Theorem~\ref{T-local} we have
	\begin{displaymath}
P(X_\mathcal{P}, f)=\lim_F \frac{1}{|F|}\log p_F(f, \mathscr{U}, X_{\mathcal{P}, F})=\lim_F \frac{1}{|F|}\log \Biggl(\sum_{x\in X_{\mathcal{P}, F}}e^{\sum_{s\in F}f(x_s)}\Biggr).
	\end{displaymath}

Take a constant $C$ such that $f+C\ge 0$. We may assume that $X_\mathcal{P}\neq \varnothing$. A result of Downarowicz, Frej and Romagnoli \cite{DFR}*{Proposition 3.3} says that if a function $\varphi \colon \mathscr{F}(\Gamma)\rightarrow [0, \infty)$ is $\Gamma$-invariant in the sense that $\varphi(sF)=\varphi(F)$ for all $F\in \mathscr{F}(\Gamma)$ and $s\in \Gamma$, and satisfies Shearer's inequality in the sense that $\varphi(F)\le \frac{1}{k}\sum_{E\in \mathcal{E}}\varphi(E)$ for any $F\in \mathscr{F}(\Gamma)$, any $k\in \mathbb{N}$ and any $k$-cover $\mathcal{E}$ of $F$ consisting of nonempty sets, then
	\begin{displaymath}
\lim_F \frac{1}{|F|}\varphi(F)=\inf_{F\in \mathscr{F}(\Gamma)}\frac{1}{|F|}\varphi(F).
	\end{displaymath}
By Lemma~\ref{L-infimum rule} the functions $\varphi, \psi\colon \mathscr{F}(\Gamma)\rightarrow [0, \infty)$ defined by
	\begin{displaymath}
\varphi(F)=\log p_F(f+C, \mathscr{U}, \pi_F(X_\mathcal{P})) \text{ and } \psi(F)=\log p_F(f+C, \mathscr{U}, X_{\mathcal{P}, F})
	\end{displaymath}
both are $\Gamma$-invariant and satisfy Shearer's inequality, so that
	\begin{align*}
P&(X_\mathcal{P}, f)=-C+P(X_\mathcal{P}, f+C)=-C+\lim_F \frac{1}{|F|}\log p_F(f+C, \mathscr{U}, \pi_F(X_\mathcal{P}))
	\\
&=-C+\inf_{F\in \mathscr{F}(\Gamma)} \frac{1}{|F|}\log p_F(f+C, \mathscr{U}, \pi_F(X_\mathcal{P}))=\inf_{F\in \mathscr{F}(\Gamma)} \frac{1}{|F|}\log \Biggl(\sum_{x\in \pi_F(X_\mathcal{P})}e^{\sum_{s\in F}f(x_s)}\Biggr),
	 \end{align*}
and
	\begin{align*}
P&(X_\mathcal{P}, f)=-C+P(X_\mathcal{P}, f+C)=-C+\lim_F \frac{1}{|F|}\log p_F(f+C, \mathscr{U}, X_{\mathcal{P}, F})
	\\
&=-C+\inf_{F\in \mathscr{F}(\Gamma)} \frac{1}{|F|}\log p_F(f+C, \mathscr{U}, X_{\mathcal{P}, F})=\inf_{F\in \mathscr{F}(\Gamma)} \frac{1}{|F|}\log \Biggl(\sum_{x\in X_{\mathcal{P}, F}}e^{\sum_{s\in F}f(x_s)}\Biggr).\tag*{\qedsymbol}
	 \end{align*}
\renewcommand{\qedsymbol}{}
\vspace{-\baselineskip}
	 \end{proof}

The equality $\htop(X_\mathcal{P})=\lim_F \frac{1}{|F|}\log |X_{\mathcal{P}, F}|$ in the case where $\Gamma =\mathbb{Z}^d$ and $\mathcal{P}$ is finite is proved by Friedland \cite{Friedland97}*{Theorem 2.5}. The equality $\htop(X_\mathcal{P})=\inf_{F\in \mathscr{F}(\Gamma)} \frac{1}{|F|}\log |X_{\mathcal{P}, F}|$ seems new even when $\Gamma=\mathbb{Z}$ and $\mathcal{P}$ is finite.

\section{Local entropy for restricted permutations}
	\label{S-local permutation}

Let $\Gamma $ be a countable group, and let $\mathscr{S}(\Gamma )$ be the group of all permutations of $\Gamma $. Fix a set $\mathsf{A}\in \mathscr{F}(\Gamma )$ and consider the set $\mathscr{S}_\mathsf{A}\subseteq \mathscr{S}(\Gamma )$ consisting of all permutations $\sigma \colon s\mapsto \sigma (s)$ of $\Gamma $ with the property that
	\begin{equation}
	\label{eq:perm}
x^{(\sigma )}_s \coloneqq s^{-1}\sigma (s)\in \mathsf{A}\enspace \textup{for every}\enspace s\in \Gamma .
	\end{equation}
Every $\sigma \in \mathscr{S}_\mathsf{A}$ is determined completely by the point $x^{(\sigma )}=(x^{(\sigma )}_s)_{s\in \Gamma }\in \mathsf{A}^\Gamma $ in \eqref{eq:perm}. Conversely we define, for every $F\subseteq \Gamma $ and $x = (x_s)_{s\in F}\in \mathsf{A}^F$, a map $\varphi ^{(x)}\colon F\rightarrow \Gamma $ by
	\begin{equation}
	\label{eq:px}
\varphi ^{(x)}(s)=sx_s,\enspace s\in F.
	\end{equation}
Put
	\begin{equation}
	\label{eq:XA}
X_\mathsf{A}^\iota =\{x \in \mathsf{A}^\Gamma \colon \varphi ^{(x)} \enspace \textup{is injective}\},\qquad X_\mathsf{A} = \{x\in \mathsf{A}^\Gamma \colon \varphi ^{(x)} \enspace \textup{is bijective}\}.
	\end{equation}
The sets $X_\mathsf{A}\subseteq X_\mathsf{A}^\iota \subseteq \mathsf{A}^\Gamma $ are closed and hence compact in the product topology, and they are invariant under the left shift action $\lambda $ of $\Gamma $ on $\mathsf{A}^\Gamma $ in \eqref{eq:shift}.

	\begin{lemma}
	\label{L-conjugacy}
For every $\mathsf{A} \in \mathscr{F}(\Gamma )$ and $t\in \Gamma $, the dynamical systems $(X_{\mathsf{A}t},\lambda )$ and $(X_\mathsf{A},\lambda )$ {\upshape(}resp., $(X_{\mathsf{A}t}^\iota ,\lambda )$ and $(X_\mathsf{A}^\iota ,\lambda )${\upshape)} are topologically conjugate.
	\end{lemma}

	\begin{proof}
For every $t\in \Gamma $ we denote by $\varrho ^t\in \mathscr{S}(\Gamma )$ right multiplication by $t$, i.e., $\varrho ^t(s)=st$ for every $s\in \Gamma $. The shift-equivariant map $\tilde{\varrho }\mspace{1mu}^t\colon \mathsf{A}^\Gamma \rightarrow (\mathsf{A}t)^\Gamma $, given by $(\tilde{\varrho }\mspace{1mu}^tx)_s=\varrho ^t(x_s)=x_st$ for all $x=(x_s)\in \mathsf{A}^\Gamma $ and $s\in \Gamma $, satisfies that $\tilde{\varrho }\mspace{1mu}^t(X_\mathsf{A}) = X_{\mathsf{A}t}$ (by \eqref{eq:perm}), and $(X_\mathsf{A},\lambda )$ and $(X_{\mathsf{A}t},\lambda )$ are topologically conjugate. Similarly we see that $(X_{\mathsf{A}t}^\iota ,\lambda )$ and $(X_\mathsf{A}^\iota ,\lambda )$ are conjugate.
	\end{proof}

	\begin{remark}
	\label{r:1 in A}
In view of Lemma \ref{L-conjugacy} we may assume without loss in generality that $\matheur{e}_\Gamma \in \mathsf{A}$.
	\end{remark}

	\begin{proposition}
	\label{P-SFT1}
For every $\mathsf{A}\in \mathscr{F}(\Gamma )$, the sets $X_\mathsf{A}\subseteq X_\mathsf{A}^\iota \subseteq \mathsf{A}^\Gamma $ are \emph{SFT}s.
	\end{proposition}

	\begin{proof}
For any two distinct $t_1,t_2\in \mathsf{A}$ we put $K_{t_1,t_2}= \{\matheur{e}_\Gamma ,t_1t_2^{-1}\}$ and define $y^{(t_1,t_2)}\in \mathsf{A}^{K_{t_1,t_2}}$ by $y^{(t_1,t_2)}_{\matheur{e}_\Gamma }=t_1$ and $y^{(t_1,t_2)}_{t_1t_2^{-1}}=t_2$. If we set
	\begin{equation}
	\label{eq:PIA}
\mathcal{P}_\mathsf{A}^\iota =\bigl\{(K_{t_1,t_2},\{y^{(t_1,t_2)}\})\colon t_1,t_2\in \mathsf{A}, t_1\ne t_2\bigr\},
	\end{equation}
then one checks easily that, in the notation of \eqref{eq:admissible} -- \eqref{eq:locally admissible}, $X_{\mathcal{P}_\mathsf{A}^\iota } = X_\mathsf{A}^\iota $. Similarly, if
	\begin{equation}
	\label{eq:PIA'}
\mathcal{P}_\mathsf{A}^\star=\bigl\{(\mathsf{A}^{-1},\{z\}) \colon z\in \mathsf{A}^{\mathsf{A}^{-1}}\enspace \textup{and}\enspace z_{t^{-1}}\ne t\enspace \textup{for all}\enspace t\in \mathsf{A}\bigr\}
	\end{equation}
and $\mathcal{P}_\mathsf{A}=\mathcal{P}_\mathsf{A}^\iota \cup \mathcal{P}_\mathsf{A}^\star$, then $X_{\mathcal{P}_\mathsf{A}} = X_\mathsf{A}$.
	\end{proof}

Now assume that $\Gamma $ is amenable. We use right F{\o}lner sets.

The case $\Gamma=\mathbb{Z}^d$ of the following lemma is proved in \cite{Elimelech21}*{Proposition 10} using the ergodic theorem, and the proof there also works for general amenable $\Gamma $. We give a different proof.

	\begin{lemma}
	\label{L-inv measure are on perm}
Let $\mu $ be a $\Gamma $-invariant Borel probability measure on $X_\mathsf{A}^\iota $. Then $\supp(\mu)\subseteq X_\mathsf{A}$.
	\end{lemma}

	\begin{proof}
Let $x\in X_\mathsf{A}^\iota \smallsetminus X_\mathsf{A}$. Then there is some $g\in \Gamma $ such that $g\neq sx_s$ for all $s\in \Gamma $. Denote by $Y$ the set of $y\in X_\mathsf{A}^\iota $ satisfying that $y=x$ on $g\mathsf{A}^{-1}$. Then $Y$ is a closed and open subset of $X_\mathsf{A}^\iota $ containing $x$, and $g\neq sy_s$ for all $y\in Y$ and $s\in \Gamma $. It suffices to show that $\mu(Y)=0$.

Suppose that $\mu(Y)>0$. Take $F\in \mathscr{F}(\Gamma)$ such that $|F(\mathsf{A}\cup \{g\})|<(1+\mu(Y))|F|$. Then
	\begin{displaymath}
\int_{X_\mathsf{A}^\iota }\sum_{t\in F} 1_Y(t^{-1}y)\, d\mu(y)=|F|\mu(Y),
	\end{displaymath}
and hence there is some $z\in X_\mathsf{A}^\iota $ satisfying $\sum_{t\in F}1_Y(t^{-1}z)\ge |F|\mu(Y)$. Set $W=\{t\in F\colon t^{-1}z\in Y\}$. Then $|W|=\sum_{t\in F}1_Y(t^{-1}z)\ge |F|\mu(Y)$.
For each $t\in W$, we have $t^{-1}z\in Y$ and hence $g\neq s(t^{-1}z)_s=sz_{ts}$ for all $s\in \Gamma $, which implies that $g\neq t^{-1}sz_s$ for all $s\in \Gamma $. Put $V=\{sz_s\colon s\in F\}\subseteq F\mathsf{A}$. Then $Wg\cap V=\varnothing $, and $|V|=|F|$. Now we have
	\begin{displaymath}
|F|\mu(Y)+|F|\le |Wg|+|V|=|Wg\cup V|\le |F(\mathsf{A}\cup \{g\})|<(1+\mu(Y))|F|,
	\end{displaymath}
a contradiction.
	\end{proof}

From Lemma~\ref{L-inv measure are on perm} and the variational principle \cite{KL16}*{Theorem 9.48} we obtain the following consequence which is proved in \cite{Elimelech21}*{Theorem 10} in the case $\Gamma =\mathbb{Z}^d$.

	\begin{proposition}
	\label{P-equal entropy}
Let $\Gamma $ be a countable amenable group, and let $\mathsf{A}\in \mathscr{F}(\Gamma )$. Then $\htop(X_\mathsf{A})=\htop(X_\mathsf{A}^\iota )$.
	\end{proposition}

\medskip For discussing local entropies of the spaces $X_\mathsf{A}$ and $X_\mathsf{A}^\iota $ we set, for all $F,F'\in \mathscr{F}(\Gamma )$,
	\begin{equation}
	\label{eq:XAF}
	\begin{gathered}
X_{\mathsf{A},F}^\iota =\bigl\{x \in \mathsf{A}^F\colon \varphi ^{(x)} \enspace \textup{is injective}\bigr\},
	\\
X_{\mathsf{A},F} = \bigl\{x\in X_{\mathsf{A},F}^\iota \colon \varphi ^{(x)}(F)\supseteq \{s\in \Gamma \colon s\mathsf{A}^{-1}\subseteq F\}\bigr\}.
	\\
X_{\mathsf{A},F,F'} = \{x\in X_{\mathsf{A},F}^\iota\colon \varphi ^{(x)}(F)\supseteq F'\}.
	\end{gathered}
	\end{equation}
It is easily checked that, in the notation of Proposition \ref{P-SFT1}, \eqref{eq:PIA} -- \eqref{eq:PIA'},
	\begin{displaymath}
X_{\mathcal{P}_\mathsf{A}^\iota , F}=X_{\mathsf{A}, F}^\iota \enspace \textup{and}\enspace X_{\mathcal{P}_\mathsf{A}, F}=X_{\mathsf{A}, F}
	\end{displaymath}
for every $F\in \mathscr{F}(\Gamma )$.

	\begin{proposition}
	\label{P-entropy in local}
We have
	\begin{displaymath}
\htop(X_\mathsf{A})=\inf_{F\in \mathscr{F}(\Gamma)}\frac{1}{|F\mathsf{A}^{-1}|}\log |X_{\mathsf{A}, F\mathsf{A}^{-1}, F}|=\lim_F\frac{1}{|F|}\log |X_{\mathsf{A}, F\mathsf{A}^{-1}, F}|=\lim_F \frac{1}{|F|}\log |X_{\mathsf{A}, F}|,
	\end{displaymath}
and
	\begin{displaymath}
\htop(X_\mathsf{A}^\iota )=\inf_{F\in \mathscr{F}(\Gamma)}\frac{1}{|F|}\log |X_{\mathsf{A}, F}^\iota |=\lim_F \frac{1}{|F|}\log |X_{\mathsf{A}, F}^\iota |.
	\end{displaymath}
	\end{proposition}
	\begin{proof}
From Theorem~\ref{T-local subshift} we have
	\begin{equation}
	\label{E-entropy in local1}
\htop(X_\mathsf{A})= \inf_{F\in \mathscr{F}(\Gamma)} \frac{1}{|F|}\log |X_{\mathsf{A}, F}|=\lim_F \frac{1}{|F|}\log |X_{\mathsf{A}, F}|,
	\end{equation}
and
	\begin{displaymath}
\htop(X_\mathsf{A}^\iota )= \inf_{F\in \mathscr{F}(\Gamma)} \frac{1}{|F|}\log |X_{\mathsf{A}, F}^\iota |=\lim_F \frac{1}{|F|}\log |X_{\mathsf{A}, F}^\iota |.
	\end{displaymath}
We are left to show that
	\begin{equation}
	\label{E-entropy in local2}
\htop(X_\mathsf{A})=\inf_{F\in \mathscr{F}(\Gamma)}\frac{1}{|F\mathsf{A}^{-1}|}\log |X_{\mathsf{A}, F\mathsf{A}^{-1}, F}|=\lim_F\frac{1}{|F|}\log |X_{\mathsf{A}, F\mathsf{A}^{-1}, F}|.
	\end{equation}
For every $F\in \mathscr{F}(\Gamma )$ we have
	\begin{displaymath}
X_{\mathsf{A},F\mathsf{A}^{-1}}\subseteq X_{\mathsf{A},F\mathsf{A}^{-1},F}.
	\end{displaymath}
With reference to Remark \ref{r:1 in A} we take it that $\matheur{e}_\Gamma \in \mathsf{A}$. Then the restriction of each $x\in X_{\mathsf{A},F\mathsf{A}^{-1},F}$ to $F$ lies in $X_{\mathsf{A},F}$, and hence
	\begin{displaymath}
|X_{\mathsf{A}, F\mathsf{A}^{-1}, F}|\le |X_{\mathsf{A}, F}|\cdot |\mathsf{A}|^{|F\mathsf{A}^{-1}\smallsetminus F|}.
	\end{displaymath}
Thus
	\begin{equation}
	\label{E-entropy in local3}
	\begin{aligned}
\inf_{F\in \mathscr{F}(\Gamma)} \frac{1}{|F|}&\log |X_{\mathsf{A}, F}|\le \inf_{F\in \mathscr{F}(\Gamma)} \frac{1}{|F\mathsf{A}^{-1}|}\log |X_{\mathsf{A}, F\mathsf{A}^{-1}}|
	\\
&\le \inf_{F\in \mathscr{F}(\Gamma)}\frac{1}{|F\mathsf{A}^{-1}|}\log |X_{\mathsf{A}, F\mathsf{A}^{-1}, F}|
\le \liminf_F\frac{1}{|F\mathsf{A}^{-1}|}\log |X_{\mathsf{A}, F\mathsf{A}^{-1}, F}|
	\\
&\le \liminf_F\frac{1}{|F|}\log |X_{\mathsf{A}, F\mathsf{A}^{-1}, F}| \le \limsup_F\frac{1}{|F|}\log |X_{\mathsf{A}, F\mathsf{A}^{-1}, F}|
	\\
&\le \limsup_F\frac{1}{|F|}\log \bigl(|X_{\mathsf{A}, F}|\cdot |\mathsf{A}|^{|F\mathsf{A}^{-1}\smallsetminus F|}\bigr)=\lim_F\frac{1}{|F|}\log |X_{\mathsf{A}, F}|.
	\end{aligned}
	\end{equation}
Now \eqref{E-entropy in local2} follows from \eqref{E-entropy in local3} and \eqref{E-entropy in local1}.
	\end{proof}

	\section{Permanents for elements of the group algebra}
	\label{S-permanent}

Let $\Gamma $ be a countable discrete group, $\mathsf{A}\in \mathscr{F}(\Gamma )$, and let $X_\mathsf{A}\subseteq X_\mathsf{A}^\iota \subseteq \mathsf{A}^\Gamma $ be the subshifts defined in \eqref{eq:XA}. For every $F\in \mathscr{F}(\Gamma )$ we define $X_{\mathsf{A},F}\subseteq X_{\mathsf{A},F}^\iota \subseteq \mathsf{A}^F$ by \eqref{eq:XAF}.

Let $\mathbb{C}\Gamma $ be the \emph{complex group ring} of $\Gamma $, i.e., the set of all finitely supported maps $f\colon s\mapsto f_s$ from $\Gamma $ to $\mathbb{C}$. For $f, g\in \mathbb{C}\Gamma $ we denote by $f\cdot g$ the \emph{pointwise} product of $f$ and $g$, i.e. $(f\cdot g)_s=f_sg_s$ for all $s\in \Gamma $ (this is \emph{not} the usual multiplication in $\mathbb{C}\Gamma $, which we write as $(f,g)\mapsto fg$). We embed $\Gamma $ in $\mathbb{C}\Gamma $ by identifying each $s\in \Gamma $ with the element in $\mathbb{C}\Gamma $ given by
	\begin{displaymath}
s_t=
	\begin{cases}
1&\textup{if}\enspace t=s,
	\\
0&\textup{otherwise}.
	\end{cases}
	\end{displaymath}
In this notation every $f\in \mathbb{C}\Gamma $ is a finite sum of the form $f=\sum_{s\in \Gamma }f_ss$. For every $f=\sum_{s\in \Gamma }f_ss\in \mathbb{C}\Gamma $ we denote by $f^*=\sum_{s\in \Gamma }\overline{f_{s^{-1}}}s\in \mathbb{C}\Gamma $ the \emph{adjoint} of $f$ and observe that $(ff^*)_{\matheur{e}_\Gamma }=(f^*f)_{\matheur{e}_\Gamma }= \sum_{s\in \Gamma }|f_s|^2 >0$ whenever $f$ is nonzero. The \emph{real} and \emph{integer group rings} $\mathbb{Z}\Gamma \subset \mathbb{R}\Gamma \subset \mathbb{C}\Gamma $ are the subrings of all real- and integer-valued elements of $\mathbb{C}\Gamma $, and we put $\mathbb{R}_{\ge0}=\{r\in \mathbb{R}\colon r\ge0\}$, $\mathbb{Z}_{\ge0}=\mathbb{Z}\cap\mathbb{R}_{\ge0}$ and write $\mathbb{Z}_{\ge0}\Gamma \subset \mathbb{R}_{\ge0}\Gamma $ for the sets of nonnegative elements in $\mathbb{Z}\Gamma $ and $\mathbb{R}\Gamma $.

\smallskip For $f\in \mathbb{R}_{\ge0}\Gamma $ and $\mathsf{A},F\in \mathscr{F}(\Gamma )$ we set
	\begin{equation}
	\label{eq:permanent}
\per _{\mathsf{A},F}(f)=\sum_{x\in X_{\mathsf{A},F}}\prod _{s\in F}f_{x_s}\quad\textup{and}\quad \iper_{\mathsf{A},F}(f)=\sum_{x\in X_{\mathsf{A},F}^\iota }\prod _{s\in F}f_{x_s}.
	\end{equation}

	\begin{lemma}
	\label{L-perm subadditivity}
Let $f, g\in \mathbb{R}_{\ge0}\Gamma $ and $\mathsf{A}\in \mathscr{F}(\Gamma )$.
	\begin{enumerate}
	\item
For any $s\in \Gamma $ and $F\in \mathscr{F}(\Gamma )$, $\per _{\mathsf{A}s,F}(fs)=\per _{\mathsf{A},F}(f)= \per _{s\mathsf{A},Fs^{-1}}(sf)$ and\linebreak $\iper_{\mathsf{A}s,F} (fs)\linebreak[0]=\iper_{\mathsf{A},F} (f)=\iper_{s\mathsf{A},Fs^{-1}} (sf)$;
	\item
For any $s\in \Gamma $ and $F\in \mathscr{F}(\Gamma )$, $\per _{\mathsf{A},sF}(f)=\per _{\mathsf{A},F}(f)$ and $\iper_{\mathsf{A},sF} (f)=\iper_{\mathsf{A},F} (f)$;
	\item
Assume that $f$ is nonzero and put $\varkappa =\min_{\{s\in \Gamma \colon f_s>0\}}f_s>0$. For any $F_1, F_2\in \mathscr{F}(\Gamma )$ one has
	\begin{gather*}
\smash[b]{\frac{\per _{\mathsf{A},F_1\cup F_2}(f)}{\varkappa ^{|F_1\cup F_2|}}\le \frac{\per _{\mathsf{A},F_1}(f)}{\varkappa ^{|F_1|}}\cdot \frac{\per _{\mathsf{A},F_2}(f)}{\varkappa ^{|F_2|}},}
	\\
\intertext{and}
\smash[t]{\frac{\iper_{\mathsf{A},F_1\cup F_2} (f)}{\varkappa ^{|F_1\cup F_2|}}\le \frac{\iper_{\mathsf{A},F_1} (f)} {\varkappa ^{|F_1|}}\cdot \frac{\iper_{\mathsf{A},F_2} (f)} {\varkappa ^{|F_2|}}.}
	\end{gather*}
	\item
For any $F\in \mathscr{F}(\Gamma )$, $\per _{\mathsf{A},F}(f\cdot g)\le \per _{\mathsf{A},F}(f)\cdot \per _{\mathsf{A},F}(g)$ and $\iper_{\mathsf{A},F} (f\cdot g)\le \iper_{\mathsf{A},F} (f)\cdot \iper_{\mathsf{A},F} (g)$.
	\end{enumerate}
	\end{lemma}

	\begin{proof} (1) and (2) are obvious.

(3): We prove the first inequality. The second inequality is proved similarly. Let $x\in X_{\mathsf{A},F_1\cup F_2}$. Then $\pi _{F_j}(x) \in X_{\mathsf{A},F_j}$ for $j=1,2$. If $f_{x_s}\neq 0$ for all $s\in F_1\cap F_2$, then
	\begin{displaymath}
\varkappa ^{|F_1\cap F_2|}\prod _{s\in F_1\cup F_2}f_{x_s}\le \Biggl(\prod _{s\in F_1\cap F_2}f_{x_s}\Biggr)\cdot \Biggl(\prod _{s\in F_1\cup F_2}f_{x_s}\Biggr)=\Biggl(\prod _{s\in F_1}f_{x_s}\Biggr)\cdot \Biggl(\prod _{s\in F_2}f_{x_s}\Biggr),
	\end{displaymath}
and hence
	\begin{equation}
	\label{E-product}
\smash[t]{\frac{\prod _{s\in F_1\cup F_2}f_{x_s}}{\varkappa ^{|F_1\cup F_2|}}\le \frac{\prod _{s\in F_1}f_{x_s}}{\varkappa ^{|F_1|}}\cdot \frac{\prod _{s\in F_2}f_{x_s}}{\varkappa ^{|F_2|}}.}
	\end{equation}
If $f_{x_s}=0$ for some $s\in F_1\cap F_2$, then \eqref{E-product} holds trivially. Now we have
	\begin{align*}
\frac{\per _{\mathsf{A},F_1\cup F_2}(f)}{\varkappa ^{|F_1\cup F_2|}}&=\sum_{x\in X_{\mathsf{A},F_1\cup F_2}}\frac{\prod _{s\in F_1\cup F_2}f_{x_s}}{\varkappa ^{|F_1\cup F_2|}} \le \sum_{x\in X_{\mathsf{A},F_1\cup F_2}}\frac{\prod _{s\in F_1}f_{x_s}}{\varkappa ^{|F_1|}}\cdot \frac{\prod _{s\in F_2}f_{x_s}}{\varkappa ^{|F_2|}}
	\\
&\le \frac{\per _{\mathsf{A},F_1}(f)} {\varkappa ^{|F_1|}}\cdot \frac{\per _{\mathsf{A},F_2}(f)} {\varkappa ^{|F_2|}}.
	\end{align*}

(4): We have
	\begin{align*}
\per _{\mathsf{A},F}(f\cdot g)&=\sum_{x\in X_{\mathsf{A},F}}\prod _{s\in F}(f\cdot g)_{x_s}
=\sum_{x\in X_{\mathsf{A},F}}\prod _{s\in F}f_{x_s}\prod _{t\in F}g_{x_t}
	\\
&\le \Biggl(\sum_{x\in X_{\mathsf{A},F}}\prod _{s\in F}f_{x_s}\Biggr)\cdot \Biggl(\sum_{x\in X_{\mathsf{A},F}}\prod _{t\in F}g_{x_t}\Biggr)
	\\
&=\per _{\mathsf{A},F}(f)\cdot \per _{\mathsf{A},F}(g).
	\end{align*}
The proof for $\iper_{\mathsf{A},F} (f\cdot g)\le \iper_{\mathsf{A},F} (f)\cdot \iper_{\mathsf{A},F}(g)$ is similar.
	\end{proof}

Now assume that $\Gamma $ is amenable. As in \eqref{eq:limit} -- \eqref{eq:htop} we use right F{\o}lner sets.

	\begin{lemma}
	\label{L-limit}
The limit $\lim_F\frac{1}{|F|}\log \per _{\mathsf{A},F}(f)$ exists in $[-\infty, +\infty)$, i.e., there is some $C\in [-\infty, \linebreak[0]+\infty)$ such that for any neighborhood $U$ of $C$ in $[-\infty, +\infty)$ there are $K\in \mathscr{F}(\Gamma )$ and $\delta >0$ such that $\frac{1}{|F|}\log \per _{\mathsf{A},F}(f)\in U$ for all $F\in \mathscr{F}(\Gamma )$ with $|FK\vartriangle F|<\delta |F|$. Also the limit $\lim_F\frac{1}{|F|}\log \iper_{\mathsf{A},F} (f)$ exists in $[-\infty, +\infty)$.
	\end{lemma}

	\begin{proof}
We may assume that $f\neq 0$ and put $\varkappa =\min_{\{s\in \Gamma \colon f_s>0\}}f_s>0$. From Lemma~\ref{L-perm subadditivity} (2) -- (3) and the Ornstein-Weiss lemma \cite{KL16}*{Theorem 4.38} we know that the limits
	\begin{displaymath}
\smash[b]{\lim_F\frac{1}{|F|}\log \frac{\per _{\mathsf{A},F}(f)}{\varkappa ^{|F|}} \quad \textup{and} \quad \lim_F\frac{1}{|F|}\log \frac{\iper_{\mathsf{A},F} (f)} {\varkappa ^{|F|}}}
	\end{displaymath}
exist. Then
	\begin{displaymath}
\smash{\lim_F\frac{1}{|F|}\log \per _{\mathsf{A},F}(f)=\lim_F\frac{1}{|F|}\log \frac{\per _{\mathsf{A},F}(f)}{\varkappa ^{|F|}}+\log \varkappa }
	\end{displaymath}
and
	\begin{displaymath}
\smash[t]{\lim_F\frac{1}{|F|}\log \iper_{\mathsf{A},F} (f)= \lim_F\frac{1}{|F|}\log \frac{\iper_{\mathsf{A},F} (f)} {\varkappa^{|F|}}+\log \varkappa .}\tag*{\qedsymbol}
	\end{displaymath}
	\renewcommand{\qedsymbol}{}
	\vspace{-\baselineskip}
	\end{proof}

	\begin{lemma}
	\label{L-indep of A}
When $\supp(f)=\{s\in \Gamma \colon f_s>0\}\subseteq \mathsf{A}$, the limits $\lim_F\frac{1}{|F|}\log \per _{\mathsf{A},F}(f)$ and $\lim_F\frac{1}{|F|}\log \iper_{\mathsf{A},F} (f)$ do not depend on the choice of $\mathsf{A}$.
	\end{lemma}

	\begin{proof}
The statement about $\lim_F\frac{1}{|F|}\log \iper_{\mathsf{A},F} (f)$ is obvious. We shall prove the statement about $\lim_F\frac{1}{|F|} \log \per _{\mathsf{A},F}(f)$. The case $f=0$ is trivial. Thus we may assume that $f\ne0$. From Lemma \ref{L-perm subadditivity} (1) we know that $\lim_F \frac{1}{|F|} \log \per _{\mathsf{A},F}(f) = \lim_F \frac{1}{|F|} \log \per _{\mathsf{A}s,F}(fs)$ for every $s\in \mathsf{A}^{-1}$. Thus we may assume that $\matheur{e}_\Gamma \in \supp(f)$. It suffices to show that
	\begin{displaymath}
\lim_F\frac{1}{|F|}\log \per _{\mathsf{A}_1,F}(f)= \lim_F\frac{1}{|F|}\log \per _{\mathsf{A}_2,F}(f)
	\end{displaymath}
for $\supp(f)\subseteq \mathsf{A}_1\subseteq \mathsf{A}_2$. Put $L_i=\lim_F\frac{1}{|F|}\log \per _{\mathsf{A}_i,F}(f)$ for $i=1, 2$. Clearly $L_1\le L_2$. Thus it suffices to show that $L_1\ge L_2$. We may assume that $L_2>-\infty$.

Let $\varepsilon >0$. Then there are some $K\in \mathscr{F}(\Gamma )$ and $\delta >0$ such that $\frac{1}{|F'|}\log \per _{\mathsf{A}_2,F'}(f)\ge L_2-\varepsilon$ for every $F'\in \mathscr{F}(\Gamma )$ satisfying that $|F'K\smallsetminus F'|<\delta |F'|$.

Let $F\in \mathscr{F}(\Gamma )$ be such that $|F\mathsf{A}_2^{-1}K\smallsetminus F|<\delta |F|$. Put $F'=F\mathsf{A}_2^{-1}\supseteq F$. Then $|F'K\smallsetminus F'|\le |F\mathsf{A}_2^{-1}K\smallsetminus F|<\delta |F|\le \delta |F'|$, and hence $\frac{1}{|F'|}\log \per _{\mathsf{A}_2,F'}(f)\ge L_2-\varepsilon$.

If we define $X_{\mathsf{A}_1,F',F}$ as in \eqref{eq:XAF}, then
	\begin{displaymath}
\per _{\mathsf{A}_2,F'}(f)\le \sum_{x\in X_{\mathsf{A}_1,F',F}}\,\prod _{s\in F'}f_{x_s}.
	\end{displaymath}
Note that for each $x\in X_{\mathsf{A}_1,F', F}$ the restriction of $x$ to $F$ lies in $X_{\mathsf{A}_1,F}$, and that $\prod _{s\in F'}f_{x_s}\le \|f\|_\infty^{|F'\smallsetminus F|}\prod _{s\in F}f_{x_s}$. Thus
	\begin{align*}
\sum_{x\in X_{\mathsf{A}_1,F', F}}\,\prod _{s\in F'}f_{x_s}&\le \|f\|_\infty^{|F'\smallsetminus F|}\sum_{x\in X_{\mathsf{A}_1,F', F}}\,\prod _{s\in F}f_{x_s}\le \|f\|_\infty^{|F'\smallsetminus F|}|\mathsf{A}_1|^{|F'\smallsetminus F|} \sum_{x\in X_{\mathsf{A}_1,F}}\,\prod _{s\in F}f_{x_s}
	\\
&= (\|f\|_\infty |\mathsf{A}_1|)^{|F'\smallsetminus F|}\,\per _{\mathsf{A}_1,F}(f).
	\end{align*}
Therefore
	\begin{displaymath}
e^{|F'|(L_2-\varepsilon)}\le \per _{\mathsf{A}_2,F'}(f)\le \sum_{x\in X_{\mathsf{A}_1,F', F}}\,\prod _{s\in F'}f_{x_s}\le (\|f\|_\infty |\mathsf{A}_1|)^{|F'\smallsetminus F|}\,\per _{\mathsf{A}_1,F}(f).
	\end{displaymath}
It follows that
	\begin{displaymath}
L_1=\lim_F\frac{1}{|F|}\log \per _{\mathsf{A}_1,F}(f)\ge \lim_F\frac{|F'|(L_2-\varepsilon)-|F'\smallsetminus F|\log (\|f\|_\infty|\mathsf{A}_1|)}{|F|}=L_2-\varepsilon.
	\end{displaymath}
Letting $\varepsilon\to 0$ we obtain that $L_1\ge L_2$, as claimed.
	\end{proof}

	\begin{definition}
	\label{D-permanent}
Let $\Gamma $ be a countable discrete group, $0 \ne f\in \mathbb{R}_{\ge0}\Gamma $, and let $\mathsf{A}=\textup{supp}(f)\in \mathscr{F}(\Gamma )$.
	\begin{enumerate}
	\item
If $\Gamma $ is finite we denote by
	\begin{displaymath}
\per (f)= \frac{1}{|\Gamma |}\sum_{x\in X_{\mathsf{A}}}\prod _{s\in \Gamma }f_{x_s}\quad\textup{and}\quad \iper (f)=\frac{1}{|\Gamma |}\sum_{x\in X_{\mathsf{A}}^\iota }\, \prod _{s\in \Gamma }f_{x_s}
	\end{displaymath}
the \emph{permanent} and the \emph{injective permanent} of $f$.
\smallskip
	\item
If $\Gamma $ is infinite and amenable we call
	\begin{displaymath}
\per(f) = \lim_F\frac{1}{|F|}\log \Biggl(\sum_{x\in X_{\mathsf{A},F}}\,\prod _{s\in F}f_{x_s}\Biggr)\quad\textup{and}\quad \iper(f)=\lim_F\frac{1}{|F|}\log \Biggl(\sum_{x\in X_{\mathsf{A},F}^\iota }\prod _{s\in F}f_{x_s}\Biggr)
	\end{displaymath}
the \emph{permanent} and the \emph{injective permanent} of $f$, where the limits are taken along right F{\o}lner sequences.
	\end{enumerate}
	\end{definition}

	\begin{remark}
	\label{R-real}
When $f\in \mathbb{R}_{\ge0}\Gamma $ is nonzero, setting $\varkappa=\max_{s\in \Gamma }f_s$ in \eqref{eq:permanent}, we have that $\per _{\mathsf{A},F}(f)\linebreak[0]\ge \varkappa^{|F|}$ for every $F\in \mathscr{F}(\Gamma )$, so that $\iper(f)\ge \per (f)\ge \log \varkappa >-\infty $. Also, for the function $\log f\colon \mathsf{A}\rightarrow \mathbb{R}$, where $\mathsf{A}=\supp(f)$, by Theorem~\ref{T-local subshift} we have $\per(f)=P(X_\mathsf{A}, \log f)$ and $\iper(f)=P(X_\mathsf{A}^\iota, \log f)$.
	\end{remark}

	\begin{question}
Is there any reasonable way to define $\per (f)$ for \emph{all} (i.e. not necessarily amenable) countably infinite groups $\Gamma $?
	\end{question}

	\begin{remark}
	\label{R-infinite support}
For $f\in \mathbb{R}_{\ge0}^\Gamma $ with infinite support, one can define $\per (f)$ as
	\begin{displaymath}
\sup_{\substack{\mathsf{A}\in \mathscr{F}(\Gamma )
	\\
\mathsf{A}\subset \supp(f)}} \lim_{F }\frac{1}{|F|}\log \per _{\mathsf{A},F}(f)= \lim_{\substack{\mathsf{A}\nearrow\supp(f) ,\,\mathsf{A}\in \mathscr{F}(\Gamma )
	\\
\mathsf{A}\subset \supp(f)}}\lim_{F }\frac{1}{|F|}\log \per _{\mathsf{A},F}(f),
	\end{displaymath}
since $\lim_F\frac{1}{|F|}\log \per _{\mathsf{A},F}(f)$ increases in $\mathsf{A}$.
	\end{remark}

	\begin{remark}
	\label{R-permanent and entropy}
For any nonempty finite subset $\mathsf{A}$ of $\Gamma $, taking $f\in \mathbb{Z}\Gamma $ to be the indicator function of $\mathsf{A}$, we get from Proposition~\ref{P-entropy in local} that $\htop(X_\mathsf{A})=\per (f)$ and $\htop(X_\mathsf{A}^\iota )=\iper(f)$.
	\end{remark}

We list some basic properties of $\per (f)$ and $\iper(f)$. Item (5) is the analogue of the Chollet conjecture for permanents of positive semidefinite matrices (cf. \cite{Chollet}, \cite{Zhang}*{Section 6}).

	\begin{proposition}
	\label{P-basic}
The following hold:
	\begin{enumerate}
	\item
For any $f\in \mathbb{R}_{\ge 0}\Gamma $ and $s\in \Gamma $, we have $\per (fs)=\per (f)=\per (sf)$ and $\iper(fs)=\iper(f)=\iper(sf)$.
	\item
For any $f, g\in \mathbb{R}_{\ge 0}\Gamma $ with $f\le g$, we have $\per (f)\le \per (g)$ and $\iper(f)\le \iper(g)$.
	\item
For any $f\in \mathbb{R}_{\ge 0}\Gamma $ and $C>0$, we have $\per (Cf)=\per (f)+\log C$ and $\iper(Cf)=\iper(f)+\log C$.
	\item
For any $\mathsf{A}\in \mathscr{F}(\Gamma)$, the maps $f\mapsto \per (f)$ and $f\mapsto \iper(f)$ are continuous on $\mathbb{R}_{\ge 0}^\mathsf{A}$.
	\item
For any $f, g\in \mathbb{R}_{\ge 0}\Gamma $, we have $\per (f\cdot g)\le \per (f)+\per (g)$ and $\iper(f\cdot g)\le \iper(f)+\iper(g)$ .
	\item
Let $G$ be an amenable group containing $\Gamma $, and let $f\in \mathbb{R}_{\ge 0}\Gamma $. Denote by $\per _\Gamma(f)$ {\textup(}resp. $\iper_\Gamma(f)${\textup)} and $\per _G(f)$ {\textup(}resp. $\iper_G(f)${\textup)} the {\textup(}resp. injective{\textup)} permanents of $f$ as elements of $\mathbb{R}_{\ge 0}\Gamma $ and $\mathbb{R}_{\ge 0}G$, respectively. Then $\per _\Gamma(f)=\per _G(f)$ and $\iper_\Gamma(f)=\iper_G(f)$.
	\end{enumerate}
	\end{proposition}

	\begin{proof}
(1) follows from (1) of Lemma~\ref{L-perm subadditivity}.

(2) and (3) are obvious.

For (4), we shall prove the statement about $\per (f)$. The proof for the statement about $\iper(f)$ is similar.
Fix $\mathsf{A}\in \mathscr{F}(\Gamma)$. It follows from (2) and (3) that the map $f\mapsto \per (f)$ is lower semicontinuous on $\mathbb{R}_{\ge 0}^\mathsf{A}$ and that when $\supp(g)=\mathsf{A}$, the map $f\mapsto \per (f)$ on $\mathbb{R}_{\ge 0}^\mathsf{A}$ is continuous at $g$. Thus to prove (4) it suffices to show that when $\supp(g)\subsetneq \mathsf{A}$, the map $f\mapsto \per (f)$ on $\mathbb{R}_{\ge 0}^\mathsf{A}$ is upper semicontinuous at $g$. If $\Gamma $ is finite this is obvious. Thus we may assume that $\Gamma $ is infinite. By (3) we may assume that $\|g\|_\infty=1$.

Let $D>\per (g)$.
Take $\beta>0$ with $\per (g)<D-4\beta$. Take $K\in \mathscr{F}(\Gamma)$ and $\delta>0$ such that for any $W\in \mathscr{F}(\Gamma)$ with $|WK\smallsetminus W|<\delta|W|$ we have
$\frac{1}{|W|}\log \per _{\mathsf{A}, W}(g)\le D-3\beta$. By Stirling's approximation formula there is some $0<\kappa<1/2$ such that for any nonempty finite set $F$ the number of subsets $W$ of $F$ with $|W|\ge (1-\kappa)|F|$ is at most $e^{\beta|F|}$ (see for example \cite{CCL}*{Appendix A}). Take $0<\eta<\kappa$ such that
$2\eta/(1-\eta)<\delta $ and $|\mathsf{A}|^\eta<e^\beta $.

Take $0<\varepsilon<1$ such that $2|\mathsf{A}|\varepsilon^\eta<e^D$ and $\log(1+\varepsilon)<\beta$. Denote by $U$ the set of $f\in \mathbb{R}_{\ge 0}^\mathsf{A}$ satisfying that $f_s<(1+\varepsilon)g_s$ for all $s\in \supp(g)$ and $f_s<\varepsilon$ for all $s\in \mathsf{A}\smallsetminus \supp(g)$. Then $U$ is an open neighborhood of $g$ in $\mathbb{R}_{\ge 0}^\mathsf{A}$. Now it suffices to show $\per (f)<D$ for every $f\in U$.

Let $f\in U$. Then $\|f\|_\infty\le 2$. Denote by $h$ the element of $\mathbb{R}_{\ge 0}\Gamma $ which coincides with $f$ on $\supp(g)$ and vanishes on $\Gamma\smallsetminus \supp(g)$.
For $F\in \mathscr{F}(\Gamma)$ denote by $X_{\mathsf{A}, F, \eta}$ the set of $x\in X_{\mathsf{A}, F}$ satisfying that $|x^{-1}(\supp(g))|\ge (1-\eta) |F|$.
For each $x\in X_{\mathsf{A}, F}\smallsetminus X_{\mathsf{A}, F, \eta}$ we have $\prod _{t\in F}f_{x_t}\le 2^{|F|}\varepsilon^{\eta |F|}$. Thus
	\begin{align*}
\sum_{x\in X_{\mathsf{A}, F}\smallsetminus X_{\mathsf{A}, F, \eta}}\prod _{t\in F}f_{x_t}\le \sum_{x\in X_{\mathsf{A}, F}\smallsetminus X_{\mathsf{A}, F, \eta}}2^{|F|}\varepsilon^{\eta |F|}\le |\mathsf{A}|^{|F|}2^{|F|}\varepsilon^{\eta |F|}<e^{D|F|}.
	\end{align*}
Let $F\in \mathscr{F}(\Gamma)$ such that $|FK\smallsetminus F|<\eta|F|$.
For each $W\subseteq F$ with $|W|\ge (1-\eta)|F|$ we have
	\begin{displaymath}
|WK\smallsetminus W|\le |FK\smallsetminus W|\le |FK\smallsetminus F|+|F\smallsetminus W|< \eta|F|+\eta |F|\le \frac{2\eta}{1-\eta}|W|<\delta |W|,
	\end{displaymath}
and hence
	\begin{align*}
\frac{1}{|W|}\log \per _{\mathsf{A}, W}(h)&\le \frac{1}{|W|}\log \per _{\mathsf{A}, W}((1+\varepsilon)g)
	\\
&=\log(1+\varepsilon)+\frac{1}{|W|}\log \per _{\mathsf{A}, W}(g) \le \beta+D-3\beta=D-2\beta.
	\end{align*}
For each $x\in X_{\mathsf{A}, F, \eta}$, putting $F_x=x^{-1}(\supp(g))\subseteq F$ we have $|F_x|\ge (1-\eta)|F|$ and
that the restriction of $x$ on $F_x$ lies in $X_{\mathsf{A}, F_x}$. Thus
	\begin{align*}
\sum_{x\in X_{\mathsf{A}, F, \eta}}\prod _{t\in F}f_{x_t}&=\sum_{W\subseteq F, |W|\ge (1-\eta)|F|}\sum_{x\in X_{\mathsf{A}, F, \eta}, F_x=W} \prod _{t\in F}f_{x_t}
	\\
&\le \sum_{W\subseteq F, |W|\ge (1-\eta)|F|}\sum_{x\in X_{\mathsf{A}, F, \eta}, F_x=W} \prod _{t\in W}f_{x_t}
	\\
&\le \sum_{W\subseteq F, |W|\ge (1-\eta)|F|}|\mathsf{A}|^{\eta |F|}\per _{\mathsf{A}, W}(h)
	\\
&\le \sum_{W\subseteq F, |W|\ge (1-\eta)|F|}e^{\beta |F|}e^{|W|(D-2\beta)}
	\\
&\le e^{\beta|F|}e^{\beta |F|}e^{|F|(D-2\beta)}= e^{|F|D}.
	\end{align*}
Then
	\begin{displaymath}
\per _{\mathsf{A}, F}(f)=\sum_{x\in X_{\mathsf{A}, F}\smallsetminus X_{\mathsf{A}, F, \eta}}\prod _{t\in F}f_{x_t}+\sum_{x\in X_{\mathsf{A}, F, \eta}}\prod _{t\in F}f_{x_t} \le e^{F|D}+ e^{|F|D}=2e^{|F|D}.
	\end{displaymath}
Since $\Gamma $ is infinite, it follows that
	\begin{displaymath}
\per (f)=\lim_F\frac{1}{|F|}\log \per _{\mathsf{A}, F}(f)\le D.
	\end{displaymath}

(5) follows from (4) of Lemma~\ref{L-perm subadditivity}.

(6): We prove $\per _\Gamma(f)=\per _G(f)$. The proof for $\iper_\Gamma(f)=\iper_G(f)$ is similar. We may assume that $f\neq 0$. Put $\mathsf{A}=\supp(f)\in \mathscr{F}(\Gamma)$.
Let $U_\Gamma $ and $U_G$ be neighborhoods of $\per _\Gamma(f)$ and $\per _G(f)$ in $\mathbb{R}$ respectively.
Then it suffices to show that $U_\Gamma\cap U_G\neq \varnothing $.

Set $\varkappa=\max_{s\in \mathsf{A}}f_s$. Then there are a convex neighborhood $V_\Gamma $ of $\per _\Gamma(f)$ in $\mathbb{R}$ and $0<\delta<1$ such that for any $0\le \eta\le \delta$, $a\in [\log \varkappa, \log \varkappa+\log |\mathsf{A}|]$, and $b\in V_\Gamma $, one has $\eta a+(1-\eta)b\in U_\Gamma $.

Take $K_\Gamma\in \mathscr{F}(\Gamma )$ containing $\matheur{e}_\Gamma $, $K_G\in \mathscr{F}(G)$, and $\varepsilon>0$ such that for any $F_\Gamma\in \mathscr{F}(\Gamma)$ with $|F_\Gamma K_\Gamma|<(1+\varepsilon)|F_\Gamma|$ one has $\frac{1}{|F_\Gamma|}\log \per _{\mathsf{A}, F_\Gamma}(f)\in V_\Gamma $, and that for any $F_G\in \mathscr{F}(G)$ with $|F_G K_G|<(1+\varepsilon)|F_G|$ one has $\frac{1}{|F_G|}\log \per _{\mathsf{A}, F_G}(f)\in U_G$.

Take $F\in \mathscr{F}(G)$ with $|F(K_\Gamma\cup K_G)|<(1+\delta \varepsilon)|F|$. Then there are some $W\in \mathscr{F}(G)$ and $F_s\in \mathscr{F}(\Gamma)$ for each $s\in W$ such that $s\Gamma\neq t\Gamma $ for all distinct $s, t\in W$ and $F$ is the disjoint union of $sF_s$ for $s\in W$. Denote by $W'$ the set of $s\in W$ satisfying $|F_sK_\Gamma|<(1+\varepsilon)|F_s|$.
Put $F_{W'}=\bigcup_{s\in W'}sF_s$ and $F_{W\smallsetminus W'}=\bigcup_{s\in W\smallsetminus W'}sF_s$.
Then
	\begin{align*}
\delta \varepsilon |F|&\ge |FK_\Gamma|-|F|=\sum_{s\in W}(|sF_sK_\Gamma|-|sF_s|)\ge \sum_{s\in W\smallsetminus W'}(|sF_sK_\Gamma|-|sF_s|)\ge \varepsilon|F_{W\smallsetminus W'}|,
	\end{align*}
and hence $|F_{W\smallsetminus W'}|\le \delta|F|$.

Note that $X_{\mathsf{A}, F}$ can be identified with $\prod _{s\in W}X_{\mathsf{A}, F_s}$ naturally. Thus
	\begin{align*}
\per _{\mathsf{A}, F}(f)=\sum_{x\in X_{\mathsf{A}, F}}\prod _{t\in F}f_{x_t}=\prod _{s\in W}\sum_{x\in X_{\mathsf{A}, F_s}}\prod _{t\in F_s}f_{x_t}=\prod _{s\in W}\per _{\mathsf{A}, F_s}(f),
	\end{align*}
and hence
	\begin{align*}
& \frac{1}{|F|}\log \per _{\mathsf{A}, F}(f)
	\\
&=\frac{|F_{W'}|}{|F|}\sum_{s\in W'}\frac{|F_s|}{|F_{W'}|}\cdot \frac{1}{|F_s|}\log \per _{\mathsf{A}, F_s}(f)+\frac{|F_{W\smallsetminus W'}|}{|F|}\sum_{s\in W\smallsetminus W'}\frac{|F_s|}{|F_{W\smallsetminus W'}|}\cdot \frac{1}{|F_s|}\log \per _{\mathsf{A}, F_s}(f).
	\end{align*}
For each $s\in W$, $\varkappa^{|F_s|}\le \per _{\mathsf{A}, F_s}(f)\le |\mathsf{A}|^{|F_s|}\varkappa^{|F_s|}$, and hence $\log \varkappa\le \frac{1}{|F_s|}\log \per _{\mathsf{A}, F_s}(f)\le \log |\mathsf{A}|+\log |\varkappa|$. It follows that
	\begin{displaymath}
\sum_{s\in W\smallsetminus W'}\frac{|F_s|}{|F_{W\smallsetminus W'}|}\cdot \frac{1}{|F_s|}\log \per _{\mathsf{A}, F_s}(f)\in [\log \varkappa, \log \varkappa+\log |\mathsf{A}|].
	\end{displaymath}
For each $s\in W'$, we have $ \frac{1}{|F_s|}\log \per _{\mathsf{A}, F_s}(f)\in V_\Gamma $. Since $V_\Gamma $ is convex, we have $\sum_{s\in W'}\frac{|F_s|}{|F_{W'}|}\cdot \frac{1}{|F_s|}\log \per _{\mathsf{A}, F_s}(f)\in V_\Gamma $. From our choice of $V_\Gamma $ and $\delta$ we get $\per _{\mathsf{A}, F}(f)\in U_\Gamma $. Since $\per _{\mathsf{A}, F}(f)\in U_G$, we conclude that $U_\Gamma\cap U_G\neq \varnothing $.
	\end{proof}

	\begin{proposition}
	\label{P-per equal iper}
For any $f\in \mathbb{R}_{\ge0}\Gamma $ and $\mathsf{A}\in \mathscr{F}(\Gamma )$ with $\supp(f)\subseteq \mathsf{A}$, we have
	\begin{displaymath}
\per (f)=\iper(f)=\inf_{F\in \mathscr{F}(\Gamma )}\frac{1}{|F|}\log \per_{\mathsf{A},F}(f)=
\inf_{F\in \mathscr{F}(\Gamma )}\frac{1}{|F|}\log \iper_{\mathsf{A},F}(f).
	\end{displaymath}
	\end{proposition}

	\begin{proof}
We may assume that $f\neq 0$.

Consider first the case $\mathsf{A}=\supp(f)$. We have the function $\log f\colon \mathsf{A}\rightarrow \mathbb{R}$. By Lemma~\ref{L-inv measure are on perm} and the variational principle \cite{MO}*{Theorem 5.2.7} we have $P(X_\mathsf{A}, \log f)=P(X_\mathsf{A}^\iota, \log f)$. Then from Remark~\ref{R-real} we have
	\begin{displaymath}
\per(f)=P(X_\mathsf{A}, \log f)=P(X_\mathsf{A}^\iota, \log f)=\iper(f).
	\end{displaymath}
According to Theorem~\ref{T-local subshift} we have
	\begin{displaymath}
\per(f)=P(X_\mathsf{A}, \log f)=\inf_{F\in \mathscr{F}(\Gamma )}\frac{1}{|F|}\log \per_{\mathsf{A},F}(f),
	\end{displaymath}
and
	\begin{displaymath}
\iper(f)=P(X_\mathsf{A}^\iota, \log f)=\inf_{F\in \mathscr{F}(\Gamma )}\frac{1}{|F|}\log \per_{\mathsf{A},F}^\iota(f).
	\end{displaymath}

Now consider the general case $\supp(f)\subseteq \mathsf{A}\in \mathscr{F}(\Gamma )$. Take a decreasing sequence $(f^{(n)})_{n\in \mathbb{N}}$ in $\mathbb{R}_{>0}^\mathsf{A}$ with limit $f$. Then by Proposition~\ref{P-basic} we have that $\per (f)=\inf_{n\in \mathbb{N}}\per (f^{(n)})$ and $\iper(f)=\inf_{n\in \mathbb{N}}\iper(f^{(n)})$. Therefore
	\begin{align*}
\per (f)&=\smash[t]{\inf_{n\in \mathbb{N}}\per(f^{(n)})
=\inf_{n\in \mathbb{N}}\inf_{F\in \mathscr{F}(\Gamma )}\frac{1}{|F|}\log \per_{\mathsf{A},F}(f^{(n)})}
	\\
&=\inf_{F\in \mathscr{F}(\Gamma )}\inf_{n\in \mathbb{N}}\frac{1}{|F|}\log \per_{\mathsf{A},F}(f^{(n)})
=\inf_{F\in \mathscr{F}(\Gamma )}\frac{1}{|F|}\log \per_{\mathsf{A},F}(f).
	\end{align*}
The proof for $\iper (f)=\inf_{F\in \mathscr{F}(\Gamma )}\frac{1}{|F|}\log \per_{\mathsf{A},F}^\iota (f)$ is similar.
	\end{proof}

	\begin{proposition}
	\label{P-adjoint and product}
The following hold:
	\begin{enumerate}
	\item
For any $f\in \mathbb{R}_{\ge0}\Gamma $, we have $\per (f)=\per (f^*)$;
	\item
For any $f, g\in \mathbb{R}_{\ge0}\Gamma $, we have $\per (fg)\ge \per (f)+\per (g)$.
	\end{enumerate}
	\end{proposition}

	\begin{proof}
(1). By Proposition~\ref{P-per equal iper} it suffices to show that $\iper(f)=\iper(f^*)$. By symmetry it suffices to show that $\iper(f)\le \iper(f^*)$. We may assume that $f\neq 0$. Put $\mathsf{A}=\supp(f)\in \mathscr{F}(\Gamma )$. Then $\mathsf{A}^{-1}=\supp(f^*)$.

Let $C>\iper(f^*)$. Take $\varepsilon>0$ with $C>\iper(f^*)+2\varepsilon$.

Take $K\in \mathscr{F}(\Gamma )$ containing $\matheur{e}_\Gamma $ and $\delta >0$ such that for any $F'\in \mathscr{F}(\Gamma )$ with $|F'K|<(1+\delta ) |F'|$ we have $\frac{1}{|F'|}\log \iper_{\mathsf{A}^{-1},F'}(f^*)\le C-2\varepsilon$.

By Stirling's approximation formula (see for example \cite{CCL}*{Appendix A}) there is some $0<\delta '<1/2$ such that for each nonempty finite set $W$ the number of subsets of $W$ with cardinality at most $\delta ' |W|$ is at most $e^{\varepsilon |W|}$.

Let $F\in \mathscr{F}(\Gamma )$ with $|F\mathsf{A}K|<(1+\min(\delta ', \delta ))|F|$.
Denote by $\Theta ^\iota (\mathsf{A},F)$ the set of $F'\subseteq F\mathsf{A}$ satisfying $|F'|=|F|$.
For each $F'\in \Theta ^\iota (\mathsf{A},F)$, we have
	\begin{displaymath}
|F'K|\le |F\mathsf{A}K|<(1+\delta )|F|=(1+\delta )|F'|,
	\end{displaymath}
and hence $\iper_{\mathsf{A}^{-1}, F'}(f^*)\le e^{|F'|(C-2\varepsilon)}$.
For each $F'\in \Theta ^\iota (\mathsf{A},F)$, we also have
	\begin{displaymath}
|F'|=|F|\ge \frac{1}{1+\delta '}|F\mathsf{A}K|\ge \frac{1}{1+\delta '}|F\mathsf{A}|\ge (1-\delta ')|F\mathsf{A}|.
	\end{displaymath}
Thus $|\Theta ^\iota (\mathsf{A},F)|\le e^{\varepsilon |F\mathsf{A}|}\le e^{2\varepsilon |F|}$.

For each $x\in X_{\mathsf{A}, F}^\iota $, the set $F'=\varphi ^{(x)}(F)=\{sx_s\colon s\in F\}$ lies in $\Theta ^\iota (\mathsf{A},F)$ (cf. \eqref{eq:px}). If we set $\bar{x}_{sx_s}= x_s^{-1}$ for every $s\in F$, we obtain a point $\bar{x}\in (\mathsf{A}^{-1})^{F'}$ with the property that $\bar{x}\in X_{\mathsf{A}^{-1},F'}^\iota $ and $\prod _{s\in F}f_{x_s}=\prod _{t\in F'}f^*_{\bar{x}_t}$.

\smallskip For each $F'\in \Theta ^\iota (\mathsf{A},F)$ we have that
	\begin{align*}
\sum_{x\in X_{\mathsf{A},F}^\iota :\, \,\varphi ^{(x)}(F)=F'}\enspace \prod _{s\in F}f_{x_s}&=\sum_{x\in X_{\mathsf{A},F}^\iota :\,\varphi ^{(x)}(F)=F'}\enspace \prod _{t\in F'}f^*_{\bar{x}_t} \le \sum_{\bar{x}\in X_{\mathsf{A}^{-1},\,F'}^\iota }\,\prod _{t\in F'}f^*_{\bar{x}_t}
	\\
&= \iper_{\mathsf{A}^{-1},\,F'}(f^*)\le e^{|F'|(C-2\varepsilon)}=e^{|F|(C-2\varepsilon)}.
	\end{align*}
Therefore
	\begin{align*}
\iper_{\mathsf{A},F}(f)&=\sum_{F'\in \Theta ^\iota (\mathsf{A},F) }\enspace \sum_{x\in X_{\mathsf{A},F}^\iota :\,\,\varphi ^{(x)}(F)=F'}\enspace \prod _{s\in F}f_{x_s}
	\\
&\le \sum_{F'\in \Theta ^\iota (\mathsf{A},F) }e^{|F|(C-2\varepsilon)}=|\Theta ^\iota (\mathsf{A},F) |e^{|F|(C-2\varepsilon)}\le e^{|F|C}.
	\end{align*}
It follows that $\iper(f)=\lim_{F \to \infty }\frac{1}{|F|}\log \iper_{\mathsf{A},F}(f)\le C$. Letting $C\to \iper(f^*)$ we obtain that $\iper(f)\le \iper(f^*)$, as desired.

\smallskip (2) By Proposition~\ref{P-per equal iper} it suffices to show that $\iper(fg)\ge \iper(f)+\iper(g)$. We may assume that $f, g\neq 0$. Put $\mathsf{A}=\supp(f)$ and $\mathsf{B}=\supp(g)$. Then $\mathsf{A}\mathsf{B}=\supp(fg)$.

Let $F\in \mathscr{F}(\Gamma )$. For each $x\in X_{\mathsf{A},F}^\iota $, put $F_x=\{sx_s\colon s\in F\}$, and denote by $\Xi_x$ the set of $y\in \mathsf{B}^F$ for which the map $w\colon F_x\rightarrow \mathsf{B}$, sending $t=sx_s$ for $s\in F$ to $y_s$, lies in $X_{\mathsf{B},F_x}^\iota $. For each $z\in X_{\mathsf{A}\mathsf{B},F}^\iota (fg)$, we have
	\begin{align*}
\prod _{s\in F}(fg)_{z_s}=\sum_{x\in \mathsf{A}^F, \, y\in \mathsf{B}^F, \, xy=z}\,\prod _{s\in F}(f_{x_s}g_{y_s})\ge \sum_{x\in X_{\mathsf{A},F}^\iota }\sum_{y\in \Xi_x, xy=z}\,\prod _{s\in F}(f_{x_s}g_{y_s}).
	\end{align*}
Thus
	\begin{align*}
\iper_{\mathsf{A}\mathsf{B},F}(fg)&=\sum_{z\in X_{\mathsf{A}\mathsf{B},F}^\iota }\,\prod _{s\in F}(fg)_{z_s} \ge \sum_{z\in X_{\mathsf{A}\mathsf{B},F}^\iota }\,\sum_{x\in X_{\mathsf{A},F}^\iota }\,\sum_{y\in \Xi_x, \,xy=z}\,\prod _{s\in F}(f_{x_s}g_{y_s})
	\\
&=\sum_{x\in X_{\mathsf{A},F}^\iota }\,\sum_{y\in \Xi_x}\,\prod _{s\in F}(f_{x_s}g_{y_s})
=\sum_{x\in X_{\mathsf{A},F}^\iota }\,\Biggl(\prod _{s\in F}f_{x_s}\Biggr)\sum_{y\in \Xi_x}\,\prod _{s\in F}g_{y_s}
	\\
&=\sum_{x\in X_{\mathsf{A}, F}^\iota }\,\Biggl(\prod _{s\in F}f_{x_s}\Biggr)\,\sum_{w\in X_{\mathsf{B},F_x}^\iota }\,\prod _{t\in F_x}g_{w_t}
=\sum_{x\in X_{\mathsf{A},F}^\iota }\,\Biggl(\prod _{s\in F}f_{x_s}\Biggr)\, \iper_{\mathsf{B},F_x}(g)
	\\
&\ge \sum_{x\in X_{\mathsf{A},F}^\iota }\,\Biggl(\prod _{s\in F}f_{x_s}\Biggr) e^{|F_x|\,\iper(g)}
=\sum_{x\in X_{\mathsf{A},F}^\iota }\Biggl(\prod _{s\in F}f_{x_s}\Biggr)e^{|F|\,\iper(g)}
	\\
&=\iper_{\mathsf{A},F}(f)\,e^{|F|\,\iper(g)}
\ge e^{|F|\,\iper(f)}e^{|F|\,\iper(g)} =e^{|F|(\iper(f)+\iper(g))},
	\end{align*}
where the 2nd and 3rd inequalities come from Proposition~\ref{P-per equal iper}. Applying Proposition~\ref{P-per equal iper} again, we obtain that
	\begin{displaymath}
\iper(fg)=\inf_{F\in \mathscr{F}(\Gamma )}\frac{1}{|F|}\log \iper_{\mathsf{A}\mathsf{B},F}(fg)\ge \iper(f)+\iper(g).\tag*{\qedsymbol}
	\end{displaymath}
	\renewcommand{\qedsymbol}{}
	\vspace{-\baselineskip}
	\end{proof}

	\begin{remark}
	\label{R-product}
In general, for $f, g\in \mathbb{R}_{\ge0}\Gamma $, the identity $\per (fg)=\per (f)+\per (g)$ may fail. Take, for example, $\Gamma =\mathbb{Z}$ and identify $\mathbb{R}\Gamma $ with $\mathbb{R}[x^{\pm1}]$. If $f=1+x$ and $g=f^*=1+x^{-1}$, we have $\per (fg)=\per (x+2+x^{-1})\ge \per (x+1+x^{-1})=\htop(X_\mathsf{A})>0$ for $\mathsf{A}=\{-1, 0, 1\}\in \mathscr{F}(\mathbb{Z})$, while $\per (f)=\per (g)=\htop(X_\mathsf{B})=0$ for $\mathsf{B}=\{0, 1\}\in \mathscr{F}(\mathbb{Z})$.
	\end{remark}

\section{Permanent vs Determinant}
	\label{S-per vs det}

For any $f=\sum_{s\in \Gamma}f_ss\in \mathbb{C}\Gamma $ we put $|f|=\sum_{s\in \Gamma}|f_s|s\in \mathbb{R}_{\ge 0}\Gamma $. Abusing notation in \eqref{eq:rho} a little we denote, for every $f\in \mathbb{C}\Gamma $, by $\lambda ^f\colon \ell ^2(\Gamma )\to \ell ^2(\Gamma )$ the bounded linear operator extending left multiplication by $f$ on $\mathbb{C}\Gamma \subset \ell ^2(\Gamma )$.

For any $F\in \mathscr{F}(\Gamma )$, we write $\iota _F$ for the natural embedding $\ell ^2(F)\rightarrow \ell ^2(\Gamma )$ and $\Upsilon _F$ for the projection $\ell ^2(\Gamma )\rightarrow \ell ^2(F)$. For $f\in \mathbb{C}\Gamma $, set $\lambda ^f_F:=\Upsilon _F \circ \lambda ^f\circ \iota_F\in \mathcal{B}(\ell ^2(F))$, the algebra of all bounded linear operators on $\ell ^2(F)$. Denote by $\Sym(F)$ the set of all permutations of $F$. For each $\sigma \in \Sym(F)$, put $\sgn(\sigma )=1$ if $\sigma $ is even and $-1$ if $\sigma $ is odd. In the canonical basis of $\ell ^2(F)$, $\lambda ^f_F$ is represented by a square matrix $M=(M_{s, t})_{s, t\in F}$ such that $M_{s, t}=f_{st^{-1}}$. Thus
	\begin{displaymath}
\det \lambda ^f_F=\det M=\sum_{\sigma \in \Sym(F)}\sgn(\sigma )\prod _{s\in F}M_{s, \sigma (s)}=\sum_{\sigma \in \Sym(F)}\sgn(\sigma )\prod _{s\in F}f_{s\sigma (s)^{-1}}.
	\end{displaymath}

Let $\mathsf{A}\in \mathscr{F}(\Gamma)$. For each $x\in X_{\mathsf{A}, F}^\iota $, denote by $\varphi ^{(x)}$ the injective map $t\mapsto tx_t$ from $F$ into $\Gamma $ as in \eqref{eq:px}. Denote by $\Theta ^\iota (\mathsf{A}, F)$ the set of $F'\subseteq F\mathsf{A}$ satisfying $|F'|=|F|$, and by $\Theta (\mathsf{A}, F)$ the set of $F'\in \Theta ^\iota (\mathsf{A}, F)$ satisfying $\{t\in \Gamma\colon t\mathsf{A}^{-1}\subseteq F\}\subseteq F'$. For each $F'\in \Theta ^\iota (\mathsf{A}, F)$, take a bijection $\psi _{F'}\colon F'\rightarrow F$.

	\begin{lemma}
	\label{L-det in IA}
Let $f\in \mathbb{C}\Gamma $ be nonzero. Put $\mathsf{A}=\supp(f)\in \mathscr{F}(\Gamma)$. For any $F\in \mathscr{F}(\Gamma)$, we have
	\begin{displaymath}
\det \lambda ^{ff^*}_{F^{-1}}=\sum_{F'\in \Theta ^\iota (\mathsf{A}, F)}\,\Biggl|\sum_{x\in X_{\mathsf{A}, F, F'} }\sgn(\psi _ {F'}\circ \varphi ^{(x)})\prod _{t\in F}f_{x_t}\Biggr|^2,
	\end{displaymath}
where $X_{\textsf{A},F,F'} $ is defined in \eqref{eq:XAF}.
	\end{lemma}

	\begin{proof}
Denote by $Q(F)$ the set of $(x, y)\in \mathsf{A}^F\times \mathsf{A}^F$ satisfying that the map
	\begin{displaymath}
\sigma _{xy^{-1}}\colon t\mapsto tx_ty_t^{-1}
	\end{displaymath}
is a permutation of $F$.
For each $x\in \mathsf{A}^F$, denote by $Q(F)_x$ the set of $y\in \mathsf{A}^F$ satisfying that $(x, y)\in Q(F)$.

For any $\sigma \in \mathscr{S}(F^{-1})$, we have
	\begin{align*}
\prod_{s\in F^{-1}}(ff^*)_{s\sigma(s)^{-1}} &=\sum_{\substack{x, y\in \mathsf{A}^{(F^{-1})}
	\\
x_sy_s^{-1}=s\sigma(s)^{-1} \,\forall \,s\in F^{-1}}} \, \prod_{s\in F^{-1}}\bigl(f_{x_s}\overline{f_{y_s}}\bigr) & \scalebox{.8}{$s= t^{-1}$}
	\\
&=\sum_{\substack{x, y\in \mathsf{A}^{(F^{-1})}
	\\
x_{t^{-1}} y_{t^{-1}}^{-1}=t^{-1}\sigma(t^{-1})^{-1} \, \forall \,t\in F}}\prod_{t\in F}\bigl(f_{x_{t^{-1}}}\overline{f_{y_{t^{-1}}}}\bigr) & \scalebox{.8}{$x_{t^{-1}}= x'_t,\,y_{t^{-1}}= y'_t$}
	\\
&=\sum_{\substack{x', y'\in \mathsf{A}^F
	\\
t x'_t{y'_t}^{-1}=\sigma(t^{-1})^{-1} \,\forall \,t\in F}}\prod_{t\in F}\bigl(f_{x'_t}\overline{f_{y'_t}}\bigr).
	\end{align*}
Thus
	\begin{equation}
	\label{E-per ge det1}
	\begin{aligned}
\det \lambda ^{ff^*}_{F^{-1}}&=\sum_{\sigma\in \mathscr{S}(F^{-1})}\sgn(\sigma)\prod_{s\in F^{-1}}(ff^*)_{s\sigma(s)^{-1}}
	\\
&=\sum_{\sigma\in \mathscr{S}(F^{-1})}\, \sum_{\substack{x', y'\in \mathsf{A}^F
	\\
tx'_t {y'_t}^{-1}=\sigma(t^{-1})^{-1} \,\forall\, t\in F}} \sgn(\sigma)\prod_{t\in F}\bigl(f_{x'_t}\overline{f_{y'_t}}\bigr) & \qquad \scalebox{.8}{$\sigma'(t)=\sigma(t^{-1})^{-1}$}
	\\
&=\sum_{\sigma'\in \mathscr{S}(F)}\sum_{\substack{x', y'\in \mathsf{A}^{F}
	\\
t x'_t {y'_t}^{-1}=\sigma'(t) \,\forall\, t\in F}}\sgn(\sigma')\prod_{t\in F}\bigl(f_{x'_t}\overline{f_{y'_t}}\bigr)
	\\
&=\sum_{\sigma\in \mathscr{S}(F)}\sum_{\substack{x, y\in \mathsf{A}^{F}
	\\
t x_t y_t^{-1}=\sigma(t) \,\forall\, t\in F}}\sgn(\sigma)\prod_{t\in F}\bigl(f_{x_t}\overline{f_{y_t}}\bigr)
	\\
&=\sum_{(x, y)\in Q(F)}\sgn(\sigma_{xy^{-1}})\prod_{t\in F}\Biggl(f_{x_t}\overline{f_{y_t}}\Biggr)
	\\
&=\sum_{x\in \mathsf{A}^F}\Biggl(\prod_{t\in F}f_{x_t}\Biggr)\cdot \sum_{y\in Q(F)_x} \sgn(\sigma_{xy^{-1}})\prod_{t\in F}\overline{f_{y_t}}.
	\end{aligned}
	\end{equation}

Let $x\in \mathsf{A}^F\smallsetminus X_{\mathsf{A}, F}^\iota $. Then there are distinct $s, t\in F$ such that $sx_s=tx_t$. Denote by $\tau $ the permutation of $F$ exchanging $s$ and $t$ and fixing all other elements of $F$. For each $y\in Q(F)_x$, we have
	\begin{displaymath}
sx_s(y\circ \tau )_s^{-1}=sx_sy_t^{-1}=tx_ty_t^{-1}=\sigma _{xy^{-1}}(t)=(\sigma _{xy^{-1}}\circ \tau )(s),
	\end{displaymath}
and similarly,
	\begin{displaymath}
tx_t(y\circ \tau )_t^{-1}=(\sigma _{xy^{-1}}\circ \tau )(t),
	\end{displaymath}
while for any $g\in F\smallsetminus \{s,t\}$, we have
	\begin{displaymath}
gx_g(y\circ \tau )_g^{-1} = gx_gy_g^{-1} = \sigma _{xy^{-1}}(g) = (\sigma _{xy^{-1}}\circ \tau )(g).
	\end{displaymath}
It follows that for each $y\in Q(F)_x$, we have
	\begin{displaymath}
y\circ \tau \in Q(F)_x\enspace \textup{and}\enspace \sigma _{x(y\circ \tau )^{-1}}=\sigma _{xy^{-1}}\circ \tau .
	\end{displaymath}
In particular, $y \mapsto y\circ \tau $ is a permutation of $Q(F)_x$. Thus
	\begin{align*}
\sum_{y\in Q(F)_x} \sgn(\sigma _{xy^{-1}})\,\prod _{t\in F}\overline{f_{y_t}}&=\sum_{y\in Q(F)_x} \sgn(\sigma _{x(y\circ \tau )^{-1}})\,\prod _{t\in F}\overline{f_{(y\circ \tau )_t}}
	\\
&=\sum_{y\in Q(F)_x} \sgn(\sigma _{xy^{-1}}\circ \tau )\,\prod _{t\in F}\overline{f_{y_t}}
	\\
&=-\sum_{y\in Q(F)_x} \sgn(\sigma _{xy^{-1}}) \,\prod _{t\in F}\overline{f_{y_t}},
	\end{align*}
and hence
	\begin{equation}
	\label{E-per ge det2}
\sum_{y\in Q(F)_x} \sgn(\sigma_{xy^{-1}})\prod_{t\in F}\overline{f_{y(t)}}=0.
	\end{equation}
From \eqref{E-per ge det1} and \eqref{E-per ge det2} we get
	\begin{align*}
\det \lambda ^{ff^*}_{F^{-1}} &=\sum_{x\in X_{\mathsf{A}, F}^\iota }\Biggl(\prod _{t\in F}f_{x_t}\Biggr)\cdot \sum_{y\in Q(F)_x} \sgn(\sigma _{xy^{-1}})\prod _{t\in F}\overline{f_{y_t}}
	\\
&=\sum_{F'\in \Theta ^\iota (\mathsf{A}, F)}\,\sum_{x\in X_{\mathsf{A}, F, F'} }\Biggl(\prod _{t\in F}f_{x_t}\Biggr)\cdot \sum_{y\in Q(F)_x} \sgn(\sigma _{xy^{-1}})\prod _{t\in F}\overline{f_{y_t}}.
	\end{align*}

For any $x\in \mathsf{A}^F$ and $y\in Q(F)_x$, we have the permutation $\sigma_{xy^{-1}}$ of $F$, whence we can set
	\begin{displaymath}
z^{(x,y)}=y\circ (\sigma_{xy^{-1}})^{-1}\in \mathsf{A}^F.
	\end{displaymath}

Let $x\in X_{\mathsf{A}, F}^\iota $ and $y\in Q(F)_x$. For any $t\in F$, setting $s=\sigma_{xy^{-1}}(t)\in F$, we have $z^{(x,y)}_s=y_t$, whence
	\begin{equation}
	\label{E-per ge det4}
tx_t=tx_ty_t^{-1}y_t=\sigma_{xy^{-1}}(t) y_t=sz^{(x,y)}_s.
	\end{equation}
For any distinct $s, s'\in F$, since $\sigma_{xy^{-1}}$ is a permutation of $F$, we can find $t, t'\in F$ with $s=\sigma_{xy^{-1}}(t)$, $s'=\sigma_{xy^{-1}}(t')$, and $t\neq t'$. Since $x\in X_{\mathsf{A}, F}^\iota $, we have $tx_t\neq t'x_{t'}$. Therefore
	\begin{displaymath}
sz^{(x,y)}_s\overset{\eqref{E-per ge det4}}=tx_t\neq t'x_{t'}\overset{\eqref{E-per ge det4}}=s'z^{(x,y)}_{s'}.
	\end{displaymath}
This shows that $z^{(x,y)}\in X_{\mathsf{A}, F}^\iota $.

Let $F'\in \Theta ^\iota (\mathsf{A}, F)$ and $x\in X_{\mathsf{A}, F, F'} $. For each $y\in Q(F)_x$, we have
	\begin{displaymath}
F'=\{tx_t \colon t\in F\}\overset{\eqref{E-per ge det4}}=\{sz^{(x, y)}_s\colon s\in F\},
	\end{displaymath}
whence $z^{(x,y)}\in X_{\mathsf{A},F,F'} $. We claim that the map $Q(F)_x\rightarrow X_{\mathsf{A}, F, F'} $ sending $y$ to $z^{(x,y)}$ is injective.

Indeed, let $y, y'\in Q(F)_x$ with $z^{(x, y)}=z^{(x, y')}$. For each $s\in F$, since $\sigma_{xy^{-1}}$ and $\sigma_{x(y')^{-1}}$ are permutations of $F$, we can find $t, t'\in F$ such that $ \sigma_{xy^{-1}}(t)=s$ and $\sigma_{x(y')^{-1}}(t')=s$. Then
	\begin{displaymath}
tx_t\overset{\eqref{E-per ge det4}}=sz^{(x, y)}_s=sz^{(x, y')}_s\overset{\eqref{E-per ge det4}}=t'x_{t'}.
	\end{displaymath}
Since $x\in X_{\mathsf{A}, F}^\iota $, we get $t=t'$.
It follows that $\sigma_{xy^{-1}}=\sigma_{x(y')^{-1}}$ and hence $y=y'$. This proves our claim.

We next claim that the map $Q(F)_x\rightarrow X_{\mathsf{A}, F, F'} $ sending $y$ to $z^{(x, y)}$ is surjective. Let $z\in X_{\mathsf{A}, F, F'} $. We have the bijections $\varphi ^{(x)}, \varphi ^{(z)}\colon F\rightarrow F'$, whence we have the permutation $(\varphi ^{(z)})^{-1}\circ \varphi ^{(x)}$ of $F$. Set
	\begin{displaymath}
y=z\circ ((\varphi ^{(z)})^{-1}\circ \varphi ^{(x)})\in \mathsf{A}^F.
	\end{displaymath}
For each $t\in F$, setting $s=(\varphi ^{(z)})^{-1}\circ \varphi ^{(x)})(t)\in F$, we have
	\begin{displaymath}
y_t=z_s \quad \text{ and } \quad \varphi ^{(z)}(s)=\varphi ^{(x)}(t),
	\end{displaymath}
and hence
	\begin{displaymath}
\sigma_{xy^{-1}}(t)=tx_ty_t^{-1}= \varphi ^{(x)}(t)y_t^{-1}=\varphi ^{(z)}(s)z_s^{-1}= s=((\varphi ^{(z)})^{-1}\circ \varphi ^{(x)})(t).
	\end{displaymath}
Thus
	\begin{equation}
	\label{E-per ge det3}
\sigma_{xy^{-1}}=(\varphi ^{(z)})^{-1}\circ \varphi ^{(x)} \in \Sym(F).
	\end{equation}
Therefore $y\in Q(F)_x$, and
	\begin{displaymath}
z^{(x, y)}=y\circ (\sigma_{xy^{-1}})^{-1}=y\circ ((\varphi ^{(z)})^{-1}\circ \varphi ^{(x)})^{-1}=z.
	\end{displaymath}
This proves our claim. Thus the map
$Q(F)_x\rightarrow X_{\mathsf{A}, F, F'} $ sending $y$ to $z^{(x, y)}$ is a bijection, and \eqref{E-per ge det3} holds for $z=z^{(x, y)}$.

For each $F'\in \Theta ^\iota (\mathsf{A}, F)$ we have
	\begin{align*}
\sum_{x\in X_{\mathsf{A}, F, F'} }&\Biggl(\prod_{t\in F}f_{x_t}\Biggr)\cdot \sum_{y\in Q(F)_x} \sgn(\sigma_{xy^{-1}})\prod_{t\in F}\overline{f_{y_t}} \hspace{30mm}\scalebox{.8}{$\text{since } (\sigma_{xy^{-1}})^{-1}\in \mathscr{S}(F)$}
	\\
&=\sum_{x\in X_{\mathsf{A}, F, F'} }\Biggl(\prod_{t\in F}f_{x_t}\Biggr)\cdot \sum_{y\in Q(F)_x} \sgn(\sigma_{xy^{-1}})\prod_{t\in F}\overline{f_{(y\circ (\sigma_{xy^{-1}})^{-1})_t}}
	\\
&=\sum_{x\in X_{\mathsf{A}, F, F'} }\Biggl(\prod_{t\in F}f_{x_t}\Biggr)\cdot \sum_{y\in Q(F)_x} \sgn(\sigma_{xy^{-1}})\prod_{t\in F}\overline{f_{z^{(x, y)}_t}}
	\\
&\overset{\eqref{E-per ge det3}}=\sum_{x\in X_{\mathsf{A}, F, F'} }\Biggl(\prod_{t\in F}f_{x_t}\Biggr)\cdot \sum_{z\in X_{\mathsf{A}, F, F'} } \sgn((\varphi ^{(z)})^{-1}\circ \varphi ^{(x)})\prod_{t\in F}\overline{f_{z_t}}
	\\
&=\sum_{x\in X_{\mathsf{A}, F, F'} }\Biggl(\prod_{t\in F}f_{x_t}\Biggr)\cdot \sum_{z\in X_{\mathsf{A}, F, F'} } \sgn\bigl((\psi _{F'}\circ \varphi ^{(z)})^{-1}\circ (\psi _{F'}\circ \varphi ^{(x)})\bigr)\prod_{t\in F}\overline{f_{z_t}}
	\\
&=\sum_{x\in X_{\mathsf{A}, F, F'} }\Biggl(\prod_{t\in F}f_{x_t}\Biggr)\cdot \sum_{z\in X_{\mathsf{A}, F, F'} } \sgn((\psi _{F'}\circ \varphi ^{(z)})^{-1})\sgn(\psi _{F'}\circ \varphi ^{(x)})\prod_{t\in F}\overline{f_{z_t}}
	\\
&=\sum_{x\in X_{\mathsf{A}, F, F'} }\Biggl(\prod_{t\in F}f_{x_t}\Biggr)\cdot \sum_{z\in X_{\mathsf{A}, F, F'} } \sgn(\psi _{F'}\circ \varphi ^{(z)})\sgn(\psi _{F'}\circ \varphi ^{(x)})\prod_{t\in F}\overline{f_{z_t}}
	\\
&=\Biggl(\sum_{x\in X_{\mathsf{A}, F, F'} }\sgn(\psi _{F'}\circ \varphi ^{(x)})\prod_{t\in F}f_{x_t}\Biggr)\cdot \Biggl(\sum_{z\in X_{\mathsf{A}, F, F'} }\sgn(\psi _{F'}\circ \varphi ^{(z)})\prod_{t\in F}\overline{f_{z_t}}\Biggr)
	\\
&=\Biggl|\sum_{x\in X_{\mathsf{A}, F, F'} }\sgn(\psi _{F'}\circ \varphi ^{(x)})\prod_{t\in F}f_{x_t}\Biggr|^2.
	\end{align*}

Finally we have
	\begin{align*}
\det \lambda ^{ff^*}_{F^{-1}}&=\sum_{F'\in \Theta ^\iota (\mathsf{A}, F)}\,\sum_{x\in X_{\mathsf{A}, F, F'} }\biggl(\prod _{t\in F}f_{x_t}\biggr)\cdot \sum_{y\in Q(F)_x} \sgn(\sigma _{xy^{-1}})\prod _{t\in F}\overline{f_{y_t}}
	\\
&=\sum_{F'\in \Theta ^\iota (\mathsf{A}, F)}\,\biggl|\sum_{x\in X_{\mathsf{A}, F, F'} }\sgn(\psi _{F'}\circ \varphi ^{(x)})\prod _{t\in F}f_{x_t}\biggr|^2.
\tag*{\qedsymbol}
	\end{align*}
\renewcommand{\qedsymbol}{}
	\end{proof}
\vspace{-\baselineskip}

For comparing permanents and determinants we assume in the following discussion that $\Gamma $ is infinite and amenable. In order to unburden notation a little we denote by $\ddet _\textup{FK}(f)$ the \emph{Fuglede-Kadison determinant} of the linear operator $\lambda ^f \in \mathcal{B}(\ell ^2(\Gamma ))$ introduced in \cite{FK}.

	\begin{proposition}
	\label{P-per ge det}
For every $f\in \mathbb{C}\Gamma $ we have $\per (|f|)\ge \ddet _\textup{FK}(f)$.
	\end{proposition}

	\begin{proof}
We may assume that $f\neq 0$. Put $\mathsf{A}=\supp(f)\in \mathscr{F}(\Gamma)$.

From \cite{FK} we have $\ddet _\textup{FK}(f)=\ddet _\textup{FK}({f^*})=\frac{1}{2}\ddet _\textup{FK}({ff^*})$.
By \cite{LT}*{Theorem 1.4} we have $\ddet _\textup{FK}({ff^*})=\lim_{F}\frac{1}{|F|}\log\det \lambda ^{ff^*}_{F^{-1}}$. By Proposition~\ref{P-per equal iper} we have $\per (|f|)=\iper(|f|)=\lim_{F}\frac{1}{|F|}\log \per _{\mathsf{A}, F}^\iota (|f|)$. Thus it suffices to show
	\begin{displaymath}
\det \lambda ^{ff^*}_{F^{-1}}\le \Biggl(\per _{\mathsf{A}, F}^\iota (|f|)\Biggr)^2
	\end{displaymath}
for every $F\in \mathscr{F}(\Gamma)$.

Indeed, according to Lemma~\ref{L-det in IA},
	\begin{align*}
\det \lambda ^{ff^*}_{F^{-1}}&=\sum_{F'\in \Theta ^\iota (\mathsf{A}, F)}\,\Biggl|\sum_{x\in X_{\mathsf{A}, F, F'} }\sgn(\psi _ {F'}\circ \varphi ^{(x)})\prod _{t\in F}f_{x_t}\Biggr|^2
	\\
&\le \sum_{F'\in \Theta ^\iota (\mathsf{A}, F)}\Biggl(\sum_{x\in X_{\mathsf{A}, F, F'} }\prod _{t\in F}|f|_{x_t}\Biggr)^2 \le \Biggl(\sum_{F'\in \Theta ^\iota (\mathsf{A}, F)}\sum_{x\in X_{\mathsf{A}, F, F'} }\,\prod _{t\in F}|f|_{x_t}\Biggr)^2
	\\
&= \bigl(\per _{\mathsf{A}, F}^\iota (|f|)\bigr)^2.
\tag*{\qedsymbol}
	\end{align*}
\renewcommand{\qedsymbol}{}
\vspace{-\baselineskip}
	\end{proof}

	\begin{remark}
	\label{R-per vs det}
In general, for $f\in \mathbb{R}_{\ge 0}\Gamma $ the identity $\per (f)=\ddet _\textup{FK}(f)$ will fail. For example, if we put $\Gamma=\mathbb{Z}$ and $\mathbb{R}\Gamma =\mathbb{R}[x^{\pm1}]$ as in Remark \ref{R-product}, and if $f=(1+x)(1+x^{-1})=x^{-1}+2+x$, we have $\per (f)\ge \per (x^{-1}+1+x)=\htop(X_\mathsf{A})>0$ for $\mathsf{A}=\{-1, 0, 1\}\in \mathscr{F}(\mathbb{Z})$, while $\ddet _\textup{FK}(f)=2 \ddet _\textup{FK}({1+x})=0$.

However, the problem of finding possible connections between $\per (f)$ and $\ddet _\textup{FK}(f)$ is intriguing. Both these quantities arise naturally in dynamical contexts: if $f$ is the indicator function of a set $\mathsf{A}\in \mathscr{F}(\Gamma )$, Remark \ref{R-permanent and entropy} shows that $\per (f)$ is the topological entropy of the restricted permutation space $X_\mathsf{A}\subseteq \mathsf{A}^\Gamma $; on the other hand, if $f\in \mathbb{Z}\Gamma $ is arbitrary, $\ddet _\textup{FK}(f)$ is the topological entropy of the principal algebraic $\Gamma $-action $(X_f,\lambda _f)$ arising from $f$ (\cites{LSW, Deninger, LT}). At least for $\Gamma =\mathbb{Z}^d$, the explicit value of $\ddet _\textup{FK}(f)$ is given by the logarithmic Mahler measure of $f$ (\cite{LSW}), allowing straightforward numerical calculations and, in special cases, explicit arithmetical formulas for $\htop(\lambda _f)=\ddet _\textup{FK}(f)$ (cf. e.g., \cite{Boyd}). On the other hand, the complexity of computing the permanent of an $n\times n$ (0,1)-matrix is \emph{NP-hard} and conjecturally not possible in polynomial time (cf. \cite{Valiant79}). This doesn't give much hope for general results concerning explicit values of $\per (f),\,f \in \mathbb{Z}_{\ge0}[\mathbb{Z}^d]$, or, indeed, of $\htop (X_\mathsf{A}), \, \mathsf{A}\in \mathscr{F}(\mathbb{Z}^d)$.

In view of the difficulty of computing the permanent of a large matrix $M$ one may try to search for linear operations (like sign changes) on the entries of $M$, resulting in a matrix $M'$ whose determinant is equal to the permanent of $M$. Although this approach works only under severe restrictions (cf. e.g., \cites{Gibson71, Marcus-Minc61}), it is the motivation for the following lemma.
	\end{remark}

	\begin{lemma}
	\label{L-per equal det}
Let $h\in \mathbb{R}_{\ge 0} \Gamma $ be nonzero and let $\mathsf{A}=\supp(h)\in \mathscr{F}(\Gamma)$. Let $W\subseteq \mathbb{R} \Gamma $ such that $|f|=h$ for every $f\in W$. Assume that there are a right F{\o}lner sequence $(F_n)_{n\in \mathbb{N}}$ for $\Gamma $
and a set $H_n\subseteq F_n$ for each $n\in \mathbb{N}$ such that the following conditions hold:
	\begin{enumerate}
	\item
$|H_n|/|F_n|\to 0$ as $n\to \infty$;
	\item
for each $n\in \mathbb{N}$, $F'\in \Theta(\mathsf{A}, F_n)$ and $y\in \mathsf{A}^{H_n}$, there is some $f\in W$ so that the sign of $\sgn(\psi _{F'}\circ \varphi ^{(x)})\prod _{t\in F_n}f_{x_t}$ is the same for all $x\in X_{\mathsf{A}, F_n, F'} $ satisfying $x|_{H_n}=y$.
	\end{enumerate}
Then $\per (h)=\max_{f\in W}\ddet _\textup{FK}(f)$.
	\end{lemma}

	\begin{proof}
By Proposition~\ref{P-per ge det} it suffices to show that $\per (h)\le \max_{f\in W}\ddet _\textup{FK}(f)$.
From \cite{FK} we have $\ddet _\textup{FK}(f)=\ddet _\textup{FK}({f^*})=\frac{1}{2}\ddet _\textup{FK}({ff^*})$ for each $f\in W$.
By \cite{LT}*{Theorem 1.4} we have $\ddet _\textup{FK}({ff^*})=\lim_{F}\frac{1}{|F|}\log\det \lambda ^{ff^*}_{F^{-1}}$ for each $f\in W$. We may assume that $F_n\smallsetminus H_n\neq \varnothing $ for every $n\in \mathbb{N}$. Since $(F_n)_{n\in \mathbb{N}}$ is a right F{\o}lner sequence and $|H_n|/|F_n|\to 0$ as $n\to \infty$, the sequence $(F_n\smallsetminus H_n)_{n\in \mathbb{N}}$ is also right F{\o}lner.

\smallskip Let $n\in \mathbb{N}$ and $f\in W$. From Lemma~\ref{L-det in IA} we get
	\begin{align*}
&\|f\|_2^{2|H_n|}\det \lambda ^{ff^*}_{(F_n\smallsetminus H_n)^{-1}}
	\\
&=\Biggl(\sum_{y\in \mathsf{A}^{H_n}}\,\prod _{t\in H_n}|f_{y_t}|^2\Biggr)\Biggl(\sum_{F'\in \Theta ^\iota (\mathsf{A}, F_n\smallsetminus H_n)}\Biggr|\,\sum_{x\in X_{\mathsf{A}, F_n\smallsetminus H_n, F'} }\sgn(\psi _{F'}\circ \varphi ^{(x)})\prod _{t\in F_n\smallsetminus H_n}f_{x_t}\Biggr|^2\Biggr)
	\\
&= \sum_{y\in \mathsf{A}^{H_n}}\sum_{F'\in \Theta ^\iota (\mathsf{A}, F_n\smallsetminus H_n)}\,\Biggl|\sum_{x\in X_{\mathsf{A}, F_n\smallsetminus H_n, F'} }\sgn(\psi _{F'}\circ \varphi ^{(x)})\Biggl(\prod _{t\in H_n}f_{y_t}\Biggr)\Biggl(\prod _{t\in F_n\smallsetminus H_n}f_{x_t}\Biggr)\Biggr|^2
	\\
&\ge \sum_{y\in X_{\mathsf{A}, H_n}^\iota }\sum_{\substack{F'\in \Theta ^\iota (\mathsf{A}, F_n\smallsetminus H_n)
	\\
F'\cap \varphi ^{(y)}(H_n)=\varnothing }}\Biggl|\sum_{x\in X_{\mathsf{A}, F_n\smallsetminus H_n, F'} }\sgn(\psi _{F'}\circ \varphi ^{(x)})\Biggl(\prod _{t\in H_n}f_{y_t}\Biggr)\Biggl(\prod _{t\in F_n\smallsetminus H_n}f_{x_t}\Biggr)\Biggr|^2.
	\end{align*}
For any $y\in X_{\mathsf{A}, H_n}^\iota $ and $F'\in \Theta ^\iota (\mathsf{A}, F_n\smallsetminus H_n)$ satisfying $F'\cap \varphi ^{(y)}(H_n)=\varnothing $, the set $F''=F'\cup \varphi ^{(y)}(H_n)$ lies in $\Theta ^\iota (\mathsf{A}, F_n)$. Denote by $\psi _ {F'', y}$ the bijection $F''\rightarrow F_n$ which is equal to $\psi _ {F'}$ and $(\varphi ^{(y)})^{-1}$ on $F'$ and $\varphi ^{(y)}(H_n)$ respectively. Fixing such $y, F'$ and $F''$, there is a natural $1$-$1$ correspondence between $x\in X_{\mathsf{A}, F_n\smallsetminus H_n, F'} $ and $z\in X_{\mathsf{A}, F_n, F''} $ equal to $y$ on $H_n$, given by $z=x$ on $F_n\smallsetminus H_n$ and $z=y$ on $H_n$. Thus
	\begin{align*}
\sum_{y\in X_{\mathsf{A}, H_n}^\iota }&\sum_{\substack{F'\in \Theta ^\iota (\mathsf{A}, F_n\smallsetminus H_n)
	\\
F'\cap \varphi ^{(y)}(H_n)=\varnothing }}\,\Biggl|\sum_{x\in X_{\mathsf{A}, F_n\smallsetminus H_n, F'} }\sgn(\psi _{F'}\circ \varphi ^{(x)})\Biggl(\prod _{t\in H_n}f_{y_t}\Biggr)\biggl(\prod _{t\in F_n\smallsetminus H_n}f_{x_t}\Biggr)\Biggr|^2
	\\
&= \sum_{F''\in \Theta ^\iota (\mathsf{A}, F_n)}\sum_{\substack{y\in X_{\mathsf{A}, H_n}^\iota
	\\
\varphi ^{(y)}(H_n)\subseteq F''}}\Biggl|\sum_{\substack{z\in X_{\mathsf{A}, F_n, F''}
	\\
z|_{H_n}=y}}\sgn(\psi _ {F'', y}\circ \varphi ^{(z)})\prod _{t\in F_n}f_{z_t}\Biggr|^2
	\\
&= \sum_{F''\in \Theta ^\iota (\mathsf{A}, F_n)}\sum_{\substack{y\in X_{\mathsf{A}, H_n}^\iota
	\\
\varphi ^{(y)}(H_n)\subseteq F''}}\Biggl|\sum_{\substack{z\in X_{\mathsf{A}, F_n, F''}
	\\
z|_{H_n}=y}}\sgn(\psi _ {F'', y}\circ \psi _ {F''}^{-1})\sgn(\psi _ {F''}\circ \varphi ^{(z)})\prod _{t\in F_n}f_{z_t}\Biggr|^2
	\\
&= \sum_{F''\in \Theta ^\iota (\mathsf{A}, F_n)}\sum_{\substack{y\in X_{\mathsf{A}, H_n}^\iota
	\\
\varphi ^{(y)}(H_n)\subseteq F''}}\Biggl|\sum_{\substack{z\in X_{\mathsf{A}, F_n, F''}
	\\
z|_{H_n}=y}}\sgn(\psi _ {F''}\circ \varphi ^{(z)})\prod _{t\in F_n}f_{z_t}\Biggr|^2
	\\
&\ge \sum_{F''\in \Theta(\mathsf{A}, F_n)}\sum_{\substack{y\in X_{\mathsf{A}, H_n}^\iota
	\\
\varphi ^{(y)}(H_n)\subseteq F''}}\Biggl|\sum_{\substack{z\in X_{\mathsf{A}, F_n, F''}
	\\
z|_{H_n}=y}}\sgn(\psi _ {F''}\circ \varphi ^{(z)})\prod _{t\in F_n}f_{z_t}\Biggr|^2.
	\end{align*}
For each $F''\in \Theta(\mathsf{A}, F_n)$ and $y\in X_{\mathsf{A}, H_n}^\iota $ satisfying $\varphi ^{(y)}(H_n)\subseteq F''$, by condition (2) we can find some $f^{(F'', y)}\in W$ such that
	\begin{displaymath}
\Biggl|\sum_{\substack{z\in X_{\mathsf{A}, F_n, F''}
	\\
z|_{H_n}=y}}\sgn(\psi _ {F''}\circ \varphi ^{(z)})\prod _{t\in F_n}f^{(F'', y)}_{z_t}\Biggr|=\sum_{\substack{z\in X_{\mathsf{A}, F_n, F''}
	\\
z|_{H_n}=y}}\prod _{t\in F_n}|f^{(F'', y)}_{z_t}|=\sum_{\substack{z\in X_{\mathsf{A}, F_n, F''}
	\\
z|_{H_n}=y}}\prod _{t\in F_n}h_{z_t}.
	\end{displaymath}
Therefore
	\begin{align*}
|W|\cdot \|h\|_2^{2|H_n|}&\max_{f\in W}\det \lambda ^{ff^*}_{(F_n\smallsetminus H_n)^{-1}} \ge \sum_{f\in W}\|f\|_2^{2|H_n|}\det \lambda ^{ff^*}_{(F_n\smallsetminus H_n)^{-1}}
	\\
&\ge \sum_{F''\in \Theta(\mathsf{A}, F_n)}\sum_{\substack{y\in X_{\mathsf{A}, H_n}^\iota
	\\
\varphi ^{(y)}(H_n)\subseteq F''}}\Biggl(\sum_{\substack{z\in X_{\mathsf{A}, F_n, F''}
	\\
z|_{H_n}=y}}\prod _{t\in F_n}h_{z_t}\Biggr)^2
	\\
&\ge \frac{1}{|\Theta(\mathsf{A}, F_n)|\cdot |\mathsf{A}|^{|H_n|}}\Biggl(\sum_{F''\in \Theta(\mathsf{A}, F_n)}\sum_{\substack{y\in X_{\mathsf{A}, H_n}^\iota
	\\
\varphi ^{(y)}(H_n)\subseteq F''}}\sum_{\substack{z\in X_{\mathsf{A}, F_n, F''}
	\\
z|_{H_n}=y}}\prod _{t\in F_n}h_{z_t}\Biggr)^2
	\\
&= \frac{1}{|\Theta(\mathsf{A}, F_n)|\cdot |\mathsf{A}|^{|H_n|}}\Biggl(\sum_{F''\in \Theta(\mathsf{A}, F_n)}\sum_{z\in X_{\mathsf{A}, F_n, F''} }\prod _{t\in F_n}h_{z_t}\Biggr)^2
	\\
&= \frac{1}{|\Theta(\mathsf{A}, F_n)|\cdot |\mathsf{A}|^{|H_n|}}\bigl(\per _{\mathsf{A}, F_n}(h)\bigr)^2.
	\end{align*}

Note that by Stirling's approximation formula we have
	\begin{displaymath}
\lim_{n\to \infty}\frac{1}{|F_n|}\log |\Theta(\mathsf{A}, F_n)|=\lim_{n\to \infty}\frac{1}{|F_n|}\log |\Theta ^\iota (\mathsf{A}, F_n)|=0.
	\end{displaymath}
Since $\Gamma $ is infinite, by condition (1) we also have
	\begin{displaymath}
\lim_{n\to \infty}\frac{1}{|F_n|}\log (|W|\cdot \|h\|_2^{2|H_n|})=\lim_{n\to \infty}\frac{1}{|F_n|}\log |\mathsf{A}|^{|H_n|}=0.
	\end{displaymath}
Thus
	\begin{align*}
\max_{f\in W}\ddet _\textup{FK}(f)&=\frac{1}{2}\lim_{n\to \infty}\frac{1}{|F_n\smallsetminus H_n|}\log \max_{f\in W}\det \lambda ^{ff^*}_{(F_n\smallsetminus H_n)^{-1}}
	\\
&=\frac{1}{2}\lim_{n\to \infty}\frac{1}{|F_n|}\log \max_{f\in W}\det \lambda ^{ff^*}_{(F_n\smallsetminus H_n)^{-1}}
	\\
&\ge \lim_{n\to \infty}\frac{1}{|F_n|}\log\per _{\mathsf{A}, F_n}(h) =\per (h).
\tag*{\qedsymbol}
	\end{align*}
\renewcommand{\qedsymbol}{}
\vspace{-\baselineskip}
	\end{proof}

Let $F, \mathsf{A}\in \mathscr{F}(\Gamma)$ and $F'\in \Theta(\mathsf{A}, F)$. For each $x\in X_{\mathsf{A}, F, F'} $ and $g\in F\smallsetminus F'$ there is a unique $n_{x, g}\in \mathbb{N}$ such that $(\varphi ^{(x)})^j(g)\in F\cap F'$ for all $1\le j<n_{x, g}$ and $(\varphi ^{(x)})^{n_{x, g}}(g)\in F'\smallsetminus F$.

\smallskip Recall that a total order on $\Gamma $ is \textsl{left invariant} if for any $s, t, g\in \Gamma $, one has $s<t$ if and only if $gs<gt$. The group $\Gamma $ is \textsl{left orderable} if it has a left invariant total order.

	\begin{lemma}
	\label{L-order}
Let $\Gamma $ be given a left invariant total order. Let $f\in \mathbb{R}\Gamma $ be nonzero with $\mathsf{A}=\supp(f)\in \mathscr{F}(\Gamma)$.
Let $F\in \mathscr{F}(\Gamma)$ and $F'\in \Theta(\mathsf{A}, F)$.
Assume that the following conditions hold:
	\begin{enumerate}
	\item
$s\ge \matheur{e}_\Gamma $ for all $s\in \mathsf{A}$;
	\item
$f_{\matheur{e}_\Gamma}>0$ if $\matheur{e}_\Gamma \in \mathsf{A}$;
	\item
The map $F\smallsetminus F'\to F'\smallsetminus F$ sending $g$ to $(\varphi ^{(x)})^{n_{x, g}}(g)$ is the same for all $x\in X_{\mathsf{A}, F, F'} $;
	\item
For each $g\in F\smallsetminus F'$, the sign of $(-1)^{n_{x, g}}\prod _{j=0}^{n_{x, g}-1}f_{x_{(\varphi ^{(x)})^j(g)}}$ is the same for all $x\in X_{\mathsf{A}, F, F'} $.
	\end{enumerate}
Then the sign of $\sgn(\psi _{F'}\circ \varphi ^{(x)})\prod _{t\in F}f_{x_t}$ is the same for all $x\in X_{\mathsf{A}, F, F'} $.
	\end{lemma}
	\begin{proof}
Since the conclusion does not depend on the choice of $\psi _ {F'}$, we may assume that $\psi _ {F'}(h)=h$ for all $h\in F\cap F'$.

Let $x\in X_{\mathsf{A}, F, F'} $. Denote by $C_x$ the set of cycles of $\psi _{F'}\circ \varphi ^{(x)}$
intersecting with $F\smallsetminus F'$. For each $C\in C_x$, list the elements of $C\cap (F\smallsetminus F')$ as $g_{C, 1}, \dots, g_{C, k_C}$ in the order appearing in $C$. Then
$C$ is
	\begin{align*}
\bigl(g_{C, 1}, &\dots, (\varphi ^{(x)})^{n_{x, g_{C, 1}}-1}(g_{C, 1}), \psi _{F'}\circ (\varphi ^{(x)})^{n_{x, g_{C, 1}}}(g_{C, 1})=
	\\
&g_{C, 2}, \dots, (\varphi ^{(x)})^{n_{x, g_{C, 2}}-1}(g_{C, 2}),
\dots, g_{C, k_C}, \dots, (\varphi ^{(x)})^{n_{x, g_{C, k_C}}-1}(g_{C, k_C})\bigr).
	\end{align*}
Thus
	\begin{align*}
(-1)^{|C|-1}\prod _{t\in C}f_{x_t}&=(-1)^{-1+\sum_{i=1}^{k_C}n_{x, g_{C, i}}}\prod _{i=1}^{k_C}\prod _{j=0}^{n_{x, g_{C, i}}-1}f_{x_{(\varphi ^{(x)})^j(g_{C_, i})}}
	\\
&=(-1)^{-1+\sum_{g\in C\cap (F\smallsetminus F')}n_{x, g}}\prod _{g\in C\cap (F\smallsetminus F')}\prod _{j=0}^{n_{x, g}-1}f_{x_{(\varphi ^{(x)})^j(g)}}.
	\end{align*}
Then
	\begin{align*}
\prod _{C\in C_x}(-1)^{|C|-1}\prod _{t\in C}f_{x_t}&=\prod _{C\in C_x}(-1)^{-1+\sum_{g\in C\cap (F\smallsetminus F')}n_{x, g}}\prod _{g\in C\cap (F\smallsetminus F')}\prod _{j=0}^{n_{x, g}-1}f_{x_{(\varphi ^{(x)})^j(g)}}
	\\
&=(-1)^{|C_x|}\prod _{g\in F\smallsetminus F'}(-1)^{n_{x, g}}\prod _{j=0}^{n_{x, g}-1}f_{x_{(\varphi ^{(x)})^j(g)}}.
	\end{align*}
By condition (1) any cycle of $\psi _{F'}\circ \varphi ^{(x)}$ not contained in $C_x$ is trivial.
By condition (2) the sign of $\sgn(\psi _{F'}\circ \varphi ^{(x)})\prod _{t\in F}f_{x_t}$ is the same as the sign of $\prod _{C\in C_x}(-1)^{|C|-1}\prod _{t\in C}f_{x_t}$, and hence is the same as the sign of $(-1)^{|C_x|}\prod _{g\in F\smallsetminus F'}(-1)^{n_{x, g}}\prod _{j=0}^{n_{x, g}-1}f_{x_{(\varphi ^{(x)})^j(g)}}$.
By condition (4) the sign of $\prod _{g\in F\smallsetminus F'}(-1)^{n_{x, g}}\prod _{j=0}^{n_{x, g}-1}f_{x_{(\varphi ^{(x)})^j(g)}}$ is the same for all $x\in X_{\mathsf{A},F, F'} $. By condition (3) the partition of $F\smallsetminus F'$ consisting of $C\cap (F\smallsetminus F')$ for $C\in C_x$ is the same for all $x\in X_{\mathsf{A}, F, F'} $, thus $|C_x|$ is the same for all $x\in X_{\mathsf{A}, F, F'} $. Therefore the sign of $\sgn(\psi _{F'}\circ \varphi ^{(x)})\prod _{t\in F}f_{x_t}$ is the same for all $x\in X_{\mathsf{A}, F, F'} $.
	\end{proof}

	\begin{example}
	\label{E-1+x+y+xy}
Let $\Gamma=\mathbb{Z}^2$ equipped with the lexicographic order $(n_1, n_2)>(m_1, m_2)$ if $n_1>m_1$ or $n_1=m_1$ and $n_2>m_2$. Let $f=a-bu_1+cu_2+du_1u_2\in \mathbb{R}\Gamma=\mathbb{R}[u_1^{\pm}, u_2^{\pm}]$ for $a, b, c, d>0$. Put $\mathsf{A}=\supp(f)=\{(0, 0), (0, 1), (1, 0), (1, 1)\}\in \mathscr{F}(\Gamma)$. Then the conditions (1) and (2) of Lemma~\ref{L-order} hold. Let $F=\{0, \dots, n\}\times \{0, \dots, m\}\in \mathscr{F}(\Gamma)$ for $n, m\ge 100$, and let $F'\in \Theta(\mathsf{A}, F)$. We claim that the conditions (3) and (4) of Lemma~\ref{L-order} hold. Let $x\in X_{\mathsf{A}, F, F'} $. From the shape of $\mathsf{A}$ we know that for any distinct $g, h\in F\smallsetminus F'$, the paths $g, \varphi ^{(x)}(g), \dots, (\varphi ^{(x)})^{n_{x, g}}(g)$ and $h, \varphi ^{(x)}(h), \dots, (\varphi ^{(x)})^{n_{x, h}}(h)$ do not cross each other. Since $F\smallsetminus F'\subseteq \{(0, k)\colon 0\le k\le m\}\cup \{(k, 0)\colon 0\le k\le n\}$ and $F'\smallsetminus F\subseteq \{(n+1, k)\colon 0\le k\le m+1\}\cup \{(k, m+1)\colon 0\le k\le n+1\}$, we conclude that condition (3) of Lemma~\ref{L-order} holds. Let $g=(i, k)\in F\smallsetminus F'$ and put $(\varphi ^{(x)})^{n_{x, g}}(g)=(i', k')\in F'\smallsetminus F$.
For each $s\in \mathsf{A}$, denote by $n_s$ the number of $0\le j\le n_{x, g}-1$ satisfying $x_{(\varphi ^{(x)})^j(g)}=s$.
Then $n_{(0, 1)}+n_{(1, 1)}=k'-k$, $n_{(1, 0)}+n_{(1, 1)}=i'-i$, and $n_{(0, 1)}+n_{(1, 0)}+n_{(1, 1)}=n_{x, g}$.
Note that the sign of $(-1)^{n_{x, g}}\prod _{j=0}^{n_{x, g}-1}f_{x_{(\varphi ^{(x)})^j(g)}}$ is the same as the sign of $(-1)^{n_{x, g}}(-1)^{n_{(1, 0)}}$. But
	\begin{displaymath}
(-1)^{n_{x, g}}(-1)^{n_{(1, 0)}}=(-1)^{n_{x, g}-n_{(1, 0)}}=(-1)^{n_{(0, 1)}+n_{(1, 1)}}=(-1)^{k'-k}.
	\end{displaymath}
Thus condition (4) of Lemma~\ref{L-order} holds. This proves our claim. Now from Lemmas~\ref{L-order} and \ref{L-per equal det} we get $\per (|f|)=\ddet _\textup{FK}(f)$.

For any $a, b, c, d>0$, we also have
	\begin{equation}
	\label{E-1+x+y+xy1}
	\begin{aligned}
\ddet _\textup{FK}({a+bu_1+cu_2-du_1u_2})&=\ddet _\textup{FK}({a-bu_1+cu_2+du_1u_2})
	\\
&=\per (a+bu_1+cu_2+du_1u_2).
	\end{aligned}
	\end{equation}
	\end{example}

	\begin{example}
	\label{E-1+xy+xy(-1)-x2}
Let $\Gamma $ be an amenable group with a subgroup $G$ isomorphic to $\mathbb{Z}^2$. Fix an isomorphism $\mathbb{Z}^2\rightarrow G$ and denote by $s$ and $t$ the images of $(1,0)$ and $(0,1)$ respectively. For any $a, b, c, d>0$, from \eqref{E-1+x+y+xy1} we have
	\begin{equation}
	\label{E-1+xy+xy(-1)-x22}
\ddet _{\textup{FK}, G}({a+bs+ct-dst})=\ddet _{\textup{FK}, G}({a-bs+ct+dst})=\per _G(a+bs+ct+dst),
	\end{equation}
where $\ddet _{\textup{FK}, G}$ and $\per _G$ denote the Fuglede-Kadison determinant and the permanent arising from elements in $\mathbb{C} G$ and $\mathbb{R}_{\ge 0} G$ respectively. Let $u\in \Gamma $. By Proposition~\ref{P-basic} we have
	\begin{displaymath}
\per _G(f)=\per _\Gamma(f)=\per _\Gamma(uf)=\per _\Gamma(fu)
	\end{displaymath}
for all $f\in \mathbb{R}_{\ge 0}G$. It is also easy to see from the definition of Fuglede-Kadison determinant that
	\begin{displaymath}
\ddet _{\textup{FK}, G}(f)=\ddet _{\textup{FK}, \Gamma}(f)=\ddet _{\textup{FK}, \Gamma}(uf)=\ddet _{\textup{FK}, \Gamma}(fu)
	\end{displaymath}
for all $f\in \mathbb{C} G$. Therefore, for any $a, b, c, d>0$, from \eqref{E-1+xy+xy(-1)-x22} we get
	\begin{equation}
	\label{E-1+xy+xy(-1)-x23}
	\begin{aligned}
\ddet _{\textup{FK}, \Gamma}({au+bsu+ctu-dstu})&=\ddet _{\textup{FK}, \Gamma}({au-bsu+ctu+dstu})
	\\
&=\per _\Gamma(au+bsu+ctu+dstu).
	\end{aligned}
	\end{equation}

Now let $\Gamma=\mathbb{Z}^2$, and let $G$ be the image of $\mathbb{Z}^2$ under the embedding $\mathbb{Z}^2\hookrightarrow \Gamma $ given by $(1, 0)\mapsto s:=(1, 1)$ and $(0, 1)\mapsto t:=(1, -1)$. Also let $u=(-1, 0)\in \Gamma $. Then \eqref{E-1+xy+xy(-1)-x23} becomes
	\begin{align*}
\ddet _{\textup{FK}, \Gamma}({au_1^{-1}+bu_2+cu_2^{-1}-du_1})&=\ddet _{\textup{FK}, \Gamma}({au_1^{-1}-bu_2+cu_2^{-1}+du_1})
	\\
&=\per _\Gamma ({au_1^{-1}+bu_2+cu_2^{-1}+du_1})
	\end{align*}
for all $a, b, c, d>0$, where $\mathbb{R}\Gamma=\mathbb{R}[u_1^{\pm}, u_2^{\pm}]$.

For any $a, b>0$, a simple calculation yields that
	\begin{align*}
(au_1^{-1}+bu_2+bu_2^{-1}&-au_1)^*(au_1^{-1}+bu_2+bu_2^{-1}-au_1)
	\\
&=2a^2+2b^2-a^2u_1^2-a^2u_1^{-2}+b^2u_2^2+b^2u_2^{-2},
	\end{align*}
so that
	\begin{equation}
	\label{E-1+xy+xy(-1)-x24}
	\begin{aligned}
& \per _\Gamma(au_1^{-1}+bu_2+bu_2^{-1}+au_1)
	\\
&\qquad \qquad =\ddet _{\textup{FK}, \Gamma}({au_1^{-1}+bu_2+bu_2^{-1}-au_1})
	\\
&\qquad \qquad =\frac{1}{2}\ddet _{\textup{FK}, \Gamma}({(au_1^{-1}+bu_2+bu_2^{-1}-au_1)^*(au_1^{-1}+bu_2+bu_2^{-1}-au_1)})
	\\
&\qquad \qquad =\frac{1}{2}\ddet _{\textup{FK}, \Gamma}({2a^2+2b^2-a^2u_1^2-a^2u_1^{-2}+b^2u_2^2+b^2u_2^{-2}})
	\\
&\qquad \qquad =\frac{1}{2}\int_0^1\int_0^1\log(2a^2+2b^2-2a^2\cos(4\pi x)+2b^2\cos(4\pi y))\, dx dy
	\\
&\qquad \qquad =\frac{1}{2}\int_0^1\int_{1/4}^{5/4} \log(2a^2+2b^2-2a^2\cos(4\pi x)+2b^2\cos(4\pi y))\, dx dy
	\\
&\qquad \qquad =\frac{1}{2}\int_0^1\int_0^1\log(2a^2+2b^2-2a^2\cos(4\pi x)-2b^2\cos(4\pi y))\, dx dy
	\\
&\qquad \qquad =\frac{1}{2}\int_0^1\int_0^1\log(2a^2+2b^2-2a^2\cos(2\pi x)-2b^2\cos(2\pi y))\, dx dy
	\\
&\qquad \qquad =\frac{1}{2}\int_0^1\int_0^1\log(2a^2+2b^2-2a^2\cos(\pi x)-2b^2\cos(\pi y))\, dx dy.
	\end{aligned}
	\end{equation}
When $a=b=1$, the penultimate term in \eqref{E-1+xy+xy(-1)-x24} is what appears in \cite{Elimelech21}*{Theorem 6}. The last term in \eqref{E-1+xy+xy(-1)-x24} is (up to a factor of $2$) what appears in \cite{Kasteleyn61}*{(17)}. The growth rate of periodic points of perfect matching is calculated in Section 4 of \cite{Kasteleyn61}, and it is stated on page 1220 there that the rate is given by \cite{Kasteleyn61}*{(17)}.
	\end{example}

	\begin{example}
	\label{E-1+x+y}
Let $\Gamma=\mathbb{Z}^2$, equipped with the lexicographic order $(n_1, n_2)>(m_1, m_2)$ if $n_1>m_1$ or $n_1=m_1$ and $n_2>m_2$. Let $f=a+bu_1+cu_2\in \mathbb{R}\Gamma=\mathbb{R}[u_1^{\pm}, u_2^{\pm}]$ for $a, b, c>0$. Put $\mathsf{A}=\supp(f)=\{(0, 0), (0, 1), (1, 0)\}\in \mathscr{F}(\Gamma)$. Then the conditions (1) and (2) of Lemma~\ref{L-order} hold. Let $F=\{0, \dots, n\}\times \{0, \dots, m\}\in \mathscr{F}(\Gamma)$ for $n, m\ge 100$, and let $F'\in \Theta (\mathsf{A}, F)$. We claim that the conditions (3) and (4) of Lemma~\ref{L-order} hold. Let $x\in X_{\mathsf{A}, F, F'} $. The argument in Example~\ref{E-1+x+y+xy} shows that condition (3) of Lemma~\ref{L-order} holds.
Let $g=(i, k)\in F\smallsetminus F'$ and put $(\varphi ^{(x)})^{n_{x, g}}(g)=(i', k')\in F'\smallsetminus F$.
For each $s\in \mathsf{A}$, denote by $n_s$ the number of $0\le j\le n_{x, g}-1$ satisfying that $x_{(\varphi^{(x)})^j(g)}=s$.
Then $n_{(0, 1)}=k'-k$, $n_{(1, 0)}=i'-i$, and $n_{(0, 1)}+n_{(1, 0)}=n_{x, g}$.
Note that the sign of $(-1)^{n_{x, g}}\prod_{j=0}^{n_{x, g}-1}f_{x_{(\varphi^{(x)})^j(g)}}$ is the same as the sign of $(-1)^{n_{x, g}}$. But
		\begin{displaymath}
(-1)^{n_{x, g}}=(-1)^{n_{(0, 1)}+n_{(1, 0)}}=(-1)^{k'-k+i'-i}.
	\end{displaymath}
Thus condition (4) of Lemma~\ref{L-order} holds. This proves our claim. Now from Lemmas~\ref{L-order} and \ref{L-per equal det} we get that $\per(f)=\ddet_{\rm FK}(f)$.
	\end{example}

When $\Gamma=\mathbb{Z}$, we may assume that $\{0, K\}\subseteq \mathsf{A}=\supp(f)\subseteq A_K=\{0, 1, \dots, K\}$. Put $F_n=\{0, 1, \dots, n\}\in \mathscr{F}(\Gamma)$. When $n>2K+1$ and $x\in X_{\mathsf{A}, F_n}$, the value $|\{t-K\le m<t\colon \varphi ^{(x)}(m)\ge t\}|$ is the same for all integers $K\le t\le n$. Denote this number by $a(x)$. In fact, $a(x)=|\varphi ^{(x)}(F_n)\cap [t, \infty)|-(n-t+1)$ for every integer $K\le t\le n$. Thus $a(x)$ is determined by $\varphi ^{(x)}(F_n)$. Then we may write $a(x)$ as $a(\varphi ^{(x)}(F_n))$. For each $F'\in \Theta(\mathsf{A}, F_n)$, take $\psi _ {F'}$ to be the unique increasing bijection $F'\rightarrow F_n$. Then for each $F'\in \Theta(\mathsf{A}, F_n)$ and $x\in X_{\mathsf{A}, F_n, F'} $, putting $b(x)$ to be the number of $(s, t)\in F_n\times F_n$ satisfying $s<t$ and $\varphi ^{(x)}(s)>\varphi ^{(x)}(t)$, one has $\sgn(\psi _{F'}\circ \varphi ^{(x)})=(-1)^{b(x)}$. For each $t\in F_n$, denote by $b_t(x)$ the number of $s\in F_n\cap [t-K, t-1]$ satisfying $\varphi ^{(x)}(s)>\varphi ^{(x)}(t)$. Then $b(x)=\sum_{t\in F_n}b_t(x)$.

	\begin{example}
	\label{E-1+x+x2}
Let $\Gamma=\mathbb{Z}$ and $f=au^2+bu-c\in \mathbb{R}\Gamma=\mathbb{R}[u^{\pm}] $ for $a, b, c\in \mathbb{R}_{>0}$. Put $\mathsf{A}=\supp(f)=\{0, 1, 2\}\in \mathscr{F}(\Gamma)$. We claim that $W=\{f\}$ satisfies the condition in Lemma~\ref{L-per equal det} for $F_n=\{0, 1, \dots, n\}$ and $H_n=\{0, 1, \dots, 20\}\cup \{n-20, \dots, n\}$. Let $n\ge 100$. Put $V_n=F_n\cap [10, n-10]$. Fix $F'\in \Theta(\mathsf{A}, F_n)$ and $y\in X_{\mathsf{A}, H_n}^\iota $ satisfying $\varphi ^{(y)}(H_n)\subseteq F'$. Let $x\in X_{\mathsf{A}, F_n, F'} $ be equal to $y$ on $H_n$. When $a(F')=0$, one has $b_t(x)=0$ and $x_t=0$ for all $t\in V_n$; $b_t(x)=1$ when $a(F')=1$, $x_t=0$, and $t\in V_n$; $b_t(x)=0$ when $a(F')=1$, $x_t\ge 1$, and $t\in V_n$; when $a(F')=2$, one has $b_t(x)=0$ and $x_t=2$ for all $t\in V_n$. Thus when $a(F')=1$ or $2$, the sign of $\sgn(\psi _{F'}\circ \varphi ^{(x)})\prod _{t\in F_n}f_{x_t}$ is $(-1)^{b(x)+|x^{-1}(0)|}=(-1)^{|x^{-1}(0)\smallsetminus V_n|+\sum_{t\in H_n\smallsetminus V_n}b_t(x)}$, which is determined by $y$. When $a(F')=0$, the sign of $\sgn(\psi _{F'}\circ \varphi ^{(x)})\prod _{t\in F_n}f_{x_t}$ is $(-1)^{b(x)+|x^{-1}(0)|}=(-1)^{|x^{-1}(0)\smallsetminus V_n|+|V_n|+\sum_{t\in H_n\smallsetminus V_n}b_t(x)}$, which is determined by $y$.
From Lemma~\ref{L-per equal det} we conclude that $\per (|f|)=\ddet _\textup{FK}(f)$.
	\end{example}

	\begin{example}
	\label{E-3 points}
Let $\Gamma=\mathbb{Z}$ and $K\ge 2$. Let $h=au^K+bu^{K-1}+c\in \mathbb{R}\Gamma=\mathbb{R}[u^{\pm}] $ for $a, b, c\in \mathbb{R}_{>0}$. Put $\mathsf{A}=\supp(h)=\{0, K-1, K\}\in \mathscr{F}(\Gamma)$. We claim that $W=\{f^{(1)}=au^K+bu^{K-1}-c, f^{(2)}=au^K+bu^{K-1}+c\}$ satisfies the condition in Lemma~\ref{L-per equal det} for $F_n=\{0, 1, \dots, n\}$ and $H_n=\{0, 1, \dots, 4K\}\cup \{n-4K, \dots, n\}$. Let $n\ge 100K$. Put $V_n=F_n\cap [2K, n-2K]$. Fix $F'\in \Theta(\mathsf{A}, F_n)$ and $y\in X_{\mathsf{A}, H_n}^\iota $ satisfying $\varphi ^{(y)}(H_n)\subseteq F'$. Let $x\in X_{\mathsf{A}, F_n, F'} $ equal to $y$ on $H_n$. Then $b_t(x)=a(F')$ when $x_t=0$ and $t\in V_n$; $b_t(x)=0$ when $x_t\neq 0$ and $t\in V_n$. Thus when $a(F')$ is odd, the sign of
	\begin{displaymath}
\sgn(\psi _{F'}\circ \varphi ^{(x)})\prod _{t\in F_n}f_{x_t}^{(1)}
	\end{displaymath}
is equal to
	\begin{displaymath}
(-1)^{b(x)+|x^{-1}(0)|}=(-1)^{|x^{-1}(0)\smallsetminus V_n|+\sum_{t\in H_n\smallsetminus V_n}b_t(x)},
	\end{displaymath}
which is determined by $y$. When $a(F')$ is even, the sign of \linebreak[2]$\sgn(\psi _{F'}\circ \varphi ^{(x)})\prod _{t\in F_n}f_{x_t}^{(2)}$ is $(-1)^{b(x)}\linebreak[0]=(-1)^{\sum_{t\in H_n\smallsetminus V_n}b_t(x)}$, which is determined by $y$.
From Lemma~\ref{L-per equal det} we conclude that $\per (h)=\max(\ddet _\textup{FK}(f^{(1)}), \ddet _\textup{FK}(f^{(2)}))$.
	\end{example}

In Example~\ref{E-3 points}, could $\ddet _\textup{FK}(f^{(1)})<\ddet _\textup{FK}(f^{(2)})$ for some $a, b, c>0$?

	\begin{example}
	\label{E-4 points}
Let $\Gamma=\mathbb{Z}$ and $K\ge 3$. Let $h=au^K+bu^{K-1}+cu+d\in \mathbb{R}\Gamma=\mathbb{R}[u^{\pm}] $ for $a, b, c, d\in \mathbb{R}_{>0}$. Put $\mathsf{A}=\supp(h)=\{0, 1, K-1, K\}\in \mathscr{F}(\Gamma)$. We claim that $W=\{f^{(1)}=au^K+bu^{K-1}+cu-d, f^{(2)}=au^K+bu^{K-1}-cu+d\}$ satisfies the condition in Lemma~\ref{L-per equal det} for $F_n=\{0, 1, \dots, n\}$ and $H_n=\{0, 1, \dots, 4K\}\cup \{n-4K, \dots, n\}$. Let $n\ge 100K$. Put $V_n=F_n\cap [2K, n-2K]$. Fix $F'\in \Theta(\mathsf{A}, F_n)$ and $y\in X_{\mathsf{A}, H_n}^\iota $ satisfying $\varphi ^{(y)}(H_n)\subseteq F'$. Let $x\in X_{\mathsf{A}, F_n, F'}$ equal to $y$ on $H_n$. Then $b_t(x)=a(F')$ when $x_t=0$ and $t\in V_n$; $b_t(x)=a(F')-1$ when $x_t=1$ and $t\in V_n$; $b_t(x)=0$ when $x_t\ge K-1$ and $t\in V_n$. Thus when $a(F')$ is odd, the sign of $\sgn(\psi _{F'}\circ \varphi ^{(x)})\prod _{t\in F_n}f_{x_t}^{(1)}$ is $(-1)^{b(x)+|x^{-1}(0)|}=(-1)^{|x^{-1}(0)\smallsetminus V_n|+\sum_{t\in H_n\smallsetminus V_n}b_t(x)}$, which is determined by $y$. When $a(F')$ is even, the sign of $\sgn(\psi _{F'}\circ \varphi ^{(x)})\prod _{t\in F_n}f_{x_t}^{(2)}$ is $(-1)^{b(x)+|x^{-1}(1)|}=(-1)^{|x^{-1}(1)\smallsetminus V_n|+\sum_{t\in H_n\smallsetminus V_n}b_t(x)}$, which is determined by $y$.
From Lemma~\ref{L-per equal det} we conclude that $\per (h)=\max(\ddet _\textup{FK}(f^{(1)}), \ddet _\textup{FK}(f^{(2)}))$.
	\end{example}

In Example~\ref{E-4 points}, could $\ddet _\textup{FK}(f^{(1)})<\ddet _\textup{FK}(f^{(2)})$ for some $a, b, c, d>0$?

\section{Approximation by finite quotient groups}
	\label{S-approximation}

In this section we discuss briefly finite approximations of permanents $\per (f)$, $f\in \mathbb{R}_{\ge0}\Gamma $, for residually finite amenable groups.

Let $\Gamma $ be amenable, $f\in \mathbb{R}_{\ge0}\Gamma $, $\mathsf{A}\coloneqq \textup{supp}(f)\in \mathscr{F}(\Gamma )$, and let $\Delta \subset \Gamma $ be a finite-index normal subgroup such that $(\mathsf{A}\mathsf{A}^{-1})\cap \Delta =\{\matheur{e}_\Gamma \}$. We denote by $\pi _{\Gamma /\Delta }\colon \Gamma \to \Gamma /\Delta $ the quotient map $s\mapsto s\Delta $ as well as the induced homomorphism $\mathbb{C}\Gamma\rightarrow \mathbb{C} (\Gamma/\Delta )$.

We define $X_\mathsf{A}\subseteq \mathsf{A}^\Gamma $ by \eqref{eq:perm} -- \eqref{eq:XA}, set $\tilde{\mathsf{A}}=\pi _{\Gamma /\Delta }(\mathsf{A})\subseteq \Gamma /\Delta $, and denote by $\tilde{\mathscr{S}}_{\tilde{\mathsf{A}}}\subseteq \mathscr{S}(\Gamma /\Delta )$ the set of all permutations $\tilde{\sigma }$ of $\Gamma /\Delta $ with
	\begin{displaymath}
\tilde{x}_{\tilde{s}}= \tilde{s}^{-1}\tilde{\sigma }(\tilde{s})\in \tilde{\mathsf{A}}
	\end{displaymath}
for every $\tilde{s}\in \Gamma /\Delta $. We define $X_{\tilde{\mathsf{A}}}\subseteq \tilde{\mathsf{A}}^{\Gamma /\Delta }$ exactly as in \eqref{eq:perm} -- \eqref{eq:XA}, with $\Gamma $ and $\mathsf{A}$ replaced by $\Gamma /\Delta $ and $\tilde{\mathsf{A}}$, respectively. For every $y\in X_{\tilde{\mathsf{A}}}$, the composition $y\circ\pi _{\Gamma /\Delta }$ lies in $X_\mathsf{A}$ (actually it is an element of $\tilde{\mathsf{A}}^\Gamma $, but since the restriction of $\pi _{\Gamma /\Delta }$ to $\mathsf{A}$ is injective by assumption we can abuse notation and treat it as an element of $X_\mathsf{A}\subseteq \mathsf{A}^\Gamma $ which is invariant under $\lambda ^t$ for every $t\in \Delta $ (cf. \eqref{eq:shift}). Conversely, every $x'\in \textup{Fix}_\Delta (X_\mathsf{A})=\{x\in X_\mathsf{A}\colon \lambda ^tx=x\enspace \textup{for every}\enspace t\in \Delta \}$ arises in this manner. It follows that
	\begin{displaymath}
|\textup{Fix}_\Delta (X_\mathsf{A})| = |X_{\tilde{\mathsf{A}}}|
	\end{displaymath}
and
	\begin{displaymath}
\per (\pi _{\Gamma /\Delta }(f)) = \frac{1}{|\Gamma /\Delta |}\log \Biggl( \sum_{y\in X_{\tilde{\mathsf{A}}}} \,\prod_{\tilde{s}\in \Gamma /\Delta }\pi _{\Gamma /\Delta }(f)_{y_{\tilde{s}}}\Biggr).
	\end{displaymath}

Let $\Gamma $ be amenable and residually finite, and let $(\Delta _n)_{n\in \mathbb{N}}$ be a sequence of finite-index normal subgroups such that $\bigcap_m\bigcup_{n\ge m}\Delta _n=\{\matheur{e}_\Gamma \}$. By \cite{Weiss} there is a right F{\o}lner sequence $(F_n)_{n\in \mathbb{N}}$ of $\Gamma $ such that $\pi_{\Gamma /\Delta _n}$ maps $F_n$ bijectively to $\Gamma /\Delta _n$ for every $n$ and $\mathsf{A}\subset F_n$ for every $n\ge N$, say.

	\begin{proposition}
	\label{P-approximation lower bound}
Let $f\in \mathbb{R}_{\ge0}\Gamma $. Then $\per (f)\ge \limsup_{n\to \infty}\per (\pi _{\Gamma /\Delta _n}(f))$.
	\end{proposition}

	\begin{proof}
We assume that $f\ne 0$ and set $\textsf{A}=\textup{supp}(f)\in \mathscr{F}(\Gamma )$.

For every $n\ge1$ we put $\mathsf{A}^{(n)} = \pi _{\Gamma /\Delta _n}(\mathsf{A})$ and $f^{(n)} = \pi _{\Gamma /\Delta _n}(f)$. For every $y\in X_{\mathsf{A}^{(n)}}\subset (\mathsf{A}^{(n)})^{\Gamma /\Delta _n}$, the composition $y\circ \pi _{\Gamma /\Delta _n}$ lies in $(\mathsf{A}^{(n)})^\Gamma $. If $n \ge N$, we may view $y\circ \pi _{\Gamma /\Delta _n}$ as an element of $\mathsf{A}^\Gamma $ whose restriction to $F_n$ lies in $X_{\mathsf{A},F_n}$. In the notation of \eqref{eq:permanent} we obtain that
	\begin{displaymath}
\per _{\mathsf{A},F_n}(f)=\sum_{x\in X_{\mathsf{A},F_n}}\prod _{s\in F_n}f_{x_s} \ge \sum_{y\in X_{\mathsf{A}^{(n)}}} \prod _{\tilde{t}\in \Gamma /\Delta _n} f^{(n)}_{y_{\tilde{t}}}.
	\end{displaymath}
According to Definition \ref{D-permanent},
	\begin{align*}
\per (f)=\lim_{n\to \infty}\frac{1}{|F_n|}\log \per _{\mathsf{A},F_n}(f) &\ge \limsup_{n\to \infty}\frac{1}{|\Gamma/\Delta _n|}\log \Biggl( \sum_{y\in X_{\mathsf{A}^{(n)}}} \prod _{\tilde{t}\in \Gamma /\Delta _n} f^{(n)}_{y_{\tilde{t}}}\Biggr)
	\\
&=\limsup_{n\to \infty}\,\per (f^{(n)}).\tag*{\qedsymbol}
	\end{align*}\renewcommand{\qedsymbol}{}
	\vspace{-\baselineskip}
	\end{proof}

For $\Gamma =\mathbb{Z}$, $f\in \mathbb{R}_{\ge0}\mathbb{Z}$, and $\Delta _n = n\mathbb{Z}$, the inequality $\per(f)\ge \limsup_{n\to \infty}\per(\pi_{\mathbb{Z}/n\mathbb{Z}}(f))$ in Proposition \ref{P-approximation lower bound} is actually an equality. This is a consequence of the following classical result, essentially taken from \cite{LM}*{Theorem 4.3.6. and Exercise 4.4.3.}

	\begin{proposition}
	\label{P-approximation SFT}
Let $\mathsf{B}$ be a nonempty finite set, $Y\subseteq \mathsf{B}^\mathbb{Z}$ a shift of finite type {\textup(}as defined in Subsection \ref{SS-local subshift}{\textup)}, and let $\lambda $ be the {\textup(}left{\textup)} shift action \eqref{eq:shift} of $\mathbb{Z}$ on $\mathsf{B}^\mathbb{Z}$. For every $N\ge1$ we denote by $\textup{Fix}_N(Y)=\{y\in Y\colon \lambda ^Ny=y\}$ the set of periodic points in $Y$ with period $N$. Then
	\begin{align*}
\lim_{N\to\infty } \frac1N \,&\log \Biggl( \sum_{y\in \pi _{\{0,\dots ,N-1\}}(Y)} \prod_{k=0}^{N-1}\phi (y_k) \Biggr) = \limsup_{N\to\infty } \frac1N \,\log \Biggl( \sum_{y\in \textup{Fix}_N(Y)} \prod_{k=0}^{N-1}\phi (y_k)\Biggr)
	\end{align*}
for every map $\phi \colon \mathsf{B}\to \mathbb{R}_{>0}$.
	\end{proposition}

By applying Proposition \ref{P-approximation SFT} to the \emph{SFT} $X_\mathsf{A}\subset \mathsf{A}^\mathbb{Z}$ and the function $f\colon \mathsf{A}\to \mathbb{R}_{>0}$ we obtain the identity
	\begin{align*}
\lim_{|J|\to\infty } \frac1{|J|} \,\log \Biggl( \sum_{x\in X_{\mathsf{A},J}} \prod_{k\in J}f(x_k)\Biggr) &=\lim_{|J|\to\infty } \frac1{|J|} \,\log \Biggl( \sum_{x\in \pi _J(X_\mathsf{A})} \prod_{k\in J}f(x_k)\Biggr)
	\\
&=\limsup_{N\to\infty } \frac1{N} \,\log \Biggl(\sum_{x\in \textup{Fix}_N(X_\mathsf{A})} \prod_{k=0}^{N-1}f(x_k)\Biggr),
	\end{align*}
where the limits are taken over increasing intervals $J=[i,\dots ,j]\subset \mathbb{Z}$, and where $X_{\mathsf{A},J}\subset \mathsf{A}^J$ is defined as in \eqref{eq:XAF} or \eqref{eq:locally admissible}. As indicated in Remark \ref{R-real}, $\per(f)$ coincides with the topological pressure $P(X_\mathsf{A},\log f)$ with $\mathsf{A}=\supp(f)$, which gives us equality in Proposition \ref{P-approximation lower bound}.

	\begin{corollary}
	\label{C-approximation lower bound}
Let $f\in \mathbb{R}_{\ge0}\mathbb{Z}$. Then $\per (f)= \limsup_{n\to \infty}\per (\pi _{\mathbb{Z}/n\mathbb{Z}}(f))$.
	\end{corollary}

Even for $\Gamma =\mathbb{Z}^2$, the analogue of Corollary \ref{C-approximation lower bound} is not known except under very restrictive conditions. We refer to \cite{SS16}*{\S\,6} and \cite{Elimelech21} for background and a few examples.

\section{Bounds for permanents}
	\label{S-bounds}

Although the problem of calculating of $\htop (X_\mathsf{A}), \mathsf{A}\in \mathscr{F}(\Gamma )$, doesn't seem tractable in general, some of the known results on permanents of matrices lead to useful estimates on the topological entropies of these spaces. For general reference on permanents we refer to \cite{Minc78} and \cite{VW01}*{Chapters 11 and 12}.

We recall that the \textsl{permanent} of a matrix $B=(b_{ij})_{1\le i\le m, 1\le j\le n}\in M_{m, n}(\mathbb{C})$ with $m\le n$ is defined as
	\begin{displaymath}
\per (B)=\sum_{\sigma } \prod _{i=1}^m b_{i\sigma (i)},
	\end{displaymath}
where $\sigma $ ranges over all injective maps $\{1, \dots, m\}\to \{1, \dots, n\}$. Note that the permanent does not change if we exchange two rows or two columns of $B$. Thus we can talk about $\per (B)$ for any $B\in M_{F_1, F_2}(\mathbb{C})$ where $F_1$ and $F_2$ are nonempty finite sets with $|F_1|\le |F_2|$.

A matrix $B\in M_n(\mathbb{R})$ is called \textsl{doubly stochastic} if all entries are nonnegative and all row-sums and column-sums are equal to $1$. We denote by $\Omega_n$ the set of all $n\times n$ doubly stochastic matrices.

We shall need the following result \cite{VW01}*{Theorem 12.1}, known as \textsl{the van der Waerden conjecture}, proved independently by Falikman (without the uniqueness part) \cite{Falikman} and Egoryčev \cite{Egorycev}.

	\begin{theorem}[van der Waerden Conjecture]
	\label{T-van der Waerden}
For any $n\in \mathbb{N}$ and $B\in \Omega_n$, one has
	\begin{displaymath}
\per (B)\ge \frac{n!}{n^n},
	\end{displaymath}
and the equality holds exactly when all entries of $B$ are $\frac{1}{n}$.
	\end{theorem}

Let $\Gamma $ be a discrete group. Let $f\in \mathbb{R}_{\ge 0}\Gamma $ and $\mathsf{A}, F\in \mathscr{F}(\Gamma)$. We denote by $B_{F, F\mathsf{A}, f}$ the matrix in $M_{F, F\mathsf{A}}(\mathbb{R})$ with value $f_s$ at $(t, ts)$ for all $t\in F$ and $s\in \mathsf{A}$. In the notation of \eqref{eq:permanent} we have
	\begin{displaymath}
\per (B_{F, F\mathsf{A}, f})=\per _{\mathsf{A}, F}^\iota (f).
	\end{displaymath}

Now assume that $\Gamma $ is infinite and amenable.

	\begin{lemma}
	\label{L-van der Waerden}
Let $f\in \mathbb{R}_{\ge 0}\Gamma $ with $\|f\|_1=1$. Then
	\begin{displaymath}
\iper(f)\ge \log \frac{1}{e},
	\end{displaymath}
where $e=2.71828 \dots $ is the Euler number.
	\end{lemma}

	\begin{proof}
We fix an $\mathsf{A}\in \mathscr{F}(\Gamma)$ containing $\supp(f)\cup \{\matheur{e}_\Gamma\}$.

Let $F\in \mathscr{F}(\Gamma)$. We define a matrix $C_{F\mathsf{A}, f}\in M_{F\mathsf{A}}(\mathbb{R})$ as follows. For each $s\in \supp(f)$, we take a bijection $\sigma _s\colon F\mathsf{A}\rightarrow F\mathsf{A}$ such that $\sigma _s(t)=ts$ for all $t\in F$. For any $t_1, t_2\in F\mathsf{A}$, we set
	\begin{displaymath}
(C_{F\mathsf{A}, f})_{t_1, t_2}=\sum_{s\in \supp(f),\, \sigma _s(t_1)=t_2} f_s.
	\end{displaymath}
Then
	\begin{displaymath}
(C_{F\mathsf{A}, f})_{t_1, t_2}=(B_{F, F\mathsf{A}, f})_{t_1, t_2}
	\end{displaymath}
for all $t_1\in F$ and $t_2\in F\mathsf{A}$. That is, $B_{F, F\mathsf{A}, f}$ is a submatrix of $C_{F\mathsf{A}, f}$.

For any $t_2\in F\mathsf{A}$, we have
	\begin{align*}
\sum_{t_1\in F\mathsf{A}}(C_{F\mathsf{A}, f})_{t_1, t_2}=\sum_{t_1\in F\mathsf{A}}\,\sum_{s\in \supp(f),\, \sigma _s(t_1)=t_2} f_s=\sum_{s\in \supp(f)}f_s=1.
	\end{align*}
For any $t_1\in F\mathsf{A}$, we have
	\begin{align*}
\sum_{t_2\in F\mathsf{A}}(C_{F\mathsf{A}, f})_{t_1, t_2}=\sum_{t_2\in F\mathsf{A}}\,\sum_{s\in \supp(f), \, \sigma _s(t_1)=t_2} f_s=\sum_{s\in \supp(f)}f_s=1.
	\end{align*}
Thus $C_{F\mathsf{A}, f}$ is a doubly stochastic matrix. From Theorem~\ref{T-van der Waerden} we have
	\begin{displaymath}
\per (C_{F\mathsf{A}, f})\ge \frac{|F\mathsf{A}|!}{|F\mathsf{A}|^{|F\mathsf{A}|}}.
	\end{displaymath}

Denote by $\Phi$ the set of injective maps $\phi\colon F\rightarrow F\mathsf{A}$, and by $\Psi$ the set of bijections $\psi\colon F\mathsf{A}\rightarrow F\mathsf{A}$. For each $\phi\in \Phi$, denote by $\Psi_\phi$ the set of $\psi\in \Psi$ extending $\phi$, and by $\Psi'_\phi$ the set of $\psi\in \Psi_\phi$ satisfying $(C_{F\mathsf{A}, f})_{t, \psi(t)}>0$ for all $t\in F\mathsf{A}\smallsetminus F$. For each $t\in F\mathsf{A}\smallsetminus F$, the $t$-row of $C_{F\mathsf{A}, f}$ has at most $|\supp(f)|$ positive entries. Thus
	\begin{displaymath}
|\Psi'_\phi|\le |\supp(f)|^{|F\mathsf{A}\smallsetminus F|}
	\end{displaymath}
for every $\phi\in \Phi$. Now
	\begin{align*}
\per (C_{F\mathsf{A}, f})&=\sum_{\psi\in \Psi}\prod _{t\in F\mathsf{A}}(C_{F\mathsf{A}, f})_{t, \psi(t)}=\sum_{\phi\in \Phi}\sum_{\psi\in \Psi_\phi}\prod _{t\in F\mathsf{A}}(C_{F\mathsf{A}, f})_{t, \psi(t)}
	\\
&=\sum_{\phi\in \Phi}\sum_{\psi\in \Psi'_\phi}\prod _{t\in F\mathsf{A}}(C_{F\mathsf{A}, f})_{t, \psi(t)}
	\\
&=\sum_{\phi\in \Phi}\sum_{\psi\in \Psi'_\phi}\left(\prod _{t\in F}(B_{F, F\mathsf{A}, f})_{t, \phi(t)}\right)\left(\prod _{t\in F\mathsf{A}\smallsetminus F}(C_{F\mathsf{A}, f})_{t, \psi(t)}\right)
	\\
&\le \sum_{\phi\in \Phi}\sum_{\psi\in \Psi'_\phi}\prod _{t\in F}(B_{F, F\mathsf{A}, f})_{t, \phi(t)}= \sum_{\phi\in \Phi}|\Psi'_\phi|\prod _{t\in F}(B_{F, F\mathsf{A}, f})_{t, \phi(t)}
	\\
&\le |\supp(f)|^{|F\mathsf{A}\smallsetminus F|}\sum_{\phi\in \Phi}\prod _{t\in F}(B_{F, F\mathsf{A}, f})_{t, \phi(t)}
	\\
&=|\supp(f)|^{|F\mathsf{A}\smallsetminus F|}\per (B_{F, F\mathsf{A}, f})=|\supp(f)|^{|F\mathsf{A}\smallsetminus F|}\per _{\mathsf{A}, F}^\iota (f).
	\end{align*}
Thus
	\begin{displaymath}
\per _{\mathsf{A}, F}^\iota (f)\ge |\supp(f)|^{-|F\mathsf{A}\smallsetminus F|}\frac{|F\mathsf{A}|!}{|F\mathsf{A}|^{|F\mathsf{A}|}}.
	\end{displaymath}
Since $\Gamma $ is infinite, $|F|\to \infty$ as $F$ becomes more and more right invariant.
Therefore, using Stirling's approximation formula, we get
	\begin{displaymath}
\iper(f)=\lim_F\frac{1}{|F|}\log \per _{\mathsf{A}, F}^\iota (f)\ge \lim_F\frac{1}{|F|}\log \left(|\supp(f)|^{-|F\mathsf{A}\smallsetminus F|}\frac{|F\mathsf{A}|!}{|F\mathsf{A}|^{|F\mathsf{A}|}}\right)=\log \frac{1}{e}.\tag*{\qedsymbol}
	\end{displaymath}\renewcommand{\qedsymbol}{}
	\vspace{-\baselineskip}
	\end{proof}

	\begin{remark}
	\label{R-Friedland}
Prior to the resolution of the van der Waerden conjecture by Falikman and Egoryčev, Friedland \cite{Friedland79} proved a weaker result that
	\begin{displaymath}
\per (B)\ge 1/e^n
	\end{displaymath}
for all $B\in \Omega_n$. In the proof of Lemma~\ref{L-van der Waerden} it suffices to use this weaker result.
	\end{remark}

	\begin{theorem}
	\label{T-lower bound}
For every $f\in \mathbb{R}_{\ge 0}\Gamma $ we have
	\begin{displaymath}
\per (f)\ge \log \frac{\|f\|_1}{e}.
	\end{displaymath}
	\end{theorem}
	\begin{proof}
This is obvious when $f=0$. Thus we may assume $f\neq 0$. We have
	\begin{displaymath}
\per (f)=\iper(f)=\iper\biggl(\frac{f}{\|f\|_1}\biggr)+\log \|f\|_1\ge \log \frac{1}{e}+\log \|f\|_1=\log \frac{\|f\|_1}{e},
	\end{displaymath}
where the first equality is from Proposition~\ref{P-per equal iper}, the second equality is from Proposition~\ref{P-basic}.(3),
and the inequality is from Lemma~\ref{L-van der Waerden}.
	\end{proof}

	\begin{corollary}
	\label{C-lower bound}
Let $\mathsf{A}\in \mathscr{F}(\Gamma)$. Then
	\begin{displaymath}
\htop(X_\mathsf{A})\ge \log \frac{|\mathsf{A}|}{e}.
	\end{displaymath}
Furthermore, the following are equivalent:
	\begin{enumerate}
	\item
$\htop(X_\mathsf{A})=0$.
	\item
$|\mathsf{A}|=1$ or $|\mathsf{A}|=2$ with $\mathsf{A}=\{s, t\}$ such that $st^{-1}$ has infinite order.
	\end{enumerate}
	\end{corollary}

	\begin{proof}
Let $f$ be the characteristic function of $\mathsf{A}$. From Remark~\ref{R-permanent and entropy} we have
$\htop(X_\mathsf{A})=\per (f)$.
Then the inequality follows from Theorem~\ref{T-lower bound}.

If $|\mathsf{A}|\ge 3$, then $\htop(X_\mathsf{A})\ge \log \frac{|\mathsf{A}|}{e}>0$.

If $|\mathsf{A}|=1$, then clearly $\htop(X_\mathsf{A})=0$.

Now assume that $|\mathsf{A}|=2$. Say, $\mathsf{A}=\{s, t\}$. Set $\mathsf{A}'=\mathsf{A}t^{-1}=\{\matheur{e}_\Gamma, st^{-1}\}$.
Denote by $G$ the subgroup of $\Gamma $ generated by $st^{-1}$. By parts (1) and (6) of Proposition~\ref{P-basic} we have that
	\begin{displaymath}
\htop(X_\mathsf{A})=\per _\Gamma(f)=\per _\Gamma(ft^{-1})=\per _G(ft^{-1})
	\end{displaymath}
is equal to the topology entropy of the action $G\curvearrowright Y_{\mathsf{A}'}$, where $Y_{\mathsf{A}'}$ is defined similar as $X_{\mathsf{A}'}$ replacing $\Gamma $ by $G$.

If $st^{-1}$ has infinite order, then $G$ is isomorphic to $\mathbb{Z}$, so that $\htop(X_\mathsf{A})=\htop(Y_{\mathsf{A}'})=0$.

If $st^{-1}$ has finite order, then
	\begin{displaymath}
\htop(X_\mathsf{A})=\htop(Y_{\mathsf{A}'})=\frac{1}{|G|}\log |Y_{\mathsf{A}'}|\ge \frac{1}{|G|}\log 2>0. \tag*{\qedsymbol}
	\end{displaymath}\renewcommand{\qedsymbol}{}
	\vspace{-\baselineskip}
	\end{proof}

The case $\Gamma=\mathbb{Z}^d$ of the second assertion of Corollary~\ref{C-lower bound} was proved in \cite{SS16}*{Corollary 5.3}.

We also need the following result of Brègman \cite{Bregman} \cite{Minc78}*{Theorem 6.2.1}, which proves a conjecture of Minc in \cite{Minc63}.

	\begin{theorem}
	\label{T-Minc}
Let $B$ be an $n\times n$ \textup{(0, 1)}-matrix with row sums $r_1, \dots, r_n$. Then
	\begin{displaymath}
\per (B)\le \prod _{i=1}^n(r_i!)^{1/r_i},
	\end{displaymath}
where as convention we set $(0!)^{1/0}=1$.
	\end{theorem}

A finitely generated group $\Gamma $ has \textsl{polynomial growth} if, given a finite generating subset $S$ of $\Gamma $ satisfying $\matheur{e}_\Gamma \in S=S^{-1}$, there is a polynomial $g$ such that $|S^n|\le g(n)$ for all $n\in \mathbb{N}$. This definition does not depend on the choice of $S$. The theorems of Wolf and Gromov state that a finitely generated group has polynomial growth exactly when it is virtually nilpotent.
We also need the following result of Deninger \cite{Deninger}*{Proposition 6.1}.

	\begin{lemma}
	\label{L-log strong}
Let $\Gamma $ be a finitely generated group of polynomial growth. Then there is a right Følner sequence $(F_n)_{n\in \mathbb{N}}$ of $\Gamma $ such that
	\begin{equation}
	\label{E-log strong}
\lim_{n\to \infty}\frac{|F_n K\smallsetminus F_n|}{|F_n|}\log (1+|F_nK \smallsetminus F_n|)=0
	\end{equation}
for all $K\in \mathscr{F}(\Gamma)$.
	\end{lemma}

	\begin{theorem}
	\label{T-upper bound}
Assume that $\Gamma $ is finitely generated of polynomial growth and infinite. Let $\mathsf{A}\in \mathscr{F}(\Gamma)$. Then
	\begin{displaymath}
\htop(X_\mathsf{A})\le \frac{1}{|\mathsf{A}|}\log (|\mathsf{A}|!).
	\end{displaymath}
	\end{theorem}
	\begin{proof}
Replacing $\mathsf{A}$ by $\mathsf{A}t^{-1}$ for some $t\in \mathsf{A}$, we may assume that $\matheur{e}_\Gamma \in \mathsf{A}$. Let $f$ be the characteristic function of $\mathsf{A}$.

Let $F\in \mathscr{F}(\Gamma)$. We denote $\{t\in \Gamma\colon t\mathsf{A}^{-1}\subseteq F\}$ by $\textup{Int}_{\mathsf{A}^{-1}}(F)$. We define a matrix $D_{F\mathsf{A}, f}\in M_{F\mathsf{A}}(\mathbb{R})$ via
	\begin{align*}
(D_{F\mathsf{A}, f})_{t_1, t_2}=
	\begin{cases}
(B_{F, F\mathsf{A}, f})_{t_1, t_2} & \quad \text{if } t_1\in F
	\\
1 & \quad \text{if } t_1\in F\mathsf{A}\smallsetminus F \text{ and } t_2\in F\mathsf{A}\smallsetminus \textup{Int}_{\mathsf{A}^{-1}}(F)
	\\
0 & \quad \text{if } t_1\in F\mathsf{A}\smallsetminus F \text{ and } t_2\in \textup{Int}_{\mathsf{A}^{-1}}(F)
	\\
	\end{cases}
	\end{align*}
for $t_1, t_2\in F\mathsf{A}$.
Then $B_{F, F\mathsf{A}, f}$ is a submatrix of $D_{F\mathsf{A}, f}$. Note that every injective map $\phi\colon F\rightarrow F\mathsf{A}$ extends to a (not necessarily unique) bijections $\psi\colon F\mathsf{A}\rightarrow F\mathsf{A}$. If $\phi(F)\supseteq \textup{Int}_{\mathsf{A}^{-1}}(F)$, then necessarily $(D_{F\mathsf{A}, f})_{t, \psi(t)}=1$ for all $t\in F\mathsf{A}\smallsetminus F$. It follows that
	\begin{displaymath}
\per (D_{F\mathsf{A}, f})\ge |X_{\mathsf{A}, F}|,
	\end{displaymath}
where $X_{\mathsf{A}, F}$ is defined in \eqref{eq:XAF}.

Note that $D_{F\mathsf{A}, f}$ is a (0, 1)-matrix. For each $t\in F$, the sum of the $t$-row of $D_{F\mathsf{A}, f}$ is $|\mathsf{A}|$. For each $t\in F\mathsf{A}\smallsetminus F$, the sum of the $t$-row of $D_{F\mathsf{A}, f}$ is $|F\mathsf{A}\smallsetminus \textup{Int}_{\mathsf{A}^{-1}}(F)|$. By Theorem~\ref{T-Minc} we have
	\begin{displaymath}
\per (D_{F\mathsf{A}, f})\le (|F\mathsf{A}\smallsetminus \textup{Int}_{\mathsf{A}^{-1}}(F)|!)^{|F\mathsf{A}\smallsetminus F|/|F\mathsf{A}\smallsetminus \textup{Int}_{\mathsf{A}^{-1}}(F)|}(|\mathsf{A}|!)^{|F|/|\mathsf{A}|}.
	\end{displaymath}
Note that
	\begin{align*}
|F\mathsf{A}\smallsetminus \textup{Int}_{\mathsf{A}^{-1}}(F)|&\le |F\mathsf{A}\smallsetminus F|+|F\smallsetminus \textup{Int}_{\mathsf{A}^{-1}}(F)|\le |F\mathsf{A}\smallsetminus F|+|\mathsf{A}|\cdot |F\mathsf{A}^{-1}\smallsetminus F|
	\\
&\le (|\mathsf{A}|+1)|F(\mathsf{\mathsf{A}}\cup \mathsf{A}^{-1})\smallsetminus F|.
	\end{align*}
Also note that $((n+1)!)^{1/(n+1)}\ge (n!)^{1/n}$ for all nonnegative integers $n$.
Thus
	\begin{align*}
|X_{\mathsf{A}, F}|&\le (|F\mathsf{A}\smallsetminus \textup{Int}_{\mathsf{A}^{-1}}(F)|!)^{|F\mathsf{A}\smallsetminus F|/|F\mathsf{A}\smallsetminus \textup{Int}_{\mathsf{A}^{-1}}(F)|}(|\mathsf{A}|!)^{|F|/|\mathsf{A}|}
	\\
&\le (((|\mathsf{A}|+1)|F(\mathsf{\mathsf{A}}\cup \mathsf{A}^{-1})\smallsetminus F|)!)^{|F\mathsf{A}\smallsetminus F|/((|\mathsf{A}|+1)|F(\mathsf{\mathsf{A}}\cup \mathsf{A}^{-1})\smallsetminus F|)}(|\mathsf{A}|!)^{|F|/|\mathsf{A}|}
	\\
&\le (((|\mathsf{A}|+1)|F(\mathsf{\mathsf{A}}\cup \mathsf{A}^{-1})\smallsetminus F|)!)^{1/(|\mathsf{A}|+1)} (|\mathsf{A}|!)^{|F|/|\mathsf{A}|}.
	\end{align*}

Since $\Gamma $ has polynomial growth, by Lemma~\ref{L-log strong} we can find a right Følner sequence $(F_n)_{n\in \mathbb{N}}$ satisfying \eqref{E-log strong}. By Proposition~\ref{P-entropy in local} we have
	\begin{align*}
\htop(X_\mathsf{A})&=\lim_{n\to \infty} \frac{1}{|F_n|}\log |X_{\mathsf{A}, F_n}|
	\\
&\le \limsup_{n\to \infty} \frac{1}{|F_n|}\log \Biggl((((|\mathsf{A}|+1)|F_n(\mathsf{A}\cup \mathsf{A}^{-1})\smallsetminus F_n|)!)^{1/(|\mathsf{A}|+1)}(|\mathsf{A}|!)^{|F_n|/|\mathsf{A}|}\Biggr)
	\\
&\le \limsup_{n\to \infty} \frac{1}{|F_n|\cdot (|\mathsf{A}|+1)}\log ((|F_n(\mathsf{A}\cup \mathsf{A}^{-1})\smallsetminus F_n|)!)+ \frac{1}{|\mathsf{A}|}\log(|\mathsf{A}|!)
	\\
&=\frac{1}{|\mathsf{A}|}\log(|\mathsf{A}|!),
	\end{align*}
where in the second inequality we use that $|F_n|\to \infty$ as $n\to \infty$ since $\Gamma $ is infinite, and in the second equality we use Stirling's approximation formula and \eqref{E-log strong}.
	\end{proof}

	\begin{question}
	\label{Q-polynomial growth}
Could we remove the polynomial growth condition in Theorem~\ref{T-upper bound}?
	\end{question}


\begin{bibdiv}
	\begin{biblist}

\bib{Baltic10}{article}{
author={Balti\'{c}, Vladimir},
title={On the number of certain types of strongly restricted
permutations},
journal={Appl. Anal. Discrete Math.},
volume={4},
date={2010},
number={1},
pages={119--135},
}

\bib{BT}{article}{
author={Bollobás, Béla},
author={Thomason, Andrew},
title={Projections of bodies and hereditary properties of hypergraphs},
journal={Bull. London Math. Soc.},
volume={27},
date={1995},
number={5},
pages={417--424},
}

\bib{Boyd}{article}{
author={Boyd, David W.},
title={Mahler's measure and special values of $L$-functions},
journal={Experiment. Math.},
volume={7},
date={1998},
number={1},
pages={37--82},
}

\bib{BT98}{article}{
author={Boyle, Mike},
author={Tomiyama, Jun},
title={Bounded topological orbit equivalence and $C^*$-algebras},
journal={J. Math. Soc. Japan},
volume={50},
date={1998},
number={2},
pages={317--329},
}

\bibitem{Bregman}
L.~M.~Br\`{e}gman. \textit{Certain properties of nonnegative matrices and their permanents.} (Russian) {\it Dokl. Akad. Nauk SSSR} {\bf 211} (1973), 27--30; translation in {\it Soviet Math. Dokl.} {\bf 14} (1973), 945--949.

\bib{CCL}{article}{
author={Ceccherini-Silberstein, Tullio},
author={Coornaert, Michel},
author={Li, Hanfeng},
title={Expansive actions with specification of sofic groups, strong
topological Markov property, and surjunctivity},
journal={J. Funct. Anal.},
volume={286},
date={2024},
number={9},
pages={Paper No. 110376, 26 pp},
}

\bib{CM21}{article}{
author={Chandgotia, Nishant},
author={Meyerovitch, Tom},
title={Borel subsystems and ergodic universality for compact $\mathbb{Z}^d$-systems via specification and beyond},
journal={Proc. Lond. Math. Soc. (3)},
volume={123},
date={2021},
number={3},
pages={231--312},
}

\bib{Chollet}{article}{
author={Chollet, John},
title={Unsolved Problems: Is there a permanental analogue to Oppenheim's
inequality?},
journal={Amer. Math. Monthly},
volume={89},
date={1982},
number={1},
pages={57--58},
}

\bib{Deninger}{article}{
author={Deninger, Christopher},
title={Fuglede-Kadison determinants and entropy for actions of discrete
amenable groups},
journal={J. Amer. Math. Soc.},
volume={19},
date={2006},
number={3},
pages={737--758},
}

\bib{Diaconis99}{article}{
author={Diaconis, Persi},
author={Graham, Ronald},
author={Holmes, Susan P.},
title={Statistical problems involving permutations with restricted
positions},
conference={
title={State of the art in probability and statistics},
address={Leiden},
date={1999},
},
book={
series={IMS Lecture Notes Monogr. Ser.},
volume={36},
publisher={Inst. Math. Statist., Beachwood, OH},
},
isbn={0-940600-50-1},
date={2001},
pages={195--222},
}

\bib{DFR}{article}{
author={Downarowicz, Tomasz},
author={Frej, Bartosz},
author={Romagnoli, Pierre-Paul},
title={Shearer's inequality and infimum rule for Shannon entropy and
topological entropy},
conference={
title={Dynamics and numbers},
},
book={
series={Contemp. Math.},
volume={669},
publisher={Amer. Math. Soc., Providence, RI},
},
date={2016},
pages={63--75},
}

\bib{Edwards15}{article}{
author={Edwards, Kenneth},
author={Allen, Michael A.},
title={Strongly restricted permutations and tiling with fences},
journal={Discrete Appl. Math.},
volume={187},
date={2015},
pages={82--90},
}

\bib{Egorycev}{article}{
author={Egorychev, Georgi\u{i} P.},
title={The solution of van der Waerden's problem for permanents},
journal={Adv. in Math.},
volume={42},
date={1981},
number={3},
pages={299--305},
}

\bib{Elimelech21}{article}{
author={Elimelech, Dor},
title={Permutations with restricted movement},
journal={Discrete Contin. Dyn. Syst.},
volume={41},
date={2021},
number={9},
pages={4319--4349},
}

\bib{Falikman}{article}{
author={Falikman, D. I.},
title={Proof of the van der Waerden conjecture on the permanent of a
doubly stochastic matrix},
language={Russian},
journal={Mat. Zametki},
volume={29},
date={1981},
number={6},
pages={931--938, 957},
}

\bib{Friedland79}{article}{
author={Friedland, Shmuel},
title={A lower bound for the permanent of a doubly stochastic matrix},
journal={Ann. of Math. (2)},
volume={110},
date={1979},
number={1},
pages={167--176},
}

\bib{Friedland97}{article}{
author={Friedland, Shmuel},
title={On the entropy of $\mathbf{Z}^d$ subshifts of finite type},
journal={Linear Algebra Appl.},
volume={252},
date={1997},
pages={199--220},
}

\bib{FK}{article}{
author={Fuglede, Bent},
author={Kadison, Richard V.},
title={Determinant theory in finite factors},
journal={Ann. of Math. (2)},
volume={55},
date={1952},
pages={520--530},
}

\bib{Gibson71}{article}{
author={Gibson, Peter M.},
title={Conversion of the permanent into the determinant},
journal={Proc. Amer. Math. Soc.},
volume={27},
date={1971},
pages={471--476},
}

\bib{HM10}{article}{
author={Hochman, Michael},
author={Meyerovitch, Tom},
title={A characterization of the entropies of multidimensional shifts of
finite type},
journal={Ann. of Math. (2)},
volume={171},
date={2010},
number={3},
pages={2011--2038},
}

\bib{Juschenko13}{article}{
author={Juschenko, Kate},
author={Monod, Nicolas},
title={Cantor systems, piecewise translations and simple amenable groups},
journal={Ann. of Math. (2)},
volume={178},
date={2013},
number={2},
pages={775--787},
}

\bib{Juschenko15}{article}{
author={Juschenko, Kate},
author={de la Salle, Mikael},
title={Invariant means for the wobbling group},
journal={Bull. Belg. Math. Soc. Simon Stevin},
volume={22},
date={2015},
number={2},
pages={281--290},
}

\bib{Juschenko22}{book}{
author={Juschenko, Kate},
title={Amenability of discrete groups by examples},
series={Mathematical Surveys and Monographs},
volume={266},
publisher={American Mathematical Society, Providence, RI},
date={2022},
pages={xi+165},
}

\bib{KR}{article}{
author={Kammeyer, Janet Whalen},
author={Rudolph, Daniel J.},
title={Restricted orbit equivalence for ergodic ${\bf Z}^d$ actions. I},
journal={Ergodic Theory Dynam. Systems},
volume={17},
date={1997},
number={5},
pages={1083--1129},
}

\bib{Kaplansky44}{article}{
author={Kaplansky, Irving},
title={Symbolic solution of certain problems in permutations},
journal={Bull. Amer. Math. Soc.},
volume={50},
date={1944},
pages={906--914},
}

\bib{Kaplansky46}{article}{
author={Kaplansky, Irving},
author={Riordan, John},
title={The problem of the rooks and its applications},
journal={Duke Math. J.},
volume={13},
date={1946},
pages={259--268},
}

\bib{Kasteleyn61}{article}{
author={Kasteleyn, Pieter W.},
title={The statistics of dimers on a lattice I. The number of dimer arrangements on a quadratic lattice}, journal={Physica},
volume={27},
date={1961},
pages={1209--1225},
}

\bib{KL16}{book}{
author={Kerr, David},
author={Li, Hanfeng},
title={Ergodic theory, independence and dichotomies},
series={Springer Monographs in Mathematics},
publisher={Springer, Cham},
date={2016},
pages={xxxiv+431},
}

\bib{Klove09}{article}{
author={Kl\o ve, Torleiv},
title={Generating functions for the number of permutations with limited
displacement},
journal={Electron. J. Combin.},
volume={16},
date={2009},
number={1},
pages={Research Paper 104, 11 pp},
}

\bib{Lehmer70}{article}{
author={Lehmer, Derrick H.},
title={Permutations with strongly restricted displacements},
conference={
title={Combinatorial theory and its applications, I-III},
address={Proc. Colloq., Balatonf\"ured},
date={1969},
},
book={
series={Colloq. Math. Soc. J\'anos Bolyai},
volume={4},
publisher={North-Holland, Amsterdam-London},
},
date={1970},
pages={755--770},
}

\bib{LT}{article}{
author={Li, Hanfeng},
author={Thom, Andreas},
title={Entropy, determinants, and $L^2$-torsion},
journal={J. Amer. Math. Soc.},
volume={27},
date={2014},
number={1},
pages={239--292},
}

\bib{LM}{book}{
author={Lind, Douglas},
author={Marcus, Brian},
title={An introduction to symbolic dynamics and coding},
series={Cambridge Mathematical Library},
edition={2},
publisher={Cambridge University Press, Cambridge},
date={2021},
pages={xix+550},
}

\bib{LSW}{article}{
author={Lind, Douglas},
author={Schmidt, Klaus},
author={Ward, Tom},
title={Mahler measure and entropy for commuting automorphisms of compact
groups},
journal={Invent. Math.},
volume={101},
date={1990},
number={3},
pages={593--629},
}

\bib{Marcus-Minc61}{article}{
author={Marcus, Marvin},
author={Minc, Henryk},
title={On the relation between the determinant and the permanent},
journal={Illinois J. Math.},
volume={5},
date={1961},
pages={376--381},
}

\bib{Minc63}{article}{
author={Minc, Henryk},
title={Upper bounds for permanents of $(0,\,1)$-matrices},
journal={Bull. Amer. Math. Soc.},
volume={69},
date={1963},
pages={789--791},
}

\bib{Minc78}{book}{
author={Minc, Henryk},
title={Permanents, with a foreword by Marvin Marcus},
series={Encyclopedia of Mathematics and its Applications},
volume={6},
publisher={Addison-Wesley Publishing Co., Reading, Mass.},
date={1978},
pages={xviii+205},
}

\bib{MO}{book}{
author={Moulin Ollagnier, Jean},
title={Ergodic theory and statistical mechanics},
series={Lecture Notes in Mathematics},
volume={1115},
publisher={Springer-Verlag, Berlin},
date={1985},
pages={vi+147},
}

\bib{R}{article}{
author={Rudolph, Daniel J.},
title={Restricted orbit equivalence},
journal={Mem. Amer. Math. Soc.},
volume={54},
date={1985},
number={323},
pages={v+150},
}

\bib{DSAO}{book}{
author={Schmidt, Klaus},
title={Dynamical systems of algebraic origin},
series={Progress in Mathematics},
volume={128},
publisher={Birkh\"{a}user Verlag, Basel},
date={1995},
pages={xviii+310},
}

\bib{SS16}{article}{
author={Schmidt, Klaus},
author={Strasser, Gabriel},
title={Permutations of $\mathbb{Z}^d$ with restricted movement},
journal={Studia Math.},
volume={235},
date={2016},
number={2},
pages={137--170},
}

\bib{Valiant79}{article}{
author={Valiant, Leslie G.},
title={The complexity of computing the permanent},
journal={Theoret. Comput. Sci.},
volume={8},
date={1979},
number={2},
pages={189--201},
}

\bib{VW01}{book}{
author={van Lint, Jacobus H.},
author={Wilson, Richard M.},
title={A course in combinatorics},
edition={2},
publisher={Cambridge University Press, Cambridge},
date={2001},
pages={xiv+602},
}

\bib{Weiss}{article}{
author={Weiss, Benjamin},
title={Monotileable amenable groups},
conference={
title={Topology, ergodic theory, real algebraic geometry},
},
book={
series={Amer. Math. Soc. Transl. Ser. 2},
volume={202},
publisher={Amer. Math. Soc., Providence, RI},
},
date={2001},
pages={257--262},
}

\bib{Zhang}{article}{
author={Zhang, Fuzhen},
title={An update on a few permanent conjectures},
journal={Spec. Matrices},
volume={4},
date={2016},
pages={305--316},
}

	\end{biblist}
	\end{bibdiv}
	\end{document}